\documentclass[11pt]{amsart}
\usepackage[utf8]{inputenc}
\usepackage{rotating}

\usepackage{geometry}
\usepackage{amsmath}
\usepackage{dsfont}
\usepackage{amsthm}
\usepackage{amssymb}
\usepackage{amsfonts}
\usepackage{mathtools}
\usepackage{tikz}
\usepackage{tikz-cd}
\usetikzlibrary{graphs,decorations.pathmorphing,decorations.markings}
\usepackage{hyperref}
\usepackage[capitalize,nameinlink]{cleveref}

\hypersetup{
	colorlinks=false,
	urlcolor=blue,
}
\usepackage{array}
\usepackage{import}

\usepackage{bm}
\usepackage{upgreek}
\makeatletter
\newcommand{\bfgreek}[1]{\bm{\@nameuse{up#1}}}

\newtheorem{theorem}{Theorem}[section]

\newtheorem{corollary}{Corollary}[section]
\newtheorem{lemma}{Lemma}[section]

\newtheorem{ass}{Assumption}[section]

\newcommand{\A}{\mathbb A}
\newcommand{\R}{\mathbb R}
\newcommand{\Q}{\mathbb Q}
\newcommand{\Z}{\mathbb Z}
\newcommand{\C}{\mathbb C}
\newcommand{\D}{\mathbb D}

\newcommand{\cO}{\mathcal O}

\newcommand{\fl}{\mathfrak l}
\newcommand{\fp}{\mathfrak p}

\newcommand{\GL}{\mathrm{GL}}
\newcommand{\aInd}{{}^{\rm a}{\rm Ind}}
\newcommand{\sfv}{{\sf v}}

\newcommand{\mirp}[1]{K_p^{(4)}(#1)}
\newcommand{\mirf}[1]{K_f^{(4)}(#1)}
\newcommand{\congp}[1]{K_p^{(2)}(#1)}
\newcommand{\congf}[1]{K_f^{(2)}(#1)}

\title[Eisenstein cohomology and congruences]{Eisenstein cohomology and congruences for the ratios of Rankin--Selberg $L$-functions}
\author{\bf P. Narayanan \ \ \ \& \ \ \ A. Raghuram}
\date{\today}

\address{Dept.\,of Mathematics, Indian Institute of Science Education and Research, Dr.\,Homi Bhabha Road, Pune 411008, INDIA.}

\email{narayanan.p@students.iiserpune.ac.in}

\address{Dept.\,of Mathematics, Fordham University at Lincoln Center, New York, NY 10023, USA.} 

\email{araghuram@fordham.edu}

\begin{document}

\begin{abstract} 
A well-known principle states that a congruence between objects should give rise to a corresponding congruence between the 
special values of $L$-functions attached to these 
objects. In this article, using the machinery of Eisenstein cohomology after refining it for integral cohomology, we prove an instance of 
this principle for the ratios of critical values for Rankin--Selberg $L$-functions attached to pairs of holomorphic cuspforms. 
\end{abstract}

\maketitle

\tableofcontents

\section{Introduction}
A principle, with origins in Iwasawa theory, says that a congruence between objects should give rise to a congruence between 
the special values of $L$-functions attached to these objects. 
See Vatsal \cite{vatsal} for an instance of this principle, and for a brief discussion of its historical origins.
In accordance with such a principle, we prove that a congruence 
between modular forms gives rise to a congruence between the ratio of special values of Rankin--Selberg $L$-functions of these forms with an auxiliary modular form. 
 The proof uses Eisenstein cohomology as in \cite{harder-raghuram} but after refining that framework to deal with integral cohomology.

\smallskip

Suppose $h(z) = \sum a(n,h) e^{2\pi i nz}$ 
and $h'(z) = \sum a(n,h') e^{2\pi i nz}$ are primitive holomorphic modular cuspforms of levels $N$ and $N'$,  weights 
$k$ and $k'$, with nebentypus characters $\chi$ and $\chi',$ respectively. Denote this as 
$h \in S_k(N,\chi)^{\textup{new}}$ and $h' \in S_{k'}(N',\chi')^{\textup{new}}$.
Let $\Q(h,h')$ be the number field obtained by adjoining the Fourier coefficients $\{a(n,h)\}$ and $\{a(n,h')\}$ to $\Q.$ 
Assume that $k' > k.$ A well-known theorem of Shimura \cite{shimura-periods} says that for 
$D_N(s, h, h'),$ the degree-$4$ Rankin--Selberg $L$-function attached to the pair $(h, h')$, and for any integer $m$ with 
$k \leq m < k',$ we have: 
$D_N(m,h,h') \ \approx \ (2\pi i)^{l+1-2m} \, \mathfrak{g}(\chi) \, u^+(f')u^-(f'),$
where $\approx$ means equality up to an element of $\Q(h, h')$, $u^\pm(h')$ are the two periods attached to $h'$ by Shimura, and 
$\mathfrak{g}(\chi)$ is the Gauss sum of $\chi.$ 
The integers $k \leq m < k'$ are all the critical points for $D_N(s, h ,h').$ 
Suppose $k' \geq k+2$, and we look at two successive critical values, then the only change in the right hand side is 
$(2\pi i)^{-2}$ which may be seen to be exactly accounted for 
by the $\Gamma$-factors at infinity. Suppose $L(s, h \times h')$ denotes the completed degree-$4$ $L$-function attached to $(h,h')$, 
then we deduce, for any $m$ with $k \leq m < m+1 < k',$ that  
$$
\frac{L(m, h \times h')}{L(m+1, h \times h')} \in \Q(h, h'). 
$$
Now suppose $h'' \in S_{k'}(N',\chi')^{\textup{new}}.$ We say $h''$ is congruent to $h'$ 
if in some number field $E$, considered as a subfield of $\C$, large enough to contain all the rationality fields, 
and some prime ideal $\fl$ of the ring of integers $\cO_E$ of $E$, and for some positive integer $n$ 
the congruence $a(p,h') \equiv a(p,h'') \pmod{\fl^n}$ holds for all rational primes $p$. If $n > 1$ such a congruence is often called a super-congruence. 
Both the $L$-functions $L(s, h \times h')$ and $L(s, h \times h'')$ have the same set of critical points. 
Under some general hypotheses we prove:
\begin{equation}
	h' \equiv h'' \pmod{\mathfrak{l}^n} \ \implies \ 
	\frac{L(m,h \times h')}{L(m+1, h\times h')} \equiv \frac{L(m, h \times h'')}{L(m+1, h\times h'')} \pmod{\mathfrak{l}^n}. 
\end{equation}
This is proven in Thm.\,\ref{thm: second-main-theorem-on-l-values} 
when $\fl$ is away from a finite set of primes (about which we will say more presently),  
the mod-$\fl$ Galois representations attached to $h'$ and $h''$ are irreducible, 
and the levels $N$ and $N'$ are square-free and relatively prime. 
If we drop the hypothesis on the levels then in Thm.\,\ref{thm:reg-sem-poly} a weaker congruence is proven. 
A natural variation by taking the congruent modular forms to be of lower weight, i.e., if $k' < k$, is addressed in Thm.\,\ref{thm: varying-the-lower-weight}.

\smallskip

In a companion paper \cite{narayanan-raghuram}, we computationally verified such a congruence in several concrete examples. The methods we use 
in this paper proving the theorems are completely independent of the companion paper. 
The reader is also referred to Vatsal \cite{vatsal} in which is proved that a congruence between modular forms gives rise to a 
congruence between special values of their degree-$2$ $L$-functions divided by certain periods provided the periods are canonically chosen. 
Vatsal's theorem was a major motivation for our own work. In our situation, one of the advantages is that for ratios of $L$-values 
there is no need to appeal to any periods, and their canonical normalization is a moot point. 
The ratio of successive $L$-values appears naturally in Langlands's famous constant term theorem. 
In the work of the second author with Harder \cite{harder-raghuram} 
this theorem of Langlands was interpreted in the general framework of Eisenstein cohomology. 
We prove the main results of this paper (Thm.\,\ref{thm:reg-sem-poly}, Thm.\,\ref{thm: second-main-theorem-on-l-values}) 
using the tools developed in \cite{harder-raghuram} but after refining that framework to deal with integral cohomology.

\smallskip 

In results concerning congruences and special values of $L$-functions, the congruence prime $\mathfrak{l}$ is usually assumed to be not among a finite set of primes; 
see, for example, Hida \cite{h81a}, Vatsal \cite{vatsal}, or Balasubramanyam and Raghuram \cite{balasubramanyam-raghuram}.  
For the main results of our paper, for similar reasons, we need to avoid certain finite sets of primes: (i) 
small primes with respect to the weight $\mathsf{S}_{\textup{weight}}= \{ \mathfrak{p} \subset \mathcal{O}_E \ | \ p\leq k, \ p \leq k'\};$ 
(ii) primes diving the levels $\mathsf{S}_{\textup{level}} = \{ \mathfrak{p} \subset \mathcal{O}_E \:\: |\:\: p| NN'\};$
(iii) primes supporting torsion in integral Eisenstein cohomology $\mathsf{S}_{\textup{Eis}}$ (see \ref{section: another-integral-structure} for the precise definition); 
and (iv) primes coming from archimedean considerations $\mathsf{S}_{c_{\infty}}.$ 

\smallskip

For the introduction, let us now adumbrate the proofs. 
Under the hypotheses of the theorems, the congruent modular forms $h'$ and $h''$ give rise to 
two cohomology classes that are congruent modulo $\mathfrak{l}^n$; this is possible since 
the Hecke algebra is Gorenstein; see Thm.\,\ref{thm:congruent-coh-classes}. These cohomology classes are in the inner cohomology of a locally symmetric space for 
$\GL(2)$ with coefficients in an integral local system attached to a lattice in a highest weight module determined by the weight $k'$. 
Tensoring with a similarly defined cohomology class attached to $h$, gives rise to cohomology classes on $\GL(2) \times \GL(2)$
which are congruent modulo $\mathfrak{l}^n$. We then go through the formalism of rank-one 
Eisenstein cohomology of \cite{harder-raghuram} on $\GL(4)$ for the $(2,2)$-parabolic subgroup. 
The configuration of maps in (6.18) of {\it loc.\,cit.}\ is then reworked at an integral level for this specific context. 
This configuration of maps is at the heart of affording a cohomological interpretation of Langlands's theorem mentioned above: 
the constant term of an Eisenstein series is essentially the standard intertwining operator between two induced representations. 
While working integrally, one needs a delicate control on integral structures on cohomology groups and such induced representations 
which necessitates avoiding a finite set $\mathsf{S}_{\textup{Eis}}$ of primes. One also needs delicate control on the local standard 
intertwining operator; at the archimedean place it gives rise to a rational number forcing us to avoid the finite set $\mathsf{S}_{c_{\infty}}$ of primes 
supporting that number;  
and at finite ramified primes the local computation is explicitly carried out, but only in the presence of mild ramification 
explaining the condition on the levels $N$ and $N'$ in Thm.\,\ref{thm: second-main-theorem-on-l-values}.

\medskip

{\Small \noindent {\it Acknowledgements:} 
The authors are grateful to the Institute for Advanced Study, Princeton, for a summer collaborator's grant in 2023 when this project got started. 
The authors thank Baskar Balasubramanyam, Haruzo Hida, Jacques Tilouine, and Eric Urban for some invaluable comments and feedback 
during the course of the work. The first author is supported by the CSIR fellowship for his Ph.D.}

\medskip
\section{Preliminaries}
 
\subsection{Basic notations}

Let $G_n$ be the algebraic group ${\rm GL}_n/\mathbb{Q}$ with the chain of subgroups
$G_n \supset B_n = T_nU_n \supset T_n \supset Z_n,$
where $B_n$ is the Borel subgroup of all upper triangular matrices in $G_n$, 
$T_n$ the torus consisting of all the diagonal matrices, $U_n$ the unipotent radical of $B_n$, and 
$Z_n$ the center of $G_n.$ Let $X^*(T_n)$ denote the group of characters of $T_n$; it is free abelian on the basis 
$\mathbf{e}_1, \mathbf{e}_2, \dots \mathbf{e}_n,$ where $\mathbf{e}_i: \textup{diag}(t_1, t_2, \dots, t_n) \mapsto t_i$ for $i=1,2,\dots, n.$ 
Then $\bfgreek{delta}_n = \mathbf{e}_1 + \mathbf{e}_2 + \dots + \mathbf{e}_n$ is the determinant character. 
Let $\rho_n$ denote half the sum of positive roots. 
Let $\Delta_n$ denote the set of roots and $\Delta^+_n$ is the subset of positive roots with respect to $B_n.$ 
Let $\Pi_n$ denote the set of simple roots $\mathbf{e}_i - \mathbf{e}_{i+1}$ for $i = 1, \dots, n-1.$
Let $W_{n}$ denote the Weyl group of $G_n$ which will be identified with the set of permutation matrices. 

\smallskip

We will be using rank-one Eisenstein cohomology for the ambient group $G_4 = {\rm GL}_4/\mathbb{Q}$; in that context, 
$P$ will denote the standard parabolic subgroup of $G_4$ of $(2,2)$ block upper-triangular matrices, 
corresponding to the deletion of the simple root
 $\mathbf{e}_2 - \mathbf{e}_3$ of $\Pi_4.$  Let $U_P$ denote the unipotent radical of $P$ and $\kappa: P \rightarrow P/U_P \cong M_P$ be the projection onto the Levi quotient  
 $M_P \cong G_2 \times G_2.$ 
 (The notation $G_2$ for ${\rm GL}_2/\Q$ will cause no confusion as we do not need any exceptional group in this article.)
 The simple roots of $M_P$ are $\Pi_{M_P} = \{ \mathbf{e}_1 - \mathbf{e}_2, \mathbf{e}_3- \mathbf{e}_4\}.$ 
 The Weyl group of $M_P$ is denoted $W^{M_P}$ which is isomorphic to $\mathbb{Z}/2\mathbb{Z} \times \mathbb{Z}/2\mathbb{Z}$ 
 and realized as a subgroup of the Weyl group $W_4 \simeq S_4$ of $G_4$. 
The set of Kostant representatives for $P$ is 
$W^P = \{w \in W_4\,\, | \,\,  w^{-1}\alpha > 0, \forall \alpha \in \Pi_{M_P}\},$ giving a complete set of representatives  
for the right cosets $W_M\backslash W_{4};$ there are six elements in $W^P.$

\smallskip

A number field $E$ is a finite extension of $\Q$; often $E$ is assumed to be Galois over $\Q$ and containing all the Fourier coefficients of various modular forms at hand. 
Let $\fl$ be a prime ideal in the ring of integers $\cO_E$ of $E$; assume $l \geq 5$, where $l$ is the rational prime lying below $\mathfrak{l}.$
Also, fix an embedding $\iota: \hat{\bar{E}}_{\mathfrak{l}} \cong \mathbb{C}$. And for any $\mathcal{O}_E$-module, say $M$, the notation 
$M\otimes \C$ will always mean $M\otimes_{\iota}\C.$ The embedding $\iota$ will often be dropped from notation.

\medskip
\subsection{Sheaves and cohomology}

\subsubsection{Locally symmetric spaces}
Let $\A$ (resp., $\A_f$) be the ring of adeles (resp., finite adeles) of $\Q$. 
Let $K_{n,\infty} = {\rm SO}(n) \times Z_n(\mathbb{R})^0 \subset G_n(\R),$ where ${\rm SO}(n)$ is the usual compact special orthogonal group, and 
$(\cdot)^0$ denotes the connected component of the identity. 
Let $K_f \subset G_n(\mathbb{A}_f)$ be an open compact subgroup. The
adelic locally symmetric space is the double-coset space:
$
	S^{(n)}_{K_f}:= G_n(\mathbb{Q}) \backslash G_n(\mathbb{A}) / K_{n,\infty} \cdot K_f.
$

\subsubsection{Highest weight representations}
 Suppose $\mu = b_1 \mathbf{e}_1 + \cdots + b_n\mathbf{e}_n$ is a dominant integral weight, i.e., 
$b_1,\dots, b_n \in \Z$, and $b_1 \geq \cdots \geq b_n$. 
For such a weight $\mu$, let $\mathcal{M}_{\mu,\mathbb{Q}}$ denote the finite-dimensional absolutely irreducible representation of $G_n = \GL_n/\mathbb{Q}$ with highest weight 
$\mu$. 
For $A = E, E_\mathfrak{l}, \mathbb{C},$ define $\mathcal{M}_{\mu, A} : = \mathcal{M}_{\mu} \otimes_\Q A.$
For an integer $m \in \Z$ put $\mu(m) = \mu + m \bfgreek{delta}_n$; then $\mathcal{M}_{\mu(m)} = \mathcal{M}_\mu \otimes {\rm det}^m.$

\subsubsection{Sheaves and their cohomology}
Let $\pi: G_n(\mathbb{A})/ K_{n,\infty}\times K_f \rightarrow S^{(n)}_{K_f}$ be the projection. 
For $A= E, E_\mathfrak{l}, \mathbb{C},$ define a sheaf $\widetilde{\mathcal{M}}_{\mu,A}$
whose sections over an open $U \subset S^{(n)}_{K_f}$ are given by:
\begin{equation*}
	\widetilde{\mathcal{M}}_{\mu,A}(U) = \{ s: \pi^{-1}(U)\rightarrow\mathcal{M}_{\mu,A}\,\,|\,\, s(\gamma. g) = \gamma s(g)\,\, \forall \gamma \in G_2(\mathbb{Q})\}.
\end{equation*}
Let $H^\bullet (S^{(n)}_{K_f}, \widetilde{\mathcal{M}}_{\mu,A})$ denote the sheaf cohomology groups. 
For $A= E, E_\mathfrak{l}, \mathbb{C},$ one has isomorphisms for changing the base: 
$H^\bullet (S^{(2)}_{K_f}, \widetilde{\mathcal{M}}_{\mu,A}) \cong H^\bullet (S^{(2)}_{K_f}, \widetilde{\mathcal{M}}_{\mu})\otimes_\Q A.$
If $K_{1,f} \subset K_{2,f}$ then there is a natural map 
$H^\bullet (S^{(n)}_{K_{2,f}}, \widetilde{\mathcal{M}}_{\mu,A}) \rightarrow H^\bullet (S^{(n)}_{K_{1,f}}, \widetilde{\mathcal{M}}_{\mu,A}),$ letting us  
define
\[ H^\bullet (S^{(n)}, \widetilde{\mathcal{M}}_{\mu,A}) = \textup{colim}_{\substack{K_f}} H^\bullet (S^{(n)}_{K_f}, \widetilde{\mathcal{M}}_{\mu,A}).\]

\subsubsection{The mirahoric congruence subgroups}
For a prime $p$ and an integer $n_p \geq 0$, define  
$$
K_p^{(2)}(n_p) \ := \ 
\{ g \in \GL_2(\Z_p) \ | \ g \equiv \begin{pmatrix} * & * \\ 0 & 1 \end{pmatrix} \pmod{p^{n_p}} \}, \quad {\rm and}
$$
$$
K_p^{(4)}({n_p}) \ := \ 
\left\{ g \in \GL_4(\Z_p) \ | \ g \equiv \begin{pmatrix}
	* & * & * & * \\ * & * & * & * \\ * & * & * & * \\ 0 & 0 & 0 & 1 \end{pmatrix} \pmod{p^{n_p}} \right\}.
$$
Let $N = \prod_p p^{v_p(N)} = \prod_p p^{n_p}$ be a positive integer. Define 
$K_f^{(2)}(N)$ to be the subgroup of $\GL_2(\hat{\mathbb{Z}})$ defined as $K^{(2)}(N) = \prod_{p<\infty } K_p(n_p)$, and  
$K_f^{(4)}(N)$ be the subgroup of $\GL_4(\hat{\mathbb{Z}})$ by $K_f^{(4)}(N) = \prod_{p<\infty } K_p^{(4)}(n_p)$.

\subsubsection{Inner cohomology of $\GL_2/\mathbb{Q}$}
\label{sec:inner-coh-gl2} Up until Sect.\,\ref{section: modifications-when-the-level-is-small} assume $N\geq 3$.
For the  level structure $\congf{N}$ in $G_2$, abbreviate $S^{(2)}_1(N) = S^{(2)}_{K_1(N)} $. 
For $A= E, E_\mathfrak{l}, \mathbb{C},$ let $H_!^{1}(S_1^{(2)}(N), \widetilde{\mathcal{M}}_{\mu,E})$ denote the inner cohomology group, 
by which one means the image of cohomology with compact supports inside full cohomology.  
Let $\mathsf{S}$ denote any set of finite places containing all primes dividing $N$ and the infinite place, then the action of the commutative 
Hecke algebra 
$\mathcal{H}_2^{\mathsf{S}} = \bigotimes'_{p \nmid N} C^\infty_c(G_2(\mathbb{Q}_p)/\!\!/G_2(\mathbb{Z}_p))$ on 
$H_!^1(S_1^{(2)}(N), \widetilde{\mathcal{M}}_{\mu,E})$ is semi-simple. 
The inner-spectrum is denoted $\textup{Coh}_!(G_2, \mu, \congf{N})$ which consists of the 
set of all isomorphism classes of eigencharacters of $\mathcal{H}_2^{\mathsf{S}}$ which appear in 
$H_!^{1}(S_1^{(2)}(N), \widetilde{\mathcal{M}}_{\mu,E}).$

\subsubsection{Representation at infinity}
\label{sec:rep-infty}
We will assume the weight $\mu$ is regular. Given an absolutely simple Hecke module $\sigma_f \in \textup{Coh}_!(G_2, \mu, \congf{N})$ and an embedding 
$\iota : E \to \C$, the module ${}^\iota\sigma_f$ is the $\congf{N}$-invariants of the finite part of a cuspidal automorphic representation which--up to a minor 
abuse of notation--will be denoted by ${}^\iota\sigma$. The $\iota$ is fixed and will be dropped from notation. 
The representation at infinity ${}\sigma_\infty$ is an essentially discrete 
series representation $\mathbb{D}_{\mu}$ of $\GL_2(\R)$ such that 
such that the relative Lie algebra cohomology 
$H^{1}(\mathfrak{g}_2, K_{2, \infty}, \mathbb{D}_{\mu} \otimes \mathcal{M}_{\mu, \mathbb{C}})$ is nonzero. The notations are as in \cite[Sect.\,3.1]{harder-raghuram}.

\medskip
\subsection{Integral structures on cohomology groups for $\GL_2$}

\subsubsection{Highest weights for modular forms}
For an integer $k \geq 2$, define: 
$$
\mu_k = (k-2)(\mathbf{e}_1 - \mathbf{e}_2)/2+ (k/2 -1) \bfgreek{delta}_2 \ = \ (k-2){\bf e}_1 + 0 {\bf e}_2.
$$ 
The underlying $\Q$-vector space of $\mathcal{M}_{\mu_k, \mathbb{Q}}$
consists of homogenous polynomials of degree $k-2$ in two variables $X$ and $Y$ with coefficients in $\Q$. Also, if $k > 2$ then $\mu_k$ is a regular weight. 
Similarly, an integer $k' \geq 2$ determines a $\mu' = \mu_{k'}.$ If the $k$ is clear from context then write 
$\mu := \mu_k$; similarly $\mu' = \mu_{k'}$. Hereafter the weights $\mu$ and $\mu'$ will be assumed to be regular.

\subsubsection{Integral sheaves}
Assume $N\geq 3.$ One can also re-define the sheaf $\widetilde{\mathcal{M}}_{\mu,E}$ with respect to the projection:
\[
\pi_1: G_2(\mathbb{Q})\backslash \Bigl( G_2(\mathbb{R})/K_{2, \infty} \times G(\mathbb{A}_f)\Bigr) \rightarrow S^{(2)}_1(N).
\] 
In this case, sections over an open set $U$ are given by:
\begin{multline*}
	\widetilde{\mathcal{M}}_{\mu,E} = 
	\{ \tilde{s}: \pi_1^{-1}(U) \rightarrow {\mathcal{M}_{\mu, E}\otimes \mathbb{A}_E^{(\infty)}}| \;\; \text{ $\tilde{s}$ is locally constant, } 
	\\ g_f\cdot \tilde{s}(x_\infty, g_f ) \in \mathcal{M}_{\mu,E}\; \text{ and }\;\tilde{s}(x_\infty, g_f k_f) =  k_f^{-1}\cdot s(x_\infty, g_f),
	\;\; \forall\; k_f \in \congf{N}\}.
\end{multline*}

Take $\mathcal{M}_{\mu,\Z}$ to be the $\Z$--lattice generated by $\{X^j Y^{k-2-j} : 0 \leq j \leq k-2\}.$  
It is clear that for $A = \mathcal{O}_E, E, \mathcal{O}_\mathfrak{l}, E_\mathfrak{l}, \mathbb{C},$ one has 
$\mathcal{M}_{\mu,\mathcal{O}_E}\otimes A = \mathcal{M}_{\mu,A}$. 
It is also clear that $\mathcal{M}_{\mu,\mathcal{O}_E} \otimes \hat{\mathcal{O}}_E$ is stable under the action of $K_1(N).$ 
For $A = \mathcal{O}_E, E, \mathcal{O}_\mathfrak{l}, E_\mathfrak{l}, \mathbb{C},$ define:
	\begin{multline*}
	\widetilde{\mathcal{M}}_{\mu,A} = \{ \tilde{s}: \pi_1^{-1}(U) \rightarrow \mathcal{M}_{\mu, \mathcal{O}_E}\otimes \hat{\mathcal{O}}_E| \;\; \text{ $\tilde{s}$ 
	is locally constant, } \\ g_f\cdot \tilde{s}(x_\infty, g_f ) \in M_{\mu,\mathcal{O}_E}\otimes A\; \text{ and }
	\;\tilde{s}(x_\infty, g_f k_f) =  k_f^{-1}\cdot s(x_\infty, g_f), \;\; \forall\; k_f \in \congf{N}\}.
\end{multline*}

\subsubsection{Classical cohomology groups}
 Let $\mathbb{H}$ denote the complex upper half space which is acted upon by ${\rm SL_2}(\mathbb{R})$ in the usual way. 
The group $\Gamma_1(N) = {\rm GL}_2^+(\mathbb{Q}) \cap \congf{N}$ is the congruence subgroup of 
${\rm SL}_2(\Z)$ of matrices which are congruent to $\left(\begin{smallmatrix} 1 & * \\ 0 & 1 \end{smallmatrix}\right)$ modulo $N.$
Put $X_1(N) = \Gamma_1(N) \backslash \mathbb{H}.$ Let $\pi_{\mathbb{H}}: \mathbb{H} \rightarrow X_1(N)$ be the canonical projection. 
For $A = \mathcal{O}_E, E, \mathcal{O}_\mathfrak{l}, E_\mathfrak{l}, \mathbb{C},$ 
define a sheaf $\underline{\mathcal{M}}_{\mu, A}$ on $X_1(N)$ whose sections over an open set $U \subset X_1({N})$ are 
$$
\{ s:\pi_{\mathbb{H}}^{-1}(U) \rightarrow \mathcal{M}_{\mu,A} \ | \ s \text{ is locally constant, and} \ 
s(\gamma( z)) = \gamma \cdot s(z), \ \forall \gamma \in \Gamma_1({N}), \ z\in \mathbb{H}\}.
$$

Given $g_f \in G_2(\mathbb{A}_f),$ one can express $g_f = \gamma k_f$ for some $\gamma \in G_2(\mathbb{Q})$ and $k_f \in \congf{N}.$ 
The map $G(\mathbb{Q})(g_\infty, g_f)K_{2, \infty}\congf{N} \mapsto \Gamma_1(N) \gamma^{-1}g_\infty\cdot \sqrt{-1}$ 
is a homeomorphism between $S_1^{(2)}(N) \xrightarrow{\sim} X_1(N),$ giving then an isomorphism of sheaves 
$\widetilde{\mathcal{M}}_{\mu,A} \xrightarrow{\sim} \underline{\mathcal{M}}_{\mu,A}$, from which one has:
$$
H^1_!(S^{(2)}_1(N), \widetilde{\mathcal{M}}_{\mu,A}) \ \cong \ H^1_!(X_1(N), \underline{\mathcal{M}}_{\mu,A}), \quad 
A = \mathcal{O}_E, E, \mathcal{O}_\mathfrak{l}, E_\mathfrak{l}, \mathbb{C}.
$$ 

Furthermore, there is also a canonical isomorphism 
$$
H_!^1(X_1(N), \underline{\mathcal{M}}_{\mu,A}) \cong H_!^1(\Gamma_1(N), \mathcal{M}_{\mu,A}),
$$
where the latter is the parabolic cohomology group defined by Shimura \cite[Chap.\,8]{iataf}; see also Hida \cite[Appendix]{lfe}. 
Since $\mathcal{O}_\mathfrak{l}, E, E_\mathfrak{l}, \mathbb{C}$ are all flat $\mathcal{O}_E$--modules one gets
\begin{equation*}
	H_!^1(\Gamma_1(N), \mathcal{M}_{\mu,\Z})\otimes A \cong H_!^1(\Gamma_1(N), \mathcal{M}_{\mu,A})\quad \text{ for } A = \mathcal{O}_E, \mathcal{O}_\mathfrak{l}, E, E_\mathfrak{l}, \mathbb{C}.
\end{equation*}
See Hida \cite[p.\,168]{lfe}. This in turn implies
\begin{equation}
H_!^1(S_1^{(2)}(N), \widetilde{\mathcal{M}}_{\mu,\mathcal{O}_E}) \otimes A \cong	
H_!^1(S_1^{(2)}(N), \widetilde{\mathcal{M}}_{\mu,A}) \quad \text{ for } A = 
\mathcal{O}_\mathfrak{l}, E, E_\mathfrak{l}, \mathbb{C}. 
\label{eqn: functorial-properties-of-cohomology-groups}
\end{equation}

\subsubsection{$\pm 1$ eigenspaces in cohomology}
For $A = E, E_\mathfrak{l}, \mathbb{C},$ the group ${\rm O}(2)/ {\rm SO}(2) \cong \Z/2\Z$ 
acts on $H^1_!(S^{(2)}_1(N),\widetilde{M}_{\mu,A}).$ For a character $\epsilon$ of ${\rm O}(2)/ {\rm SO}(2)$, the 
$\epsilon$-eigenspace  
will be denoted $H^1_!(S^{(2)}_1(N),\widetilde{M}_{\mu,A})(\epsilon)$. 
When $A{^\circ} = \mathcal{O}_E, \mathcal{O}_\mathfrak{l},$ the notation $H^1_!(S^{(2)}_1(N),\widetilde{M}_{\mu,A{^\circ}})(\epsilon)$ 
means the image of $H^1_!(S^{(2)}_1(N),\widetilde{M}_{\mu,A{^\circ}})$ inside 
$H^1_!(S^{(2)}_1(N),\widetilde{M}_{\mu,A})(\epsilon)$ for $A =E, E_\mathfrak{l}$, respectively.

\subsubsection{Avoiding torsion in integral cohomology}
\label{sec:avoiding-torsion}
For an integer $N \geq 1$, define a finite set of prime ideals: 
\begin{equation}
\mathsf{S}_N := \{ \mathfrak{p} \ |\  \mbox{$\fp$ is a prime ideal of $\mathcal{O}_E$ which divides $6N$} \}.
\end{equation}
Similarly, for an integer $k \geq 2$, define: 
\begin{equation}
\mathsf{S}_{k} \:= \ \{ \mathfrak{p} \ |\  \mbox{$\fp$ is a prime ideal of $\mathcal{O}_E$ over any prime $p \leq k$} \}.
\end{equation}
 If $\fl \not\in \mathsf{S}_N \cup \mathsf{S}_{k}$ then by 
Hida \cite[($1.14_b$)]{h81a}, the group 
$H^1_!(S_1^{(2)}(N), \widetilde{\mathcal{M}}_{\mu,\mathcal{O}_E})$ has no $\fl$-torsion.

\subsubsection{Tate twists}  For $A = \mathcal{O}_E, E, \mathcal{O}_\mathfrak{l}, E_\mathfrak{l}, \mathbb{C}$ it is clear that 
$\underline{\mathcal{M}}_{\mu,A} \cong \underline{\mathcal{M}}_{\mu(m),A}$ since the sheaves are defined by the action of 
$\gamma \in {\rm SL}_2(\mathbb{R}).$ We fix \textit{one} integral structure, namely the image of 
\begin{equation}
	H^1_!(\Gamma_1(N), {\mathcal{M}}_{\mu,\mathcal{O}_E}) \,\, (\text{resp}., \,\, H^1_!(\Gamma_1(N), {\mathcal{M}}_{\mu,\mathcal{O}_\mathfrak{l}}) )
\end{equation} 
in all of the cohomology groups 
$H^1_!(S_1^{(2)}(N), \widetilde{\mathcal{M}}_{\mu(m),E})$ 
(resp., $H^1_!(S_1^{(2)}(N), \widetilde{\mathcal{M}}_{\mu(m),E_\mathfrak{l}})$) with $m \in \mathbb{Z}.$ The notations
$\tilde{H}^1_!(S^{(2)}_1(N), \widetilde{\mathcal{M}}_{\mu,\mathcal{O}_E})$ and $\tilde{H}^1_!(S^{(2)}_1(N), \widetilde{M}_{\mu,\mathcal{O}_\mathfrak{l}})$ 
will be used to denote the images respectively. This is done to ensure there are no torsion cohomology classes. If we avoid a suitable finite set of primes as in 
\ref{sec:avoiding-torsion} then there is no $\fl$-torsion and we may simplify the notation $\tilde{H}^1_!(...)$ to  $H^1_!(...)$. 
It should be kept in mind that 
the twists appear when one considers the action of an integral Hecke algebra on integral cohomology. 
In general, we shall reserve the notation $\tilde{H}^\bullet(...)$ to denote the image of the cohomology with integral coefficients inside the 
cohomology with rational coefficients.

\subsubsection{Modifications when the level is small}
 \label{section: modifications-when-the-level-is-small} 
 The group $\Gamma_1(N)$ acts freely on $\mathbb{H}$ only when $N\geq 3.$ 
 When $N=1 \text{ or }2$ we follow Hida \cite[Sect.\,5.3]{mfg} to define the integral cohomology groups. 
 Define $K^{(2)}_{f,0}(3) := \{\gamma \in \GL_2(\hat{\mathbb{Z}})\,|\, \gamma \equiv \mathbf{1}_2 \pmod{3 \hat{\mathbb{Z}}}\}$ and set 
 $K_{f,0}^{(2)}(N) = \congf{N} \cap K_{f,0}^{(2)}(3).$
Note that $K_{f,0}^{(2)}(N) \cap \GL^+_2(\mathbb{Q}) = \Gamma_1(N) \cap \Gamma(3)$. 
Put $\mathcal{M}_{\mu,\Z[1/6]} = \mathcal{M}_{\mu,\mathbb{Z}} \otimes \mathbb{Z}[1/6].$ Then for $N =1,2,$ we have:
\begin{multline*}
	H^1_!(S^{(2)}_{K_{f,0}(N)} , \widetilde{\mathcal{M}}_{\mu, \Z[1/6]}) \cong 
	H^1_!(\Gamma_1(N) \cap \Gamma(3) \backslash \mathbb{H}, \widetilde{\mathcal{M}}_{\mu, \Z[1/6]}) \\ \cong 
	H^1_!(\Gamma_1(N), \mathcal{M}_{\mu, \Z[1/6]})^{\Gamma_1(N)/ \Gamma_1(N) \cap \Gamma(3)} \cong 
	H^1_!(\Gamma_1(N), \mathcal{M}_{\mu, \Z[1/6]}).
\end{multline*}
The last isomorphism is because $6$ is invertible in the ring $\Z[1/6]$ and the index of 
$\Gamma_1(N) \cap \Gamma(3)$ in $\Gamma_1(N)$ divides $24$. Since $l \geq 5$ we have
$H_!^1(\Gamma_1(N), \mathcal{M}_{\mu,\Z[1/6]})\otimes A \cong H_!^1(\Gamma_1(N), \mathcal{M}_{\mu,A})$ for 
$A =  \mathcal{O}_\mathfrak{l}, E, E_\mathfrak{l}, \mathbb{C}.$ 
So for $N = 1, 2,$ we fix $S_1^{(2)}(N) = S^{(2)}_{K_{f,0}(N)}$, $X_1(N) = \Gamma_1(N) \cap \Gamma(3) \backslash \mathbb{H},$ and 
$\mathcal{M}_{\mu, \Z[1/6]}$ to be the lattice in the coefficient system. 

\medskip
\subsection{Cohomology of $M_P = G_2 \times G_2$}

\subsubsection{K\"unneth isomorphisms}
\label{sec:kunneth-levi}
The weights $\mu$ and $\mu'$ give a highest weight $\mu + \mu'$ for $M_P$. For $A= E, E_\mathfrak{l}, \mathbb{C},$ let $H^2(S^{M_P}_{N \times N'}, \widetilde{\mathcal{M}}_{\mu +\mu', A})$ 
 denote the cohomology groups at degree $2$ of the locally 
symmetric space associated to the Levi $M_P$ with level structure $\congf{N} \times \congf{N'},$
and coefficient system $\widetilde{\mathcal{M}}_{\mu +\mu', A}$. 
If the weights $\mu$ and $\mu'$ are regular, then the notion of inner and strongly inner cohomology in \cite{harder-raghuram} coincide. 
Moreover, we have a K\"unneth isomorphism: 
$$
H^2_{!}(S^{M_P}_{N \times N'}, \widetilde{\mathcal{M}}_{\mu +\mu', A}) \cong 
H^1_!(S_1^{(2)}(N), \widetilde{\mathcal{M}}_{\mu, A}) \otimes_A H^1_!(S_1^{(2)}(N'), \widetilde{\mathcal{M}}_{\mu', A}).
$$

\subsubsection{Integral structures}
\label{sec:integral-kunneth-levi}
Via the K\"unneth isomorphism, the image of 
 $$
 \tilde{H}^1_!(S_1^{(2)}(N'), \widetilde{\mathcal{M}}_{\mu, \mathcal{O}_E}) 
 \otimes_{\mathcal{O}_E} 
 \tilde{H}^1_!(S_1^{(2)}(N'), \widetilde{\mathcal{M}}_{\mu', \mathcal{O}_E}) \ \hookrightarrow \ 
 H^2_{!}(S^{M_P}_{N \times N'}, \widetilde{\mathcal{M}}_{\mu +\mu', E})
 $$ 
gives an $\mathcal{O}_E$--lattice 
which will be denoted $\tilde{H}^2_{!}(S^{M_P}_{N' \times N}, \widetilde{\mathcal{M}}_{\mu' +\mu, \mathcal{O}_E})$.
Similarly, an $\cO_\fl$-lattice: 
$$
\tilde{H}^2_{!}(S^{M_P}_{N \times N'}, \widetilde{\mathcal{M}}_{\mu +\mu', \mathcal{O}_{E_\fl}}) \ \subset \ 
\tilde{H}^2_{!}(S^{M_P}_{N \times N'}, \widetilde{\mathcal{M}}_{\mu +\mu', E_\fl}).
$$

\smallskip

Suppose for the moment, $R$ and $S$ are commutative rings with $1$ and $R \rightarrow S$ is a ring homomorphism, and if $M$ and $N$ are  $R$-modules then
$(M\otimes_R N) \otimes_R S \cong (M\otimes_R S) \otimes_S (N\otimes_R S).$ 
Applying this for $\cO_E \hookrightarrow A$, for $A =\mathcal{O}_\mathfrak{l}, E, E_\mathfrak{l}, \mathbb{C},$ we get
\begin{equation}
	\tilde{H}^2_{!}(S^{M_P}_{N \times N'}, \widetilde{\mathcal{M}}_{\mu +\mu', \mathcal{O}_E})\otimes_{\mathcal{O}_E} A 
	\cong \tilde{H}^2_{!}(S^{M_P}_{N \times N'}, \widetilde{\mathcal{M}}_{\mu +\mu', A}). \label{eqn: functorial-relations-for-the-levi-1}
\end{equation}

\smallskip
There will be variations on the cohomology of the Levi (as when we look at both sides of an intertwining operator), 
but the same recipe as above will be adopted for all variations.

\medskip
\subsection{Rankin--Selberg $L$-functions}

\subsubsection{Classical Rankin--Selberg $L$-functions}
\label{section: classical-rankin-selberg}
For integers $N', N \geq 1$, Dirichlet characters $\chi'$ and $\chi$ of levels $N'$ and $N$, respectively, and 
integers $k' > k \geq 2$ consider primitive forms 
$h' \in S_{k'}(N',\chi')^{\textup{new}}$ and $h \in S_k(N,\chi)^{\textup{new}}$, with Fourier expansions: 
$h'(z) = \sum_{n=1}^\infty a(n,h')q^n$ and $h(z) = \sum_{n=1}^\infty a(n,h)q^n,$ where as usual $q = e^{2\pi i z}.$
By a primitive form one means an eigenform, newform, and normalized as $a(1,h) = a(1,h') = 1$. 
For $s\in \mathbb{C}$ with $\Re(s) \gg 0$ define the finite-part of the Rankin--Selberg $L$-function as a Dirichlet series:
$$
L^{(\infty)}(s, h \times h') \ := \ L^{(M)}(s, \chi \chi') \left( \sum_{n=1}^\infty a(n,h) \, a(n,h')n^{-s}\right),
$$
where $M$ is the least common multiple of $N$ and $N',$ and $L^{(M)}(s, \chi'\chi)$ denotes the Dirichlet $L$-function attached to the character $\chi\chi'$ 
of level $M$ with the Euler factors at $p|M$ deleted. 
From Shimura \cite[Lem.\,1]{shimura-cusp} one has an Euler product: $L^{(\infty)}(s, h \times h') =\Pi_{p < \infty} L_p(s, h \times h' ).$ 
Keeping the assumption $k' > k$ in mind, the archimedean factor is defined by:
$$	
L_\infty(s, h \times h') \ := \ (2\pi)^{-2s} \Gamma(s)\Gamma(s + 1 - k).
$$
The completed $L$-function is defined by: $L(s, h \times h') = L_\infty(s, h \times h') L^{(\infty)}(s, h \times h').$
It is well-known that $L(s, h \times h')$ can be analytically continued to all of the complex plane, and satisfies a functional equation, towards which, 
define the action of complex conjugation via: 
$h'^\rho = \sum_{n=1}^\infty \overline{a(n,h')} q^n$ and $h^\rho = \sum_{n=1}^\infty \overline{a(n,h)}q^n.$ 
Then $h'$ and $h$ are newforms in $S_{k'}(N', \chi'^{-1})$ and $S_k(N, \chi^{-1}),$ respectively; 
see, for example, Miyake \cite[Thm.\,4.6.15]{miyake}. 
Also, $\overline{a(p,h)} = \chi(p)^{-1} a(p,h)$ and $\overline{a(p,h')} = \chi(p)^{-1} a(p,	h');$ see, for example, 
Shimura \cite[Prop.\,3.56]{iataf}. The functional equation then can be roughly stated as
$$
L(k'+k - 1 -s, h \times h') \ \approx \ L(s,h^\rho \times h'^\rho). 
$$
See Hida \cite[Thm.\,9.1]{hida-measure} for the precise factors involved. For our purposes it is enough to observe that for ratios of 
$L$-functions one has: 
\begin{equation}\label{eqn: relation-between-ratios}
	\dfrac{L(k'+k - 1 - s, h \times h')}{L(k'+k -s,h \times h')} = c(N,N',k,k')\dfrac{L(s, h^\rho\times h'^\rho)}{L(s-1, h^\rho \times h'^\rho)}, 
\end{equation}
where $c(N,N',k,k') \in \mathbb{Q}^\times$ is a constant which depends only on the prime factors of $N$ and $N'$ and on the weights $k'$ and $k$.

\smallskip
\subsubsection{Critical points for classical Rankin--Selberg $L$-functions}

The line of symmetry for the functional equation is $\Re(s) = (k+k'-1)/2.$  
An integer $m$ is {\it critical} for $L(s, h \times h'),$ if the archimedean factors on both sides of the functional equation are finite at $s = m,$ i.e., 
if $\Gamma(m)\Gamma(m+1-k)$ and $\Gamma(k+k' -1-m)\Gamma(k' -m)$ are finite. Therefore the critical set is:
\begin{equation}
\label{eqn:critical-set}
\{m \in \mathbb{Z} \,|\, k \leq m \leq k'-1\}.
\end{equation}
The number of critical points is $k'-k$. The condition $k' > k$ was imposed to guarantee the existence of critical points. 
For the main results on congruences for ratios of successive critical values, we will furthermore need to assume that $k'-k > 2.$

\subsubsection{Relation between classical and automorphic $L$-functions} \label{section: relation-between-classical-and-automorphic-l-functions}
Given primitive forms $h \in S_k(N,\chi)^{\textup{new}}$ and $h' \in S_{k'}(N',\chi')^{\textup{new}}$ as above, 
consider highest weights  
$\mu = \mu_k = (k-2,0)$ and $\mu' = \mu_{k'} = (k'-2,0),$ and Hecke modules in inner cohomology 
$\sigma \in \text{Coh}_!(G_2, \mu)$ and $\sigma' \in \text{Coh}_!(G_2,\mu')$, such that with respect to the embedding $\iota$
\begin{equation}
\label{eqn:choice-of-repns}
\sigma \cong \Pi(\mathbf{h})|\cdot|^{-k/2+1}, \quad \sigma' \cong \Pi(\mathbf{h}'^\rho)|\cdot|^{-k'/2+1}, 
\end{equation}
where $\mathbf{h}$  and $\mathbf{h}'^\rho$ are the $\C$ valued automorphic forms attached to $h$ and $h'^\rho$ respectively and $\Pi(\mathbf{h})$ and $\Pi(\mathbf{h}'^\rho)$ are the auotomorphic reprenstations generated by them. 
The reason for taking $h'^\rho$ for $\sigma'$ (instead of $h'$ itself) will become clear in \eqref{eqn:auto-L-fn-classical-L-fn} below. 
The reader is referred to Raghuram and Tanabe \cite{raghuram-tanabe} for details of the dictionary between the modular forms and cohomological cuspidal representations. 
In particular, one has the following relations: 
$$
L(s, \sigma) \ = \ L(s+ \frac12, h), \quad L(s, \sigma^{\sf v}) \ = \ L(s+ k - \frac32, h), 
$$ 
and similarly for $\sigma'$, $h'$ and $k'$. Furthermore, for an integer $m$, the Tate-twist $\sigma(-m)$ has cohomology with respect to $\mu(m)$, 
then we have the following equality (up to a nonzero constant) between the automorphic-representation theoretic and the classical 
Rankin--Selberg $L$-functions:
\begin{equation}
\label{eqn:auto-L-fn-classical-L-fn}
L(s, \sigma(-m) \times  \sigma'^{\sf v}) \ = \ L(s + k' -m -1, h \times h').
\end{equation}
The nonzero constant alluded to above will not play a role in this paper as we only consider the ratio of critical values. 
All this applies just the same to the pair $h \in S_k(N,\chi)^{\textup{new}}$ and $h'' \in S_{k'}(N',\chi')^{\textup{new}}$.

\medskip
\subsubsection{Setting-up the context of Eisenstein cohomology}
\label{sec:set-up-eis-coh}
To apply the machinery of \cite{harder-raghuram}, we will be looking at the intertwining operator 
between algebraically and parabolically induced representations:
$$
T_{\rm st}(s)|_{s = -2} : \aInd_P^G\left(\sigma(-m) \times \sigma')\right) \ \longrightarrow \ 
\aInd_P^G(\sigma'(2) \times \sigma(-m-2)), 
$$
which, as in {\it loc.\,cit.}, gives a rationality result for the ratio:
$$
\frac{L(-2, \sigma(-m) \times \sigma'^{\sf v})}{L(-1, \sigma(-m) \times \sigma'^{\sf v})} \ = \ 
\frac{L(-2-m, \sigma \times \sigma'^{\sf v})}{L(-1-m, \sigma \times \sigma'^{\sf v})} \ = \ 
\frac{L(k' -m -3, h \times h')}{L(k' -m -2, h \times h')},
$$
provided $m$ satisfies the constraints imposed by the combinatorial lemma (\cite[Lem.\,7.14]{harder-raghuram})
which is exactly equivalent to the above $L$-values being critical; from 
\eqref{eqn:critical-set} this imposes the following bounds on permissible Tate-twists $m$:
$$
-1 \ \leq \ m \ \leq \ k'-k-3.
$$
Furthermore, to carry out \cite{harder-raghuram}, the data $(\mu_k(m), \mu_{k'})$ needs to be on the right of the unitary axis 
(required for a certain Eisenstein series to to be holomorphic), which is the condition:
$$
-2 + \frac{k'-2}{2} - \left(\frac{k-2}{2} + m \right) \geq 0 \quad \iff \quad m \leq \frac{k'-k}{2} - 2.
$$
Hence, as $m$ varies from $-1$ to $\frac{k'-k}{2} - 2$, we are looking at the string of ratios of $L$-values from the rightmost up to a little more than the central value:
$$
\frac{L(k' -2, h \times h')}{L(k' -1, h \times h')}, \ \ 
\frac{L(k' -3, h \times h')}{L(k' -2, h \times h')}, \ \dots, \ \ 
\frac{L(\lfloor \frac{k+k' -1}{2} \rfloor, h \times h')}{L( \lfloor \frac{k+k' + 1}{2} \rfloor, h \times h')}. 
$$
If we are on the left of the unitary axis, then reversing the direction of the intertwining operator and using the functional equation 
offers the possibility of a result for all successive ratios critical values exactly as in \cite{harder-raghuram}; see the discussion in Sect.\,\ref{subsection: left-of-unitary-axis}.

\medskip
\section{Hecke algebras and Gorenstein property}
\label{sec:Hecke-Gorenstein}

\subsection{Classical Hecke algebras} For $A= \mathcal{O}_E, E, \mathcal{O}_\mathfrak{l}, E_\mathfrak{l},$ define an $A$ sub-module $S_k({N}, A)$ of $S_k(N)$: 
$$
S_k({N},  A) = \{ f = \sum_{n=1}^\infty a(n, f) q^n \in S_k(N) \ |\ \forall n \in \mathbb{N}, \,\, a(n,f) \in A\}
$$ 
where $S_k({N}) := S_k(\Gamma_1(N))$ is the $\C$-vector space of classical cusp forms. Here the isomorphism $\iota : \hat{\overline{E}}_\mathfrak{l} \cong \mathbb{C}$ is used implicitly.  
Recall a theorem of Shimura, Deligne, Rapoport and Katz; see Hida \cite[Chap.\,3]{mfg}.
\begin{theorem}
	For $A= \mathcal{O}_E, E, \mathcal{O}_\mathfrak{l}, E_\mathfrak{l},$ the space $S_k({N},A)$ is an $A$-module of full rank in $S_k({N})$, i.e.,
	\begin{equation}
		S_k({N},  A) \otimes_{\iota} \mathbb{C} = S_k({N}). \label{eqn: functorial-cusp-forms}
	\end{equation}
\end{theorem}
 For $A = \mathcal{O}_E, E, \mathcal{O}_\mathfrak{l}, E_\mathfrak{l},$ define $h_k({N}, A) \subset \text{End}_{A}(S_k({N}, A))$ 
 to be the Hecke algebra over $A$ generated by the operators $T(p)$ for all primes $p$ and $T(p,p)$ for $p \nmid N.$ 
 Due to the perfect pairing $( \cdot , \cdot ): S_k(N,A) \times h_k(N,A) \rightarrow A,\,\, (f,T) \mapsto a(1,f|T)$, Hida \cite[Thm.\,3.17]{mfg},  
 one gets $h_k({N}, \mathcal{O}_E) \otimes A = h_k({N}, A)$ for $A = E, \mathcal{O}_\mathfrak{l}, E_\mathfrak{l}, \mathbb{C}$ and
 as $h_k(N,A)$-modules:
\begin{gather}
	 S_k(N,A) \ \cong \ \textup{Hom}_{A}(h_k(N,A) , A) 
 \label{eqn: duality-between-hecke-algebras-and-cusp-forms}
\end{gather}

\subsection{Formalism of a Gorenstein datum} \label{section: gorenstein-formalism}
Suppose $R$ is the ring of integers of a local field of characteristic $0.$ Let
$\mathfrak{l} \subset R$ will be the unique maximal principal ideal of $R$ generated by $\varpi_R.$
Let $\mathbb{T}$ be a commutative $R$-algebra with $1$ which is also finite and free as an $R$-module. Since $R$ is complete 
$\mathbb{T}$ is complete as well. It is well known that $\mathbb{T} $ has only finitely many maximal ideals and each such ideal 
$\mathfrak{L}$ defines an idempotent $e_\mathfrak{L} \in \mathbb{T}.$ Furthermore, $e_\mathfrak{L}\mathbb{T} \cong \mathbb{T}_\mathfrak{L}$ and 
$\mathbb{T} = \sum_{\mathfrak{L}} e_\mathfrak{L} \mathbb{T} \cong \oplus_{\mathfrak{L}} \mathbb{T}_\mathfrak{L},$ where the sum is over the finitely many maximal ideals. 
Let $H$ be a \textit{fixed} free $R$--module and also a $\mathbb{T}$--module (not necessarily free over $\mathbb{T}$) and $\mathfrak{L} \subset \mathbb{T}$ a \textit{fixed} maximal ideal. 
Observe $H_\mathfrak{L} \cong e_\mathfrak{L} H$ and so $H_\mathfrak{L} \hookrightarrow H.$
In applications we will be under the following assumptions:
	\begin{enumerate}
		\item  There is a $\mathbb{T}_\mathfrak{L}$ equivariant isomorphism $\Phi_\mathfrak{L}: \mathbb{T}_\mathfrak{L} \xrightarrow{\sim} H_\mathfrak{L}$,
		\item $\text{Hom}_{R-\text{mod}}\left(\mathbb{T}_\mathfrak{L}, R \right) \cong 
		\mathbb{T}_\mathfrak{L} \;\;\; \text{(as $\mathbb{T}_{\mathfrak{L}}$-modules)}.$ 
	\end{enumerate}
	
The second assumption is the definition of a ring (here $\mathbb{T}_\mathfrak{L}$) being Gorenstein. 
Hereafter call the tuple $(R, \mathbb{T}, H, \mathfrak{L})$ which satisfies the the above assumptions to be a \textit{freely Gorenstein datum}.
Let $\Psi_\mathfrak{L}$ denote the isomorphism $\text{Hom}_{R-\text{mod}}\left(\mathbb{T}_\mathfrak{L}, R \right) \cong \mathbb{T}_\mathfrak{L}$
as $\mathbb{T}_{\mathfrak{L}}$-modules. 
	
\subsection{Presence of two congruent morphisms}
	Assume now there are two \textit{distinct} $R$-algebra morphisms $\Theta', \Theta'': \mathbb{T} \rightarrow R$ 
	such that their compositions with the map $R \rightarrow R/\mathfrak{l}^n$ are the same, i.e.,
	\[ \overline{\Theta}'_n = \overline{\Theta}''_n, \]
	where, $\overline{\Theta}'_n := {\Theta'}\pmod{\mathfrak{l}^n}$ and $\overline{\Theta}_n'' := {\Theta''}\pmod{\mathfrak{l}^n}.$ Here $n$ is assumed to be a positive integer. In particular,
	the kernels of $\overline{\Theta}_1'$ and $\overline{\Theta}_1''$ are the same; 
	put $\mathfrak{L} := \ker{\overline{\Theta}_1' = \ker{\overline{\Theta}_1''}}$ which is a maximal ideal. Hence the morphisms $\Theta'$ and $\Theta''$ factors through $\mathbb{T}_\mathfrak{L}$ which will be denoted again by the same symbols. It will be \textit{assumed} that $(R, \mathbb{T}, H, \mathfrak{L})$ is a freely-Gorenstein datum.

	\begin{lemma}
	\label{lem:theta'-theta''}
		Under the $\mathbb{T}_\mathfrak{L}$-equivariant isomorphism  $\Psi_\mathfrak{L}$ one has
		\begin{gather*}
			\Psi_\mathfrak{L}(\Theta'), \, \Psi_\mathfrak{L}(\Theta'') \ \not\in \ \mathfrak{l}^n\mathbb{T}_{\mathfrak{L}},\qquad
			\Psi_\mathfrak{L}(\Theta') - \Psi_\mathfrak{L}(\Theta'') \ \in \ \mathfrak{l}^n \mathbb{T}_{\mathfrak{L}}.
		\end{gather*}
		The algebra $\mathbb{T}_\mathfrak{L}$ acts on $\Psi_\mathfrak{L}(\Theta')$ and $\Psi_\mathfrak{L}(\Theta'')$ by the characters 
		$\Theta'$ and $\Theta''$  respectively.
	\end{lemma}
	
	\begin{proof}
		Since $\mathbb{T}$ is a free $R$-module, $\mathbb{T}_\mathfrak{L}$ is also a free $R$-module. Fix an $R$ basis $\{e_i\}_{i=1}^n$ of $\mathbb{T}_{\mathfrak{L}}$ to get a dual basis $\{e_i^\vee\}_{i=1}^n$ of $\text{Hom}_{R-\text{mod}}(\mathbb{T}_\mathfrak{L}, R).$
		The natural map $ \text{Hom}_{R-\text{mod}}(\mathbb{T}_\mathfrak{L}, R) \otimes_R R/\mathfrak{l}^n \longrightarrow \text{Hom}_{R-\text{mod}}(\mathbb{T}_\mathfrak{L}, R/\mathfrak{l}^n) $
		is $\sum_{i} e_i^\vee \otimes (r_i +\mathfrak{l}^n) \mapsto f : f(e_i) = r_i + \mathfrak{l}^n$ for all $i$.
		One checks the map is bijective.
		Using the isomorphism $\Psi_\mathfrak{L}$ we get the commutative diagram
		\[\begin{tikzcd}
			{\text{Hom}_{R-\text{mod}}\left(\mathbb{T}_\mathfrak{L}, R \right)} && {\mathbb{T}_\mathfrak{L}} \\
			{\text{Hom}_{R-\text{mod}}(\mathbb{T}_\mathfrak{L}, R) \otimes_R R/\mathfrak{l}^n} && {\mathbb{T}_\mathfrak{L} \otimes_R R/\mathfrak{l}^n} \\
			{\text{Hom}_{R-\text{mod}}(\mathbb{T}_\mathfrak{L}, R/ \mathfrak{l}^n)} && {\mathbb{T}_\mathfrak{L}/\mathfrak{l}^n\mathbb{T}_\mathfrak{L}.}
			\arrow["{\sim \;\;\text{via} \;\;\Psi_\mathfrak{L}}", from=1-1, to=1-3]
			\arrow[two heads, from=1-1, to=2-1]
			\arrow[two heads, from=1-3, to=2-3]
			\arrow["\sim", from=2-1, to=2-3]
			\arrow["\sim"', from=2-1, to=3-1]
			\arrow["\sim", from=2-3, to=3-3]
			\arrow["{\sim \;\;\text{via} \;\;\overline{\Psi}_\mathfrak{L}}", from=3-1, to=3-3]
		\end{tikzcd}\]
		
		Since $\overline{\Theta}_n' = \overline{\Theta}_n''$ in $\text{Hom}_{R-\text{mod}}(\mathtt{T}_\mathfrak{L}, R/\mathfrak{l}^n)$ we have that $ \overline{\Psi}_\mathfrak{L}(\overline{\Theta}'_n) - \overline{\Psi}_\mathfrak{L}(\overline{\Theta}''_n) = \overline{\Psi}_\mathfrak{L}(\overline{\Theta}'_n - \overline{\Theta}''_n) = 0$ in $\mathbb{T}_\mathfrak{L}/ \mathfrak{l}^n\mathbb{T}_\mathfrak{L}.$ Due to the commutativity of the above diagram the first assertion is true as ${\Theta}'(1)= {\Theta}''(1) = 1 \not\in \mathfrak{l}^n .$
		
		Now, for  the second claim. For all $t \in \mathbb{T}_\mathfrak{L}$ one has  $t\cdot \Theta' = \Theta'(t) \Theta'$ since for all $x \in \mathbb{T}_\mathfrak{L}$ we have $(t\cdot \Theta')(x) = \Theta'(x t)
		= \Theta'(t x) 
		= \Theta'(t) \Theta'(x).$ Since $\Psi_\mathfrak{L}$ is $\mathbb{T}_\mathfrak{L}$ equivariant it is $R$--linear as well and so $	t \cdot \Psi_\mathfrak{L}( \Theta') = \Psi_{\mathfrak{L}}(t\cdot \Theta')
			= \Psi_{\mathfrak{L}}(\Theta'(t) \Theta')
			= \Theta'(t) \Psi_{\mathfrak{L}}(\Theta').$
	\end{proof}

	\begin{lemma}
		Set $v(\Theta'):= (\Phi_\mathfrak{L} \circ \Psi_\mathfrak{L})(\Theta')$ and 
		$v(\Theta''):= (\Phi_\mathfrak{L} \circ \Psi_\mathfrak{L})(\Theta'').$ 
		On the vectors $v(\Theta')$ and $v(\Theta''),$ the algebra $\mathtt{T}_{\mathfrak{L}}$ acts by $\Theta'$ and $\Theta'',$ respectively. Moreover,
		$$
			v(\Theta'), \ v(\Theta'') \ \not\in \ \mathfrak{l}^n  H_\mathfrak{L}, \quad \text{and} \quad 
			v(\Theta')- v(\Theta'') \ \in \ \mathfrak{l}^n  H_{\mathfrak{L}}.
		$$ 
	\label{Lem:congruent-vectors}
	\end{lemma}
	\begin{proof}
	Follows from Lem.\,\ref{lem:theta'-theta''} and that both $\Psi_\mathfrak{L}$ and $\Phi_\mathfrak{L}$ are $\mathbb{T}_\mathfrak{L}$ equivariant.
	\end{proof}

\subsection{Specializing to a particular $\mathbb{T}$ and $H$}
\label{sec:particular-T-H}
Take $\mathbb{T} $ to be the Hecke algebra $ h_{k'}(N',\mathcal{O}_\mathfrak{l})$ 
and $H$ the cohomology group $H^1_!(X_1(N'), \widetilde{\mathcal{M}}_{\mu', \mathcal{O}_\mathfrak{l}})(\epsilon')$. 
Let $\Theta': h_{k'}(N', \mathcal{O}_\mathfrak{l}) \rightarrow \mathcal{O}_\mathfrak{l}$ 
(resp., $\Theta'': h_{k'}(N', \mathcal{O}_\mathfrak{l}) \rightarrow \mathcal{O}_\mathfrak{l}$) be the morphisms 
$T(p) \mapsto a(p,h')$ (resp., $T(p) \mapsto a(p,h'')$). Recall our context: $h' \equiv h'' \pmod{\mathfrak{l}^n}$ which implies 
$\overline{\Theta}'_n = \overline{\Theta}''_n$. 
Let $\mathfrak{L} \subset h_{k'}(N',\mathcal{O}_\mathfrak{l})$ 
be the maximal ideal determined by the hypothesis $\overline{\Theta}'_1 = \overline{\Theta}''_1.$
Assume $E$ is large so that $\mathbb{T}/\mathfrak{L} \cong  \mathcal{O}_\mathfrak{l}/\mathfrak{l}.$
Let ${\varrho}_{\Theta'} : G_{\mathbb{Q}} \rightarrow \GL_2(\mathbb{T}/\mathfrak{L})$ be the Galois representation modulo $l$ attached to the 
Hecke algebra morphism $\Theta'$, constructed by Deligne \cite{deligne}. It is a semi-simple representation determined by
\begin{equation}
	\textup{Trace} \,\, {\varrho}_{\Theta'}(\textup{Frob}_{p}) = a(p,h') \pmod{\mathfrak{l}}\quad \text{ and } \quad \det{{\varrho}_{\Theta'}(\textup{Frob}_p)} = \chi'(p) p^{k'-1} \pmod{\mathfrak{l}}. 
\end{equation}
for $p\nmid N'l .$ (See, for example, Gross \cite[p.\,483, Prop.\,11.1]{gross}.) 
Consider the twist ${\varrho}_{\Theta'}\chi'^{-1}$ of the representation of ${\varrho}_{\Theta'}$ by the character ${\chi'}^{-1}$ which satisfies: 
\begin{multline}
		\textup{Trace} \,\, {\varrho}_{\Theta'}\chi'^{-1}(\textup{Frob}_{p}) = a(p,h')\chi'^{-1}(p) \pmod{\mathfrak{l}}\quad \text{ and }\\ \quad 
		\det{{\varrho}_{\Theta'}\chi'^{-1}(\textup{Frob}_p)} = \chi'^{-2}(p) p^{k'-1} \pmod{\mathfrak{l}}.
\end{multline}
for $p \nmid N' l.$ This is the mod $l$ representation attached to the newform $h'^\rho = \sum_{n=1}^\infty \overline{a(n,h')}q^n$ 
or equivalently the Hecke algebra morphism 
$\Theta'^\rho: h_{k'}(N', \mathcal{\mathfrak{l}}) \rightarrow \mathcal{O}_\mathfrak{l},\,\, T(p) \mapsto \chi'^{-1}(p)\,\,a(p,h),$ 
due to $\overline{ a(p,h)}= \chi'^{-1}(p)\,\,a(p,h)$ for $p\nmid lN'.$
Also, ${\varrho}_{\Theta'}$ is irreducible if and only if ${\varrho}_{\Theta'}\chi'^{-1}$ is irreducible, and 
the maximal ideal of $h_{k'}(N', \mathcal{O}_\mathfrak{l})$ determined by $\Theta'$ and $\Theta'^\rho$ are the same.

\begin{theorem}
\label{thm:Hecke-Gorenstein}
	Assume $l > k'$ and $(l,6N') = 1$ and the Galois representation ${\varrho}_{\Theta'}$ is irreducible. Then 
	\begin{equation}
		(\mathcal{O}_\mathfrak{l}, \ h_{k'}(N', \mathcal{O}_\mathfrak{l}), \ 
		{H}^1_!(X_1(N'), \underline{\mathcal{M}}_{\mu',\mathcal{O}_\mathfrak{l}})(\epsilon'), \ \mathfrak{L})
	\end{equation}
	is a freely Gorenstein datum.
\end{theorem}

\begin{proof}
	The $\mathcal{O}_\mathfrak{l}$-freeness of $H^1_!(X_1(N'), \underline{\mathcal{M}}_{\mu',\mathcal{O}_\mathfrak{l}})(\epsilon')$ is due to the fact $l > k'$. 
	It is proved in Faltings and Jordan \cite[Thm.\,2.1]{faltings-jordan} that $h_{k'}(N',\mathcal{O}_\mathfrak{l})_\mathfrak{L}$ is Gorenstein. 
	The freeness condition in the definition of a freely Gorenstein datum follows from the perfect pairing between Hecke algebras and 
	cusp forms as in \eqref{eqn: duality-between-hecke-algebras-and-cusp-forms} and the Eichler-Shimura isomorphism. 
\end{proof}

\subsection{Explicit congruent cohomology classes}
\label{sec:explicit-coh-classes}

Lem.\,\ref{Lem:congruent-vectors} when applied to the particular context of Sect.\,\ref{sec:particular-T-H} gives the following

\begin{theorem} 
\label{thm:congruent-coh-classes}
Under the assumption on the congruence prime $l > k'$ and $(l,6N') = 1,$
	there are cohomology classes 
	${^\circ}v^{\epsilon'}_{\mu'}(h'^\rho)$ and ${^\circ}v^{\epsilon'}_{\mu'}(h''^\rho)$ in 
	$\tilde{H}^1_!(S_1^{(2)}({N}'), \widetilde{\mathcal{M}}_{\mu', \mathcal{O}_\mathfrak{l}})(\epsilon')$ 
	on which $h_{k'}({N'}, \mathcal{O}_\mathfrak{l})$ acts by $\Theta'^\rho$ and $\Theta''^\rho,$ respectively. 
	They are non-zero modulo $\mathfrak{l}^n$ but their difference is $0$ modulo $\mathfrak{l}^n$, i.e., 
	\begin{align} 
	\begin{aligned}
		{^\circ}v_{\mu'}^{\epsilon'}(h'^\rho), \ {^\circ}v_{\mu'}^{\epsilon'}(h''^\rho) &
		 \ \not\in \ \mathfrak{l}^n \tilde{H}^1_!(S_1^{(2)}({N}'), \widetilde{\mathcal{M}}_{\mu', \mathcal{O}_\mathfrak{l}})(\epsilon'), \\
		{^\circ}v^{\epsilon'}_{\mu'}(h'^\rho) - {^\circ}v^{\epsilon'}_{\mu'}(h''^\rho) &
		 \ \in \ \mathfrak{l}^n  \tilde{H}^1_!(S_1^{(2)}({N}'), \widetilde{\mathcal{M}}_{\mu', \mathcal{O}_\mathfrak{l}})(\epsilon').
	\end{aligned}
	\end{align}
\end{theorem}

Also, fix a cohomology class for the cusp form $h$ which is non-zero modulo $\mathfrak{l}^n$:
\begin{equation}
	{^\circ}v^{\epsilon'}_{\mu}(h) \in \tilde{H}^1_!(S_1^{(2)}(N), \widetilde{\mathcal{M}}_{\mu,\mathcal{O}_\mathfrak{l}})(\epsilon' \otimes \sigma_f) 
	\setminus \mathfrak{l}^n\tilde{H}^1_!(S_1^{(2)}(N), \widetilde{\mathcal{M}}_{\mu,\mathcal{O}_\mathfrak{l}})(\epsilon' \otimes \sigma_f),
\end{equation}
where $\tilde{H}^1_!(S_1^{(2)}(N), \widetilde{\mathcal{M}}_{\mu,\mathcal{O}_\mathfrak{l}})(\epsilon' \otimes \sigma_f)$ is 
$\tilde{H}^1_!(S_1^{(2)}(N), \widetilde{\mathcal{M}}_{\mu,\mathcal{O}_\mathfrak{l}})(\epsilon')\cap 
\tilde{H}^1_!(S_1^{(2)}(N), \widetilde{\mathcal{M}}_{\mu,E_\mathfrak{l}})(\epsilon' \otimes \sigma_f).$ 

It is helpful to keep in mind that 
${^\circ}v^{\epsilon'}_{\mu'}(h'^\rho)= {^\circ}v^{\epsilon'}_{\mu'(-n)}(h'^\rho)$ and 
${^\circ}v^{\epsilon'}_{\mu}(h)= {^\circ}v^{\epsilon'}_{\mu(-m)}(h)$ for any $m,n\in \mathbb{Z}.$ Only the action of the Hecke algebra is different on these vectors.
Fix the notations
\begin{gather}
	\begin{split}
	{^\circ}v^{\epsilon'}_{\mu + \mu'}(h,h'^\rho) \ := \ {^\circ}v^{\epsilon'}_{\mu}(h) \otimes {^\circ}v^{\epsilon'}_{\mu'}(h'^\rho),\qquad
	{^\circ}v^{\epsilon'}_{\mu'(-2)+\mu(2)}(h'^\rho,h) \ := \ {^\circ}v^{\epsilon'}_{\mu'(-2)}(h'^\rho)\otimes {^\circ}v^{\epsilon'}_{\mu(2)}(h),\\
	{^\circ}v^{\epsilon'}_{\mu + \mu'}(h,h''^\rho) \ := \ {^\circ}v^{\epsilon'}_{\mu}(h) \otimes {^\circ}v^{\epsilon'}_{\mu'}(h''^\rho), \qquad
	{^\circ}v^{\epsilon'}_{\mu'(-2)+\mu(2)}(h''^\rho,h) \ := \ {^\circ}v^{\epsilon'}_{\mu'(-2)}(h''^\rho)\otimes {^\circ}v^{\epsilon'}_{\mu(2)}(h).
	\end{split}
\end{gather}
The vectors ${^\circ}v^{\epsilon'}_{\mu}(h) \otimes {^\circ}v_{\mu^?}^{\epsilon}(h^{?\rho})$  
are in 
$H^2_{!}(S_{N\times N'}^{M_P}, \widetilde{\mathcal{M}}_{\mu+\mu', E})(\tilde{\epsilon}'\otimes \sigma_f \otimes \sigma^?_f)$ for ${^?}\in \{ ', ''\}.$ 
Similar comments apply to the other vectors.

\begin{corollary}
	The vectors ${^\circ}v^{\epsilon'}_{\mu + \mu'}(h,h'^\rho), {^\circ}v^{\epsilon'}_{\mu + \mu'}(h,h''^\rho)$ 
	are not zero modulo $\mathfrak{l}^n$, but are congruent modulo $\mathfrak{l}^n$, i.e.,
	\begin{gather*}
		{^\circ}v^{\epsilon'}_{\mu + \mu'}(h,h'^\rho), \ {^\circ}v^{\epsilon'}_{\mu + \mu'}(h,h''^\rho) \ \not\in \ 
		\mathfrak{l}^n  H^2_{!}(S^{M_P}_{N\times N'}, \widetilde{\mathcal{M}}_{\mu+\mu', \mathcal{O}_\mathfrak{l}})(\tilde{\epsilon}'),\\
		{^\circ}v^{\epsilon'}_{\mu + \mu'}(h,h'^\rho) - {^\circ}v^{\epsilon'}_{\mu + \mu'}(h,h''^\rho) \ \in \ 
		\mathfrak{l}^n H^2_{!}(S^{M_P}_{N\times N'}, \widetilde{\mathcal{M}}_{\mu+\mu', \mathcal{O}_\mathfrak{l}})(\tilde{\epsilon}').
	\end{gather*}
	Similarly,
	\begin{gather*}
		{^\circ}v^{\epsilon'}_{\mu'(-2)+\mu(2)}(h'^\rho,h), \ {^\circ}v^{\epsilon'}_{\mu'(-2)+\mu(2)}(h''^\rho,h)  \ \not\in \ 
		\mathfrak{l}^n  H^2_{!}(S^{M_P}_{N'\times N}, \widetilde{\mathcal{M}}_{_{\mu'(-2)+\mu(2)}, \mathcal{O}_\mathfrak{l}})(\tilde{\epsilon}'),\\
		{^\circ}v^{\epsilon'}_{\mu'(-2)+\mu(2)}(h'^\rho,h) - {^\circ}v^{\epsilon'}_{\mu'(-2)+\mu(2)}(h''^\rho,h) \ \in \ 
		\mathfrak{l}^n H^2_{!}(S^{M_P}_{N'\times N}, \widetilde{\mathcal{M}}_{_{\mu'(-2)+\mu(2)}, \mathcal{O}_\mathfrak{l}})(\tilde{\epsilon}').
	\end{gather*}
\end{corollary}

The auxiliary cusp form $h$ has no bearing on the choice of the prime $\mathfrak{l}$. 
So we do not impose the hypothesis that $l \not\in \mathsf{S}_{k},$ where $k$ is the weight of the form $h$.

\bigskip
\section{Double coset representatives}
\label{section: double-coset-representatives}
Next, we need to consider certain specific vectors in induced representations built from the vectors in Sect.\,\ref{sec:explicit-coh-classes}. 
Towards this, while using Mackey theory, we need to understand certain double cosets. Specifically, in this section, 
we calculate a set of representatives $\Xi_p$ of the double cosets 
$P(\mathbb{Q}_p) \backslash {\rm GL}_4(\mathbb{Q}_p) / \mirp{n_p+n_p'},$ 
where $\mirp{n_p+n'_p}$ is the Mirahoric subgroup of ${\rm GL}_4(\mathbb{Z}_p)$ of level $n_p+n'_p$ and $P$ is the $(2,2)$ parabolic subgroup.

\medskip
\subsection{Calculation for the Borel and principal congruence subgroups}
Thm.\,\ref{thm: coset-borel-principal} and Cor.\,\ref{corollary: coser-rep-disjoint} below are essentially due to Januszewski \cite{fabian}. 
We follow closely the notation therein and reproduce the proof with a minor modification.
Recall the Iwahori decomposition of $\mathbb{Q}_p$ points of $G_n= {\rm GL}_n/\mathbb{Q}.$
\[ G_n(\mathbb{Q}_p) = \coprod_{w \in W_n} B_n(\mathbb{Q}_p) w I_n, \]
where $B_n$ is the set of upper triangular matrices, $W_n$ is the Weyl group identified with the set of permutation matrices and 
$I_n$ is the Iwahori subgroup of $G_n(\mathbb{Z}_p).$
Let $J_m$ denote the set of principal congruence subgroup of level $m$ of $G_n(\mathbb{Z}_p)$, i.e., 
the set of matrices of $g \in G_n(\mathbb{Z}_p)$ such that $g \equiv \mathbf{1}_N \pmod{p^m}.$ 
Let $R_m \subset \mathbb{Z}_p$ denote the complete set of coset representatives $\{0,p,p^2, \cdots, p^{m-1}\}$ of $\mathbb{Z}_p/p^m\mathbb{Z}_p.$ 
Then the set 
\[ \mathcal{R}_m = \{r = (r_{ij})_{ij} \in I_N | r_{ij} \in R_m\}\]
forms a complete set of (left) coset representatives for $I_n/J_m.$ 
For $w \in W_n$ we have the obvious inclusion:
$B_n(\mathbb{Q}_p)w s J_m \subset  B_n(\mathbb{Q}_p)wI_n,$ for $s \in \mathcal{R}_m.$
We have
\[  B_n(\mathbb{Q}_p) w I_n = \bigcup_{s \in \mathcal{R}_{m}}  B_n(\mathbb{Q}_p) w s J_m.\]
The cosets on the right need not be distinct. 
The following theorem shows that we can take the union over a smaller set of representatives  
and still get $B(\mathbb{Q}_p) w I_n.$ Before stating it we need some more notations.
Let $U_B$ denote the unipotent radical of $B_n$ consisting of strictly upper traingular matrices and $U_B^{-}$ denote the unipotent radical of 
the opposite Borel subgroup $B_n^{-}$ of lower triangular matrices of $G_n$.
For a fixed $w\in W_n$ define $B^w := B_n(\mathbb{Q}_p) \cap w I_n w^{-1}$ and 
$\mathcal{R}_m^w := \mathcal{R}_{m,B} \cap w^{-1} U_B^{-}(\mathbb{Q}_p) w.$
 
\begin{theorem}
\label{thm: coset-borel-principal}
	For a fixed $w \in W_n$ the set $\mathcal{R}^w_{m,B}$ forms a complete system of coset representatives for 
	$B(\mathbb{Q}_p) \backslash B(\mathbb{Q}_p)wI_n/J_m,$ i.e.,
	\[ B_n(\mathbb{Q}_p)wI_n = \bigcup_{s \in \mathcal{R}_{m,B}^w} B_n(\mathbb{Q}_p)w s J_m.\] 
\end{theorem}

\begin{proof}
The map sending the coset $B^w w s J_m \mapsto B_n(\mathbb{Q}_p) w s J_m$, for 
$s \in \mathcal{R}_m$, is injective. Clearly, it is surjective as well. So it is enough to show that $w \mathcal{R}_{m,B}^w$ forms a system of representatives for 
$B^w\backslash B^w w I_n/ J_m.$ Consider the following
\begin{ass}
\label{ass} 
	One can find elements $u^{(0)} := \mathbf{1}_n, u^{(1)}, u^{(2)}, \cdots , u^{(n-1)} \in B^w,$ and recursively define 
	$r^{(0)}, r^{(1)}, \cdots, r^{(n)} \in I_n$ where $r^{(0)} = \mathbf{1}_n$ and $r^{(v+1)} = w^{-1}u^{(v)}w \cdot r^{(v)}$ for $v > 1$, 
	such that $r^{(n)} \in w^{-1}U_B^{-}(\mathbb{Z}_p)w.$ 
\end{ass}	
From this assumption it follows if we define $u := u^{(n-1)} \cdots u^{(1)} \in B^w$, then
$$
u \cdot w r w^{-1} \ = \ w \cdot w^{-1} u wr w^{-1} 
\ = \ 
w w^{-1} u^{(n-1)} \cdots u^{(1)} w r w^{-1}
\ = \ 
w r^{(n)} w^{-1} \in U_B^{-}(\mathbb{Z}_p).
$$
Suppose $s \in \mathcal{R}_m$ is a representative of the left coset $w^{-1}u w r J_m \in I_n / J_m$ then it follows 
	\[  wsJ_m = u wr J_m \implies B^w ws J_m = B^w wr J_m.\]
So $ws$ represents the same double coset as $wr$ in $B^w \backslash B^w wI_n/ J_m$ (equivalently, the same coset in $B(\mathbb{Q}_p)\backslash B(\mathbb{Q}_p)wI_n/ J_m$). 
Since $s \in w^{-1}u w r J_m$, this implies
$wsw^{-1} \in u w r w^{-1} J_m \subset U_B^{-}(\mathbb{Z}_p)J_m,$ hence 
$s \in w^{-1}U_B^{-}(\mathbb{Z}_p)w J_m,$ whence 
$s \in \mathcal{R}_m \cap w^{-1} U_B^{-}(\mathbb{Z}_p) w = \mathcal{R}_{m,B}^w.$
\end{proof}

Assump.\,\ref{ass} can be shown by changing the definition of the $(n-v, n-v)$ entry of $u^{(v)}$ as in 
Januszewski \cite[Prop.\,2.2]{fabian}. This is possible because here $u^{(v)} \in B_n(\mathbb{Q}_p)$ and not just in $U_B(\mathbb{Q}_p)$ as in \textit{loc.cit.}.

\begin{corollary} \label{corollary: coser-rep-disjoint}
	With the notations as in the previous theorem we have
	\[B_n(\mathbb{Q}_p)wI_n = \coprod_{s \in \mathcal{R}_{m,B}^w} B_n(\mathbb{Q}_p)w s J_m. \]
	In other words, the double cosets in the previous theorem are all disjoint.
\end{corollary}

\begin{proof}
	Assume two of the cosets, $B_n(\mathbb{Q}_p) ws J_m = B_n(\mathbb{Q}_p) ws' J_m$ for $s, s' \in \mathcal{R}_{m,B}^w$ are the same. 
	This means $B^w ws J_m = B^w ws' J_m.$ Since $B^w = B_n(\mathbb{Q}_p) \cap w I_N w^{-1} = B_n(\mathbb{Z}_p) \cap w I_N w^{-1},$ 
	we see that there exists $u\in B^w \subset B_n(\mathbb{Z}_p)$ and $j \in J_m$ such that $u ws = ws' j.$
	Observing that $J_m$ is normal in $G_n(\mathbb{Z}_p)$ we get that
	$u = ws'w^{-1} \cdot w s^{-1} w^{-1} j'$  
	for some $j' \in J_m.$ But both the elements $ ws'w^{-1}$ and $w s^{-1} w^{-1}$ are in $U_B^{-}(\mathbb{Z}_p)$ from the previous theorem. 
	Combined with the fact that $u \in B_n(\mathbb{Z}_p)$ and $ u \in U_B^{-}(\mathbb{Z}_p)J_m$ we get 
	$u \equiv \textbf{1}_n \pmod {J_m}$  if and only if $u \in J_m.$
	Then for some $j'' \in J_m$ depending on $u$ we have 
	$u ws = ws j'' = ws'j$ because  $J_m$ is normal in $G_n(\mathbb{Z}_p).$ Hence $s\equiv s' \pmod{J_m}$, whence $s = s'.$
\end{proof}

\subsection{Calculation for the parabolic and the mirahoric subgroups}

Now we focus on $G_4 = \GL_4/\mathbb{Q}$, and $P$ the $(2,2)$ parabolic subgroup of $G_4$ with the Levi decomposition $P = M_PU_P.$ Let $U_P^-$ be the opposite unipotent radical contained in $B_4^-.$
The Iwahori decomposition gives 
\[ G_4(\mathbb{Q}_p) = \coprod_{w \in W^P} P(\mathbb{Q}_p) w I_4,\]
where $W^P$ is the set of Kostant representatives.
This is due to the fact that the Weyl group of the Levi quotient $W_{M_P} \subset P(\mathbb{Q}_p)$ and 
there is a bijection between $W_{M_P} \backslash W_4 \cong W^P.$ 
For a fixed $w \in W_4$ define $P^w := P(\mathbb{Q}_p) \cap wI_4 w^{-1}$ and $\mathcal{R}^w_m := \mathcal{R}_m \cap w^{-1} U_P^{-}(\mathbb{Q}_p) w.$ 
(This $\mathcal{R}_m^w$ is different from the one in the previous subsection; this abuse of notation will not cause any confusion.)

\begin{theorem}\label{thm: coset-parabolic-principal}
	For a fixed $w \in W^P$ the set $\mathcal{R}^w_m$ forms a complete system of coset representatives for 
	$P(\mathbb{Q}_p) \backslash P(\mathbb{Q}_p)wI_4/J_m,$ i.e., 
	\[ P(\mathbb{Q}_p)wI_4 = \coprod_{s \in \mathcal{R}_m^w} P(\mathbb{Q}_p)w s J_m.\]
	Moreover, we have from the Iwahori decomposition
	\[ G_4(\mathbb{Q}_p) = \coprod_{w \in W^P} \coprod_{s \in \mathcal{R}_m^w} P(\mathbb{Q}_p)w s J_m.\]
\end{theorem}

\begin{proof}
(The proof is essentially the same as Thm.\,\ref{thm: coset-borel-principal}.)
The map $P^w w s J_m \mapsto P(\mathbb{Q}_p)ws J_m$, for $s \in \mathcal{R}_m$, is injective. Clearly, it is surjective as well. 
So it is enough to show that $w\mathcal{R}_m^w$ forms a system of represetnatives for $P^w \backslash P^w w I_N/ J_m.$ 
Suppose Assump.\,\ref{ass} holds for $u \in U_B^{-}(\mathbb{Q}_p)$, giving 
$u^{(0)}, u^{(1)}, u^{(2)}, u^{(3)} \in B^w$ and $r^{(0)}, r^{(1)}, r^{(2)}, r^{(3)}, r^{(4)} \in I_4$. Define $u^{(4)} \in P(\mathbb{Z}_p)$ and $r^{(5)} \in I_4$ by
	\begin{equation*}
		u^{(4)} = \left(\begin{smallmatrix}
			1 & 0 & 0 & 0 \\
			-r^{(4)}_{\sigma_w(2)\sigma_w(1)} & 1 & 0 & 0\\
			0 & 0 & 1 & 0\\
			0 & 0 & -r^{(4)}_{\sigma_w(4)\sigma_w(3)} & 1
		\end{smallmatrix}\right), \qquad 
		r^{(5)} = w^{-1}u^{(4)}w \cdot r^{(4)}.
	\end{equation*}
Here $\sigma_w$ is the image of $w$ under the usual (group) isomorphism $W_4 \cong S_4,$ where $S_4$ is the permutation group of the set $\{1,2,3,4\}$.  It is clear that $u^{(4)} \in P^w$ and $w r^{(5)} w ^{-1} \in U_P^{-}(\mathbb{Z}_p).$ Therefore if $u = u^{(4)}u^{(3)}u^{(2)}u^{(1)}$ then 
$u \cdot w r w^{-1} \in w I_4 w^{-1} \cap U_P^{-}(\mathbb{Z}_p).$ The rest of the arguments are essentially the same as in 
Thm.\,\ref{thm: coset-borel-principal} and Cor.\,\ref{corollary: coser-rep-disjoint}.
\end{proof}

Recall the notations $n_p = v_p(N)$ and $ n'_p = v_p(N'),$ and $\mirp{n_p + n'_p}$ is the mirahoric subgroup of $\GL_4(\mathbb{Z}_p)$ of level 
$p^{n_p+n'_p}\mathbb{Z}_p$. Since $J_{n_p +n'_p} \subset \mirp{n_p+n'_p},$ from Thm.\,\ref{thm: coset-parabolic-principal} one gets

\begin{corollary}
	The set $\bigcup_{w\in W^P} \{ws | s \in \mathcal{R}_{n_p+n'_p}^w\}$ contains a complete set of coset representatives for 
	$P(\mathbb{Q}_P) \backslash G_4(\mathbb{Q}_p) / \mirp{n_p+n'_p}.$ 
\label{corollary: abstract-reprsentatives}
\end{corollary}
The representatives in the above corollary up to left action of $P(\mathbb{Q}_p)$ and right action of 
$K_p^{n'_p + n_p}$ need not be distinct.

\subsection{Explicit representatives}
\hfill

For $\xi_p \in P(\mathbb{Q}_p) \backslash \GL_4(\mathbb{Q}_p) / \mirp{n_p + n'_p}$ 
define  $K_p^P(\xi_p) = P(\mathbb{Q}_p) \cap \xi_p \mirp{n_p + n'_p} \xi_p^{-1}$ and 
$K_p^{M_P}(\xi_p) = \kappa(K_p^P(\xi_p))$ denote its projection onto the Levi quotient via the canonical map $\kappa_P: P \rightarrow P/U_P \cong M_P.$ 
Let us enumerate the set $W^P$ of Kostant representatives thus:
\begin{multline*} w_1=
	\left(\begin{smallmatrix}
		1 & 0 & 0 & 0 \\
		0 & 1 & 0 & 0 \\
		0 & 0 & 1 & 0 \\
		0 & 0 & 0 & 1
	\end{smallmatrix}\right), \,\, w_2= 
	\left(\begin{smallmatrix}
		1 & 0 & 0 & 0 \\
		0 & 0 & 1 & 0 \\
		0 & 1 & 0 & 0 \\
		0 & 0 & 0 & 1
	\end{smallmatrix}\right), \,\, w_3 =
	\left(\begin{smallmatrix}
		0 & 1 & 0 & 0 \\
		0 & 0 & 1 & 0 \\
		1 & 0 & 0 & 0 \\
		0 & 0 & 0 & 1
	\end{smallmatrix}\right), \\ w_4 = 
	\left(\begin{smallmatrix}
		1 & 0 & 0 & 0 \\
		0 & 0 & 0 & 1 \\
		0 & 1 & 0 & 0 \\
		0 & 0 & 1 & 0
	\end{smallmatrix}\right), \,\, w_5 = 
	\left(\begin{smallmatrix}
		0 & 1 & 0 & 0 \\
		0 & 0 & 0 & 1 \\
		1 & 0 & 0 & 0 \\
		0 & 0 & 1 & 0
	\end{smallmatrix}\right), \,\, w_6 = 
	\left(\begin{smallmatrix}
		0 & 0 & 1 & 0 \\
		0 & 0 & 0 & 1 \\
		1 & 0 & 0 & 0 \\
		0 & 1 & 0 & 0
	\end{smallmatrix}\right). \end{multline*}
Cor.\,\ref{corollary: abstract-reprsentatives} may be restated as that $\cup_{I=1}^6 \{w_i s | s \in \mathcal{R}_{n_p + n'_p}^{w_i}\}$ 
contains a complete set of coset representatives for $P(\mathbb{Q}_P) \backslash G_4(\mathbb{Q}_p) / \mirp{n_p+n_p'}.$
For $ u = \left(\begin{smallmatrix}
	1 & 0 & 0 & 0\\
	0 & 1 & 0 & 0\\
	x_1 & x_2 & 1 & 0\\
	x_3 & x_4 & 0 & 1
\end{smallmatrix}\right) \in U_P^{-}(\mathbb{Z}_p),$ and $i=4,5,6,$ the matrices $u_i := w_i^{-1}uw_i$, explicitly given by:
$$
 u_4= 
	\left(\begin{smallmatrix}
		1 & 0 & 0 & 0 \\
		x_1 & 1 & 0 & x_2\\
		x_3 & 0 & 1 & x_4 \\
		0 & 0 & 0 & 1
	\end{smallmatrix}\right), \quad 
	u_5= 
	\left(\begin{smallmatrix}
		1 & x_1 & 0 & x_2 \\
		0 & 1 & 0 & 0 \\
		0 & x_3 & 1 & x_4 \\
		0 & 0 & 0 & 1
	\end{smallmatrix}\right), \quad {\rm and} \quad 
	u_6 = 
	\left(\begin{smallmatrix}
		1 & 0 & x_1 & x_2 \\
		0 & 1 & x_3 & x_4 \\
		0 & 0 & 1 & 0 \\
		0 & 0 & 0 & 1
	\end{smallmatrix}\right), 
$$
are clearly in $\mirp{n_p + n'_p}.$ Therefore
$P(\mathbb{Q}_p) w_i s \mirp{n_p + n'_p} = P(\mathbb{Q}_p) w_i \mirp{n_p + n'_p}$ for $i=4,5,6,$ and $s \in \mathcal{R}_{n_p + n'_p}^{w_i}.$ 
Define
$
	\xi_p^{(0)} := \left(\begin{smallmatrix}
		1 & 0 & 0 & 0 \\
		0 & 1 & 0 & 0 \\
		0 & 0 & 1 & 0 \\
		0 & 1 & 0 & 1
	\end{smallmatrix}\right). 
$
\begin{lemma}
\label{Lem:some-coset-reductions}
\hfill
\begin{enumerate}
\item[(i)] \begin{multline*}P(\mathbb{Q}_p)w_4\mirp{n_p + n'_p} = 	P(\mathbb{Q}_p)w_5\mirp{n_p + n'_p} \\ =
		P(\mathbb{Q}_p)w_6\mirp{n_p + n'_p} =	P(\mathbb{Q}_p)\xi_p^{(0)}\mirp{n_p + n'_p}.\end{multline*}

\smallskip
\item[(ii)] $K_p^M(\xi_p^{(0)}) = \kappa_P( P(\mathbb{Q}_p) \cap \xi_p^{(0)}  \mirp{n_p + n'_p} {\xi_p^{(0)}}^{-1} )  = K_p(n_p + n'_p) \times \GL_2(\mathbb{Z}_p).$
\end{enumerate}
\end{lemma}

\begin{proof}
For (i), observe that
$$
		w_4 \ = \ w_6\left(\begin{smallmatrix}
			0 & 1 & 0 & 0 \\
			0 & 0 & 1 & 0 \\
			1 & 0 & 0 & 0 \\
			0 & 0 & 0 & 1
		\end{smallmatrix}\right) 
		\quad \text{and} \quad 
		\left(\begin{smallmatrix}
			0 & 1 & 0 & 0 \\
			0 & 0 & 1 & 0 \\
			1 & 0 & 0 & 0 \\
			0 & 0 & 0 & 1
		\end{smallmatrix}\right) \in \mirp{n_p + n'_p}
$$
and, similarly, 
$$
w_5 \ = \ w_6 \left(\begin{smallmatrix}
			1 & 0 & 0 & 0 \\
			0 & 0 & 1 & 0 \\
			0 & 1 & 0 & 0 \\
			0 & 0 & 0 & 1
		\end{smallmatrix}\right)
		\quad \text{ and } \quad 
		\left(\begin{smallmatrix}
			1 & 0 & 0 & 0 \\
			0 & 0 & 1 & 0 \\
			0 & 1 & 0 & 0 \\
			0 & 0 & 0 & 1
		\end{smallmatrix}\right) \in \mirp{n_p + n'_p}.
$$
Hence $P(\mathbb{Q}_p)w_4 \mirp{n_p + n'_p} = P(\mathbb{Q}_p)w_5 \mirp{n_p + n'_p} =
	P(\mathbb{Q}_p)w_6 \mirp{n_p + n'_p}.$ To get the last equality of (i), further observe that
	\begin{equation*}
		w_6=
		\left(\begin{smallmatrix}
			1 & 0 & 0 & 0 \\
			0 & -1 & 0 & 1 \\
			0 & 0 & 1 & 0 \\
			0 & 0 & 0 & 1
		\end{smallmatrix}\right)
		\left(\begin{smallmatrix}
			1 & 0 & 0 & 0 \\
			0 & 1 & 0 & 0 \\
			0 & 0 & 1 & 0 \\
			0 & 1 & 0 & 1
		\end{smallmatrix}\right)
		\left(\begin{smallmatrix}
			1 & 0 & 0 & 0 \\
			0 & 1 & 0 & -1 \\
			0 & 0 & 1 & 0 \\
			0 & 0 & 0 & 1
		\end{smallmatrix}\right)
		\left(\begin{smallmatrix}
			0 & 0 & 1 & 0 \\
			0 & 1 & 0 & 0 \\
			1 & 0 & 0 & 0 \\
			0 & 0 & 0 & 1
		\end{smallmatrix}\right)
	\end{equation*}
	with
$
		\left(\begin{smallmatrix}
			1 & 0 & 0 & 0 \\
			0 & -1 & 0 & 1 \\
			0 & 0 & 1 & 0 \\
			0 & 0 & 0 & 1
		\end{smallmatrix}\right) \in P(\mathbb{Q}_p)$ and 
$\left(\begin{smallmatrix}
			1 & 0 & 0 & 0 \\
			0 & 1 & 0 & -1 \\
			0 & 0 & 1 & 0 \\
			0 & 0 & 0 & 1
		\end{smallmatrix}\right), \left(\begin{smallmatrix}
			0 & 0 & 1 & 0 \\
			0 & 1 & 0 & 0 \\
			1 & 0 & 0 & 0 \\
			0 & 0 & 0 & 1
		\end{smallmatrix}\right) \in \mirp{n_p + n'_p}.$
This completes the proof of (i). 

For (ii), since 
$P(\mathbb{Q}_p) \cap \xi_p^{(0)}  \mirp{n_p + n'_p} {\xi_p^{(0)}}^{-1} = P(\mathbb{Q}_p) \cap w_6 \mirp{n_p + n'_p}w_6^{-1}$, 
and for an element 
$k =\left(\begin{smallmatrix}
		k_{11} & k_{12} & k_{13} & k_{14} \\
		k_{21} & k_{22} & k_{23} & k_{24} \\
		k_{31} & k_{32} & k_{33} & k_{34} \\
		k_{41} & k_{42} & k_{43} & k_{44}
	\end{smallmatrix}\right)$ in $\mirp{n_p + n'_p}$, 
one has
$w_6  k w_6^{-1} = 
	\left(\begin{smallmatrix}
		k_{33} & k_{34} & k_{31} & k_{32} \\
		k_{43} & k_{44} & k_{41} & k_{42} \\
		k_{13} & k_{14} & k_{11} & k_{12} \\
		k_{23} & k_{24} & k_{21} & k_{22}
	\end{smallmatrix}\right),$ 
(ii) follows. 
\end{proof}

From Lem.\,\ref{Lem:some-coset-reductions}, and the discussion preceding it, one has:
\[ 
\bigcup_{i=4,5,6 \ s \in \mathcal{R}_{n_p + n'_p}^{w_i}}
P(\mathbb{Q}_p) w_i s \mirp{n_p + n'_p} \ = \ 
P(\mathbb{Q}_p) \xi_p^{(0)} \mirp{n_p + n'_p}.
\]
Now, we will consider the other double cosets represented by $w_i s$ with $i = 1, 2, 3,$ and $s$ as before.  
For an element $ u  \in U_P^{-}(\mathbb{Z}_p)$ as before, and for $i=1,2,3,$ the matrices $u_i := w_i^{-1}uw_i$ are
\begin{equation*}
u_1 =
	\left(\begin{smallmatrix}
		1 & 0 & 0 & 0 \\
		0 & 1 & 0 & 0 \\
		x_1 & x_2 & 1 & 0 \\
		x_3 & x_4 & 0 & 1
	\end{smallmatrix}\right), \quad 
u_2 = 
	\left(\begin{smallmatrix}
		1 & 0 & 0 & 0 \\
		x_1 & 1 & x_2 & 0 \\
		0 & 0 & 1 & 0 \\
		x_3 & 0 & x_4 & 1
	\end{smallmatrix}\right), \quad 
u_3 = 
	\left(\begin{smallmatrix}
		1 & x_1 & x_2 & 0 \\
		0 & 1 & 0 & 0 \\
		0 & 0 & 1 & 0 \\
		0 & x_3 & x_4 & 1
	\end{smallmatrix}\right).
\end{equation*}

For $i = 1, 2, 3,$ if $w_i^{-1}uw_i \in \mathcal{R}_{n_p + n'_p}^{w_i}$ then since $\mathcal{R}_{n_p + n'_p}^{w_i} \subset I_4,$ it  is necessary 
that either $x_1,x_2,x_3,x_4 \in p\mathbb{Z}_p,$ or $x_1,x_3,x_4 \in p\mathbb{Z}_p,$ or $x_3, x_4 \in p \mathbb{Z}_p$ depending on whether $i=1$ or $2$ or $3$. 
Henceforth, assume $v_p(x_3), v_p(x_4) > 0.$ Moreover, $P(\mathbb{Q}_p) w_i w_i^{-1}uw_i \mirp{n_p + n'_p} = P(\mathbb{Q}_p) u \mirp{n_p + n'_p}$ 
since $w_i \in \mirp{n_p+n_p'}.$

\begin{lemma}
For $u \in U_P^{-}(\mathbb{Z}_p)$ such that $u \in \mathcal{R}_{n_p + n'_p}^{w_i}$ for $i=1,2,3,$ 
The double coset $P(\mathbb{Q}_p) u K_p^{n_p + n'_p}$ is also represented by one of the elements: 
\[ 
\xi_p^{(j)}:=\left(\begin{smallmatrix} 1 & 0 & 0 & 0 \\ 0 & 1 & 0 & 0\\ 0 & 0 & 1 & 0 \\ 0 & p^j & 0 & 1\end{smallmatrix}\right) , 
\]
for some $0 < j \leq n_p + n'_p.$ 
\end{lemma}
\begin{proof}
Let
$
u = \left(\begin{smallmatrix}
			1 & 0 & 0 & 0 \\
			0 & 1 & 0 & 0 \\
			x & y & 1 & 0 \\
			z & w & 0 & 1
		\end{smallmatrix}\right) \in U_P^{-}(\Z_p),$ 
with $v_p(z), v_p(w) > 0.$ If both $w, z = 0$ then $u \in \mirp{n_p + n'_p}$. So assume that's not the case. 
If necessary, conjugating $u$ by the matrix
$\left(\begin{smallmatrix}
		0 & 1 & 0 & 0 \\
		1 & 0 & 0 & 0 \\
		0 & 0 & 1 & 0 \\
		0 & 0 & 0 & 1
	\end{smallmatrix}\right)$
which belongs to both $P(\mathbb{Q}_p)$ and $\mirp{n_p + n'_p}$, assume
$v_p(z) \geq v_p(w) > 0$ with $v_p(w) \neq  \infty$ or in other words $v_p(z) \geq v_p(w) > 0$ and $w \neq 0.$
If $z = 0$ then we skip to the next step. If $z \neq 0$ then since
$$
\left(\begin{smallmatrix}
			1 & 0 & 0 & 0 \\
			0 & 1 & 0 & 0 \\
			x & y & 1 & 0 \\
			z & w & 0 & 1
		\end{smallmatrix}\right) 
= 
		\left(\begin{smallmatrix}
			z^{-1}w p^{v_p(z) - v_p(w)} & 0 & 0 & 0 \\
			- p^{v_p(z) - v_p(w)} & 1 & 0 & 0 \\
			(x z^{-1}w - y) p^{v_p(z) - v_p(w)} & y & 1 & 0 \\
			0 & w & 0 & 1
		\end{smallmatrix}\right) 
		\left(\begin{smallmatrix}
			1 & 0 & 0 & 0 \\
			p^{v_p(z) - v_p(w)} & 1 & 0 & 0 \\
			0 & 0 & 1 & 0\\
			0 & 0 & 0 & 1
		\end{smallmatrix}\right)
		\left(\begin{smallmatrix}
			w^{-1}z p^{-(v_p(z) - v_p(w))} & 0 & 0 & 0 \\
			0 & 1 & 0 & 0 \\
			0 & 0 & 1 & 0 \\
			0 & 0 & 0 & 1
		\end{smallmatrix}\right).
$$
and the observation that the last two matrices are in $\mirp{n_p + n'_p}$, one is reduced to the case that $z = 0$. 
Define $x' := (xz^{-1}w-y)p^{v_p(z) - v_p(w)}.$ It is clear that $x'\in \mathbb{Z}_p$.
Again, if $y = 0$ then skip to the next step. If $y \neq 0$ then since
$$
		\left(\begin{smallmatrix}
			1 & 0 & 0 & 0 \\
			0 & 1 & 0 & 0 \\
			x' & y & 1 & 0 \\
			0 & w & 0 & 1
		\end{smallmatrix}\right)
		= \left(\begin{smallmatrix}
			1 & 0 & 0 & 0 \\
			0 & y^{-1}p^{v_p(y)} & 0 & 0 \\
			x' & 0 & 1 & 0 \\
			0 & wy^{-1} v_p(y) & 0 & 1
		\end{smallmatrix}\right) 
		\left(\begin{smallmatrix}
			1 & 0 & 0 & 0 \\
			0 & 1 & 0 & 0 \\
			0 & p^{v_p(y)} & 1 & 0 \\
			0 & 0 & 0 & 1
		\end{smallmatrix}\right)
		\left(\begin{smallmatrix}
			1 & 0 & 0 & 0 \\
			0 & p^{-v_p(y)} y & 0 & 0\\
			0 & 0 & 1 & 0\\
			0 & 0 & 0 & 1
		\end{smallmatrix}\right)
$$
and the observation that the last two matrices are in $\mirp{n_p + n'_p}$ and so one can assume $y = 0.$ Note that $v_p(wy^{-1}p^{v_p(y)}) = v_p(w)$. 
Using 
	\[ \left(\begin{smallmatrix}
		1 & 0 & 0 & 0 \\
		0 & 1 & 0 & 0 \\
		b & 0 & 1 & 0 \\
		0 & c & 0 & 1
	\end{smallmatrix}\right)=
	\left(\begin{smallmatrix}
		1 & 0 & 0 & 0 \\
		0 & 1 & 0 & 0 \\
		0 & 0 & 1 & 0 \\
		0 & c & 0 & 1 
	\end{smallmatrix}\right)
	\left(\begin{smallmatrix}
		1 & 0 & 0 & 0 \\
		0 & 1 & 0 & 0 \\
		b & 0 & 1 & 0 \\
		0 & 0 & 0 & 1
	\end{smallmatrix}\right)
	\]
it can be assumed $x' = 0.$ Finally, if necessary, conjugating by a diagonal matrix which is in both 
$P(\mathbb{Q}_p)$ and $K_p^{(4)}(n_p + n'_p),$ one see that what remains is one of the $\xi_p^{(j)}.$
\end{proof}

\begin{theorem}
For $0\leq i \leq n_p + n'_p,$ we have 
$K_p^{M_P}(\xi_p^{(i)}) = \congp{n_p + n'_p - i} \times \congp{i}.$ In particular,
$$
K_p^{M_P}(\xi_p^{(n_p)}) = \congp{n'_p} \times \congp{n_p}, \quad \text{and} \quad
K_p^{M_P}(\xi_p^{(n'_p)}) = \congp{n_p} \times \congp{n'_p}.
$$ \label{thm: local-lemma-on-subgroup-of-the-levi}
\end{theorem}

\begin{proof} 
For $k = \left(\begin{smallmatrix}
	k_{11} & k_{12} & k_{13} & k_{14} \\
	k_{21} & k_{22} & k_{23} & k_{24} \\
	k_{31} & k_{32} & k_{33} & k_{34} \\
	k_{41} & k_{42} & k_{43} & k_{44}
\end{smallmatrix}\right) \in \mirp{n_p + n'_p}$ note that 
\begin{equation} 
\xi_p^{(i)} k {\xi_p^{(i)}}^{-1}  = \left(\begin{smallmatrix}
		k_{11} & -k_{14} p^{i} + k_{12} & k_{13} & k_{14} \\
		k_{21} & -k_{24} p^{i} + k_{22} & k_{23} & k_{24} \\
		k_{31} & -k_{34} p^{i} + k_{32} & k_{33} & k_{34} \\
		k_{21} p^{i} + k_{41} & -{\left(k_{24} p^{i} + k_{44}\right)} p^{i} + k_{22} p^{i} + k_{42} & k_{23} p^{i} + k_{43} & k_{24} p^{i} + k_{44}
	\end{smallmatrix}\right). \label{eqn: conjuagate-wrt-the-rep}
\end{equation}
The case $i = 0$ has already been proved. For $i=n_p + n'_p$, since $\xi_p^{(n_p + n'_p)} \in \mirp{n_p + n'_p}$, observe that
\begin{multline*}
K_p^M(\xi_p^{(n_p + n'_p)}) \ = \ \kappa_P( P(\mathbb{Q}_p) \cap \xi_p^{(n_p + n'_p)} \mirp{n_p + n'_p} {\xi_p^{(n_p + n'_p)}}^{-1}) \\
= \ \kappa_P( P(\mathbb{Q}_p) \cap  \mirp{n_p + n'_p} ) 
= \ {\rm GL}_2(\mathbb{Z}_p) \times \congp{n_p + n'_p}.
\end{multline*}
Assume now that $0< i < n_p + n'_p.$ Since $k_{44} \equiv 1 \pmod{p^{n_p + n'_p}},$ one has 
$k_{24}p^{i} + k_{44} \equiv 1 \pmod{p^{i}}.$ 
Also, since $v_p(k_{42}) \geq n_p + n'_p$ one has $v_p(k_{42}p^{-i}) \geq n_p + n'_p-i.$
The $(4,2)$-entry of the matrix in \eqref{eqn: conjuagate-wrt-the-rep} is $0$ (because we are looking at a situation when 
$\xi_p^{(i)} k {\xi_p^{(i)}}^{-1}$ is in $P(\mathbb{Q}_p)$), which is
\begin{multline*} 
-(k_{24}p^{i} + k_{44})p^{i} + k_{22}p^{i} + k_{42} = (-k_{24}p^{i} +k_{22})p^{i} -k_{44}p^{i} + k_{42} = 0 \\ \implies  
(-k_{24}p^{i} +k_{22}) = k_{44} - k_{42}p^{-i} \equiv 1 \pmod{p^{n_p +n'_p-i}}.
\end{multline*}
In other words, under the assumption $\xi_p^{(i)} k {\xi_p^{(i)}}^{-1} \in P(\mathbb{Q}_p),$ the $(2,2)$ and the $(4,4)$ entry of the matrix in 
\eqref{eqn: conjuagate-wrt-the-rep} are congruent to $1 \pmod{p^{n_p + n'_p-i}}$ and $1 \pmod{p^{i}},$ respectively.
Since $k_{21}p^{i} + k_{41} = 0$, we get that $v_p(k_{21}) \geq n_p + n'_p - i$. 
Similarly, $v_p(k_{23}p^{i} + k_{43}) \geq i$ as $v_p(k_{43}) \geq n_p + n'_p.$ 
Also, note that $-k_{34}p^{i} + k_{32}=0 \implies v_p(k_{32}) \geq i> 0$ and $k_{31} = 0.$ Now, calculating the determinant by expanding the last row we get 
	\begin{align*}
		\mathbb{Z}_p^\times \ni \det{k} & = -k_{44}[ k_{31}(\cdots) - k_{32}(\cdots) + k_{33}( k_{11} k_{22} - k_{12}k_{21})] \\
		&\qquad + k_{43}[\cdots] -k_{42}[\cdots] + k_{41}[\cdots]\\
		&= -k_{44}[0 - p^{i} (\cdots) + k_{33}(k_{11}k_{22} - p^{n_p + n'_p - i}(\cdots))] \\&\qquad + p^{n_p + n'_p}[\cdots] - p^{n_p + n'_p}[\cdots] + p^{n_p + n'_p}[\cdots]\\
		&= -k_{44}k_{33}k_{11}k_{22} + p^{i}(\cdots) + p^{n_p + n'_p -i}(\cdots)  \:\:\: (\text{after re-grouping}).
	\end{align*}
	This shows that $k_{11}k_{22}$ and $k_{33}k_{44}$ are units in $\mathbb{Z}_p$ as $ 0 < i < n_p + n'_p.$ Therefore,
$$
\det{\left(\begin{array}{rr}
				k_{11} & -k_{14} p^{i} + k_{12} \\
				k_{21} & -k_{24} p^{i} + k_{22}
			\end{array}\right)}  
= -{\left(k_{24} p^{i} - k_{22}\right)} k_{11} + {\left(k_{14}p^{i} - k_{12}\right)} k_{21} 
= k_{11}k_{22} + p^{i} (\cdots) \in \mathbb{Z}_p^\times.
$$
Similarly,
$$
\det \left(\begin{array}{rr}
		k_{33} & k_{34} \\
		k_{23} p^{i} + k_{43} & k_{24} p^{i} + k_{44}
	\end{array}\right) \in \mathbb{Z}_p^\times .
$$
Combined with the previous observations $v_p(k_{21}) \geq n_p + n'_p -i$ and $v_p(k_{23}p^{n_p} + k_{43}) \geq i,$ and $(-k_{24}p^i + k_{22}) \equiv 1 \pmod{p^{n_p + n'_p - i}}$ and  $k_{24}p^{n_p} + k_{44} \equiv 1 \pmod{p^i}$ shows that 
	\[ K_p^{M_P}(\xi_p^{(i)}) \subset \congp{n_p + n'_p - i}\times \congp{i}.\]

For the reverse containment take an arbitrary $\left( \left( \begin{smallmatrix}
	a & b \\ c & d
\end{smallmatrix}\right), \left( \begin{smallmatrix}
a' & b' \\ c' & d'
\end{smallmatrix}\right) \right) \in K_p^{(2)}(n_p+n'_p - i) \times K_p^{(2)}(i).$ Then, one checks that 
$$
\tilde{k} = \left( \begin{smallmatrix}
	a & b & 0 & 0 \\
	c & d + d'-1 & c'p^{-i} & (d'-1)p^{-i} \\
	0 & b'p^{i} & a' & b' \\
	-cp^i & (1-d)p^i & 0 & 1
\end{smallmatrix}\right)
$$ 
is in $K_p^{(4)}(n_p + n'_p), \,\xi_p^{(i)} \tilde{k} {\xi_p^{(i)}}^{-1} \in P(\mathbb{Q}_p)$ and 
$\kappa_P( \xi_p^{(i)} \tilde{k} {\xi_p^{(i)}}^{-1}) = 
\left( \left( \begin{smallmatrix}
	a & b \\ c & d
\end{smallmatrix}\right), \left( \begin{smallmatrix}
	a' & b' \\ c' & d'
\end{smallmatrix}\right) \right).$
\end{proof}

\subsection{A corollary for global level structures}

Let $\underline{i} = ( i_p)_{p|NN'}$ with $i_p \in \{0,\cdots, n_p + n'_p\}.$ Set 
$$	
\xi_f^{(\underline{i})} = \prod_{p|NN'} \{\xi_p^{(i_p)}\} \times \prod_{p \nmid NN'} \{\mathbf{1}_4\} \in {\rm GL}_4(\mathbb{A}_f), \quad 
N_{(\underline{i})} = \prod_{p | NN'} p^{i_p}, \quad 
N^{(\underline{i})} = NN'/N_{(\underline{i})}.
$$

\begin{corollary}
\label{corollary: main-corollary-on-coset-reprentatives}
For $K_f = \mirf{N'+N}$
	\begin{gather*}
		K_f^{M_P}(\xi_f^{(\underline{i})}) = \congf{N^{(\underline{i})}} \times \congf{N_{(\underline{i})}}.
	\end{gather*}
In particular, if $\underline{i} = (n_p)_{p|NN'}$ \textup{(}{\rm resp.}, $\underline{i} = (n'_p)_{p|NN'}$\textup{)} then we have
$$
K_f^{M_P}{(\xi_f^{(\underline{i})})} = \congf{N'} \times \congf{N} \quad 
\mbox{\textup{(}{\rm resp.}, $K_f^{M_P}{(\xi_f^{(\underline{i})})} = \congf{N} \times \congf{N'}$.\textup{)}}
$$
\end{corollary}

As a shorthand for the notation $\xi_f^{(\underline{i})}$ when $\underline{i} = (n_p)_{p|NN'}$ (resp., $\underline{i} = (n'_p)_{p|NN'}$) 
will be denoted as $\xi_f^{(N)}$ (resp., $\xi_f^{(N')}$.)

\medskip
\section{Integral structures on the induced space}
\label{sec:int-on-ind-space}
Parabolically induced representations appear in the cohomology of the Borel--Serre boundary stratum for a given parabolic subgroup 
$P$ in an ambient reductive group $G$. In arithmetic applications as in \cite{harder-raghuram} one compares two such pieces in the 
cohomology of the boundary. In this article, we need to further refine the constructions of {\it loc.\,cit.}\ to work integrally. 
In this section we define an integral structure on the invariants under of an open compact subgroup of an induced space 
via the Mackey isomorphism.

\smallskip

Suppose $V$ is an admissible $M_P(\mathbb{A}_f)$--module over $E$ (resp., $E_\mathfrak{l}$). 
Let ${^a}\text{Ind}(V)$ denote the algebraic induction from $P(\mathbb{A}_f)$ to $\GL_4(\mathbb{A}_f)$ of $V$ 
after inflating it to $P(\mathbb{A}_f).$ 
If $K_f$ is an open compact subgroup of $\GL_4(\mathbb{A}_f),$ then one has the Mackey isomorphism:
\begin{equation}
{^a}\text{Ind}(V)^{K_f} \xrightarrow{\sim} \bigoplus_{\xi_f \in P(\mathbb{A}_f) \backslash \GL_4(\mathbb{A}_f) / K_f} V^{K_f^{M_P}(\xi_f)},  
\quad \quad 
\phi_f \mapsto \sum_{\xi_f} \phi_f(\xi_f), \label{eqn: the-mackey-isomorphism}
\end{equation} 	where $K_f^{M_P}(\xi_f) = \kappa_P(P(\mathbb{A}_f) \cap \xi_f\, K_f \,\xi_f^{-1})$ is a subgroup of $M_P(\mathbb{A}_f)$ for every $\xi_f.$ 
Suppose now each $V^{K_f^{M_P}(\xi_f)}$ has an $\mathcal{O}_E$ (resp. $\mathcal{O}_\mathfrak{l}$)--lattice, 
say ${^\circ}V^{K_f^{M_P}(\xi_f)},$ then an $\mathcal{O}_E$ (resp. $\mathcal{O}_\mathfrak{l}$)--lattice in ${^a}\text{Ind}(V)^{K_f}$ 
is defined to be all the vectors $\phi_f$ in the algebraically induced space  such that $\phi_f(\xi_f) \in {^\circ}V^{K_f^{M_P}(\xi_f)}.$

\smallskip

Now, specialize to the mirahoric subgroup $K_f = \mirf{N+N'}.$  For Hecke modules in inner cohomology 
$\sigma' \in \textup{Coh}_!(G_2,\mu')$ and $\sigma \in \textup{Coh}_!(G_2, \mu),$ 
and for $A = E, E_\mathfrak{l}, \mathbb{C},$ define 
$I_4^\mathsf{S}(\sigma_f, \sigma'_f, \epsilon',A)$ to be the $\mirf{N+N'}$-invariants of the algebraic-parabolic induction of 
an isotypic component in the cohomology of $M_P$:
\begin{equation}
\label{eqn:I-sigma-sigma'}
	I_4^\mathsf{S}(\sigma_f, \sigma'_f, \epsilon',A) \ := \ 
	{^a}\text{Ind}\left(H^2_{!}(S^{M_P}, \widetilde{\mathcal{M}}_{\mu+\mu',A})(\tilde{\epsilon}'\otimes \sigma_f \otimes \sigma'_f)\right)^{\mirf{N+N'}},
\end{equation}
where $\tilde{\epsilon'} = \epsilon'\times\epsilon'.$  
Similarly, define the spaces: 
$$
I_4^\mathsf{S}(\sigma_f, \sigma''_f, \epsilon',A), \quad I_4^\mathsf{S}(\sigma'_f(2), \sigma_f(-2), \epsilon',A), 
\quad 
I_4^\mathsf{S}(\sigma''_f(2), \sigma_f(-2), \epsilon',A).
$$ 
The latter two are the $\mirf{N+N'}$-invariants of the representations induced from isotypic components in the cohomology of the $M_P$ with 
coefficients $\mathcal{M}_{\mu'(-2) + \mu(2), A}.$ 
Now, collect all these induced spaces (for $A = E, E_\mathfrak{l}, \mathbb{C}$): 
\begin{gather*} 
	I_4^\mathsf{S}(\mu+\mu', A) := \bigoplus_{\epsilon'}\bigoplus_{\sigma \in \text{Coh}_!(G_2,\mu)} 
	\bigoplus_{\sigma' \in \text{Coh}_!(G_2,\mu')} I_4^\mathsf{S}(\sigma_f, \sigma'_f, \epsilon',A),\\
	I_4^\mathsf{S}(\mu'(-2) + \mu(2), A) := \bigoplus_{\epsilon'}\bigoplus_{\sigma \in \text{Coh}_!(G_2,\mu)} 
	\bigoplus_{\sigma' \in \text{Coh}_!(G_2,\mu')} I_4^\mathsf{S}(\sigma'_f(2), \sigma_f(-2), \epsilon',A ).
\end{gather*}
The notations $I_4^\mathsf{S}(\mu + \mu', \epsilon', A)$ and $I_4^\mathsf{S}(\mu'(-2)+\mu(2), \epsilon', A)$ will mean the $\tilde{\epsilon}'$ isotypic component. 
There are only finitely many summands because we have taken $\mirf{N+N'}$ invariants.

\smallskip

Applying the isomorphism in \eqref{eqn: the-mackey-isomorphism} and Cor.\,\ref{corollary: main-corollary-on-coset-reprentatives} 
one gets for $A = E, E_\mathfrak{l}, \mathbb{C}$
\begin{gather*}
	I_4^\mathsf{S}(\mu+\mu', A) \cong \bigoplus_{\epsilon'}\bigoplus_{\substack{M, M' \\ MM' = NN'}} H_{!}^2(S^{M_P}_{M\times M'}, \widetilde{\mathcal{M}}_{\mu+\mu', A})(\tilde{\epsilon}'),\\
		I_4^\mathsf{S}(\mu'(-2) + \mu(2), A) \cong \bigoplus_{\epsilon'} \bigoplus_{\substack{M, M' \\ MM' = NN'}} H_{!}^2(S^{M_P}_{M'\times M}, \widetilde{\mathcal{M}}_{\mu'(-2) + \mu(2), A})(\tilde{\epsilon}').
\end{gather*}

\smallskip
Now we can appeal to the discussion in the beginning of this section; the $\mathcal{O}_E$ or $\mathcal{O}_\mathfrak{l}$--lattice are clear; 
define for ${A^\circ} = \mathcal{O}_E, \mathcal{O}_\mathfrak{l}$:
\begin{gather*}
	I_4^\mathsf{S}(\mu+\mu',A^\circ) := \bigoplus_{\epsilon'}\bigoplus_{\substack{M,M' \\ MM' = NN'}} \tilde{H}_{!}^2(S^{M_P}_{M\times M'}, \widetilde{\mathcal{M}}_{\mu+\mu', A^\circ}),\\
	I_4^\mathsf{S}(\mu'(-2) + \mu(2),A^\circ) := \bigoplus_{\epsilon'}\bigoplus_{\substack{M,M' \\ MM' = NN'}} \tilde{H}_{!}^2(S^{M_P}_{M'\times M}, \widetilde{\mathcal{M}}_{\mu'(-2) + \mu(2), A^\circ}).
\end{gather*}
These are lattices in $I_4^\mathsf{S}(\mu+\mu',A)$ and $I_4^\mathsf{S}(\mu'(-2) + \mu(2),A)$ respectively when $A = E, E_\mathfrak{l}.$ 
It follows from the definitions that for $A = \mathcal{O}_\mathfrak{l}, E, E_\mathfrak{l}, \mathbb{C}$: 
\begin{multline*}
	I_4^\mathsf{S}(\mu + \mu', \mathcal{O}_E) \otimes_{\mathcal{O}_E} A \cong I_4^\mathsf{S}(\mu+\mu', A), \quad \\ \text{and} \quad 
	I_4^\mathsf{S}(\mu'(-2) + \mu(2),\mathcal{O}_E)\otimes_{\mathcal{O}_E} A \cong I_4^\mathsf{S}(\mu'(-2) + \mu(2), A).
\end{multline*}

\medskip
\section{Congruence of the Eisenstein operator}
	
\subsection{Review of Harder and Raghuram \cite{harder-raghuram}}
We briefly summarize the technical results of \cite{harder-raghuram}; especially, Sect.\,5.3.7, Thm.\,5.12, and the proof of Thm.\,6.2 in Sect.\,6.3.7.

\smallskip

 Here we assume that the pair of weights $(\mu, \mu')$ satisfies the conditions of the combinatorial lemma (see \cite[Lem.\,7.14]{harder-raghuram}); 
this then gives us $\lambda$ a weight on $\GL_4/\mathbb{Q}$ which is of the form 
$\lambda = w^{-1}\cdot(\mu+\mu')$ for a Kostant representative $w$ with $l(w) = 2 = \dim(U_P)/2$. The Eisenstein operator comes about as follows: 
assume the pair of weights $(\mu,\mu')$ is on the \text{right} of the unitary axis. 
For $\tau' \in \textup{Coh}_!(G_2, \mu')$ and $\tau \in \textup{Coh}_!(G_2, \mu)$ the image of the composition of maps:
	\begin{multline*}
		H^{4}(S^{(4)}, \widetilde{\mathcal{M}}_{\lambda, E})^{\mirf{N+N'}} \xrightarrow{\mathfrak{r}^*} 
		H^4(\partial S^{(4)}, \widetilde{\mathcal{M}}_{\lambda, E})^{\mirf{N+N'}}\\ 
		\xrightarrow{\mathfrak{R}^4_{\tau_f, \tau'_f, \epsilon'}}
		 I_4^\mathsf{S}(\tau_f, \tau'_f, \epsilon',E) \oplus I_4^{\mathsf{S}}(\tau'_f(2), \tau_f(-2), \epsilon',E)
	\end{multline*}
is a $\mathsf{k}_{\tau}$-dimensional subspace of $I_4^\mathsf{S}(\tau_f, \tau'_f, \epsilon',E) \oplus I_4^{\mathsf{S}}(\tau'_f(2), \tau_f(-2), \epsilon',E),$ 
where $\mathsf{k}_\tau$ is the common dimension of the two summands $I_4^\mathsf{S}(\tau_f, \tau'_f, \epsilon',E)$ and  
$I_4^{\mathsf{S}}(\tau'_f(2), \tau_f(-2), \epsilon',E).$ The image, denoted as $\mathfrak{I}^4(\tau_f,\tau'_f,\epsilon',E)$, is of the form
	\begin{equation*}
		\mathfrak{I}^4(\tau_f,\tau'_f,\epsilon',E)= \{ (\phi_f , \phi_f + T_{\text{Eis}}(\tau, \tau', \epsilon',E) \phi_f ) \,|\,\, \phi_f \in I_4^\mathsf{S}(\tau_f, \tau'_f, \epsilon',E)\}
	\end{equation*}
where the Eisenstein operator $T_{\rm Eis}$ is such that
	\begin{equation*}
		T_{\textup{Eis}}(\tau,\tau',\epsilon', E) \otimes_{\iota} \mathbb{C} = T_{\textup{st}}(-2, \tau \otimes \tau')^\bullet, 
	\end{equation*}
i.e., the Eisenstein operator is a rational `avatar' of the map induced in cohomology by the standard intertwining operator which is only defined at the transcendental level. 
Now collect all the summands by running over all the Hecke modules in inner cohomology. Define: 
$$
\mathfrak{R}^4_{\mu,\mu'} \ := \ \sum_{\epsilon'} \sum_{\tau'_f} \sum_{\tau_f} \ \mathfrak{R}^4_{\tau_f, \tau'_f, \epsilon'}.
$$ 
Applying the above discussion on the Eisenstein operator, the image 
	\begin{multline}
		H^{4}(S^{(4)}, \widetilde{\mathcal{M}}_{\lambda, E})^{\mirf{N+N'}} \xrightarrow{\mathfrak{r}^*} H^4(\partial S^{(4)}, \widetilde{\mathcal{M}}_{\lambda, E})^{\mirf{N+N'}}\\ \xrightarrow{\mathfrak{R}^4_{\mu,\mu'}}
	I_4^\mathsf{S}(\mu+\mu',E) \oplus 	I_4^{\mathsf{S}}(\mu'(-2)+\mu(2),E) \label{eqn: third-integral-structure}
	\end{multline}
is a $\mathsf{k}$-dimensional subspace of $I_4^\mathsf{S}(\mu+\mu',E) \oplus I_4^{\mathsf{S}}(\mu'(-2)+\mu(2),E),$ where 
$\mathsf{k}$ now is the common dimension of $I_4^\mathsf{S}(\mu+\mu',E)$ and $I_4^{\mathsf{S}}(\mu'(-2)+\mu(2),E).$
The image of the composition of maps, denoted by $\mathfrak{I}^4(\mu,\mu',E)$, is of the form:
	\begin{equation*}
		\mathfrak{I}^4(\mu,\mu',E) \ = \ 
		\bigoplus_{\epsilon'}\bigoplus_{\tau'} \bigoplus_{\tau} 
		\{ (\phi_f, \phi_f  + T_{\text{Eis}}(\tau, \tau', \epsilon',E) \phi_f ) \,|\,\, \phi_f \in I_4^\mathsf{S}(\tau_f, \tau'_f, \epsilon',E)\}.
	\end{equation*}

\medskip
Define an $E$--linear isomorphism from the sum of induced representations to this image as follows:
$$		
\pi_{\mu+\mu'}^{\mathfrak{I}}: I_4^\mathsf{S}(\mu+\mu',E) \ \rightarrow \ \mathfrak{I}^4(\mu,\mu',E),
$$
given by 
$$
	\sum_{\tau,\tau', \epsilon'} \phi_{\tau_f,\tau'_f,\epsilon'} \ \mapsto \ 
	\sum_{\tau,\tau',\epsilon'} (\phi_{\tau_f,\tau'_f,\epsilon'}\,,\, \phi_{\tau_f,\tau'_f,\epsilon'} + 
		T_{\textup{Eis}}(\tau,\tau',\epsilon',E)\phi_{\tau_f,\tau'_f,\epsilon'}),
$$
and, similarly, define another $E$--linear isomorphism from the image to the sum of induced representations: 
$$
	\pi_{\mathfrak{I}}^{\mu'(-2)+\mu(2)}: \mathfrak{I}^4(\mu,\mu', E)  \ \rightarrow \ I_4^\mathsf{S}(\mu'(-2)+\mu(2),E)
$$
given by 
$$
\sum_{\tau,\tau',\epsilon'} (\phi_{\tau_f,\tau'_f,\epsilon'}\,,\, \phi_{\tau_f,\tau'_f,\epsilon'} + 
T_{\textup{Eis}}(\tau,\tau',\epsilon',E)\phi_{\tau_f,\tau'_f,\epsilon'}) \ \mapsto \ 
\sum_{\tau',\tau,\epsilon'} T_{\textup{Eis}}(\tau,\tau',\epsilon',E)\phi_{\tau_f,\tau'_f,\epsilon'},
$$
where $\phi_{\tau_f,\tau'_f,\epsilon'} \in I_4^\mathsf{S}(\tau_f, \tau'_f, \epsilon', E).$ 

\medskip

For $A = E, E_\mathfrak{l}, \mathbb{C}$ define
\begin{equation}
	T_{\textup{Eis}}(\tau,\tau',\epsilon',A) := T_{\textup{Eis}}(\tau,\tau',\epsilon',E)\otimes A.
\end{equation}

\subsection{Another integral structure on induced space}
\label{section: another-integral-structure}
Using \eqref{eqn: third-integral-structure} define an $\mathcal{O}_E$--lattice of full rank in $\mathfrak{I}^4(\mu,\mu',E)$ as follows:
\begin{gather*}
	\mathfrak{I}^4(\mu,\mu',\mathcal{O}_E) := \textup{Im}\left( \tilde{H}^4(\partial S^{(4)},\widetilde{\mathcal{M}}_{\lambda, \mathcal{O}_E})^{\mirf{N+N'}}
	 \xrightarrow{\mathfrak{R}^4_{\mu,\mu'}} \mathfrak{I}^4(\mu,\mu',E) \right),\\
	\mathfrak{I}^4(\mu,\mu',A) := \mathfrak{I}^4(\mu,\mu',\mathcal{O}_E)\otimes_{\mathcal{O}_E} A \qquad\text{for }A=\mathcal{O}_\mathfrak{l}, E_\mathfrak{l}, \mathbb{C}.
\end{gather*}
The extension of $\pi_{\mu+\mu'}^\mathfrak{I}$ and $\pi_{\mathfrak{I}}^{\mu'(-2) + \mu(2)}$ to the $\mathfrak{l}$--adic completions will again be denoted by the same symbols.
The $E$--linear isomorphisms $\pi_{\mu+\mu'}^\mathfrak{I}$ and $\pi_{\mathfrak{I}}^{\mu'(-2)+\mu(2)}$ need not 
preserve the $\mathcal{O}_E$--lattices in either of the co-domains. Also, to obtain the main result for other critical values, one would also like to consider 
Tate twists. For an integer $m$ with $-1 \leq m <\frac{k'-k}{2}-1$ consider the ideals 
\begin{gather} 
	\begin{aligned}
		\{ x \in \mathcal{O}_E \ | \  x \cdot \pi_{\mu(m) + \mu'}^\mathfrak{I} \left(I_4^\mathsf{S}(\mu(m)+ \mu,\mathcal{O}_E)\right) \subset \mathfrak{I}^4(\mu(m), \mu,\mathcal{O}_E) \},\\
		\{ x \in \mathcal{O}_E \ |\ x \cdot \pi^{\mu'(-2) + \mu(2+m)}_\mathfrak{I} 
		\left(\mathfrak{I}^4(\mu(m), \mu' , \mathcal{O}_E) \right) \subset I_4^\mathsf{S}(\mu'(-2) + \mu(2+m), \mathcal{O}_E)\}, 
	\end{aligned}
\end{gather}
and define the union of their supports to be the set $\mathsf{S}_{\textup{Eis}}$ of primes which we would like to avoid in Eisenstein cohomology. 
By definition it follows that $\mathsf{S}_{\textup{Eis}} = \mathsf{S}_{\textup{Eis}}(\mu,\mu',m)$ 
depends only on the weights $\mu'$ and $\mu$ and the finitely many Tate twists and not on any of the isotypic components of the cohomology group of $M_P.$
The following lemma follows from the definition of $\mathsf{S}_{\textup{Eis}}$: 

\begin{lemma}
	If $\mathfrak{l} \not\in S_{\textup{Eis}}$ then
	\begin{align}
		\begin{aligned}
			\pi_{\mu' + \mu}^\mathfrak{I} \left(I_4^\mathsf{S}(\mu(m)+ \mu',\mathcal{O}_\mathfrak{l})\right) &\subseteq \mathfrak{I}^4(\mu(m), \mu',\mathcal{O}_\mathfrak{l}), \\
			\pi^{\mu'(-2) + \mu(2+m)}_\mathfrak{I} \left(\mathfrak{I}^4(\mu(m), \mu' , \mathcal{O}_\mathfrak{l}) \right) &\subseteq I_4^\mathsf{S}(\mu'(-2)+\mu(2+m), \mathcal{O}_\mathfrak{l}).
		\end{aligned}
	\end{align} 
\end{lemma}

\medskip
\subsection{Congruence of the Eisenstein operator} 
\label{section: congruence-of-the-eisenstein-operator}

\begin{theorem}
\label{thm: congruence-of-eisenstein-operators}
Let $\phi'_f \in I_4^{\,\mathsf{S}}(\sigma_f, \sigma'_f, \epsilon', E_\mathfrak{l})$ and 
$\phi_f'' \in I_4^{\,\mathsf{S}}(\sigma_f, \sigma''_f, \epsilon', E_{\mathfrak{l}})$. Assume that $\mathfrak{l} \not\in \mathsf{S}_{\textup{Eis}}.$ We have:
$$ 
\phi'_f \,\,\equiv \phi_f'' \pmod{\mathfrak{l}^n} \ \implies \ 
T_{\textup{Eis}}(\sigma, \sigma', \epsilon', E_\mathfrak{l}) \phi'_f \,\,\equiv T_{\textup{Eis}}(\sigma, \sigma'', \epsilon',  E_\mathfrak{l}) \phi''_f \pmod{\mathfrak{l}^n}.
$$
\end{theorem}

\begin{proof}
Suppose $B \subset V$ and $B' \subset V'$ are $\mathcal{O}_\mathfrak{l}$-lattices inside vector spaces over $E_\mathfrak{l}$. Let $T: V \rightarrow V'$ 
be a morphism of vector spaces such that $T(B) \subset B'$ then for $x , y \in V$ one has $x \equiv y \pmod{\mathfrak{l}^n}$, which, by definition 
means $x - y \in \mathfrak{l}^n B,$ implies $T(x) \equiv T(y) \pmod{\mathfrak{l}^n},$ or that $ T(x) - T(y) \in \mathfrak{l}^n B' .$
Since $\mathfrak{l} \not\in S_{\text{Eis}}$ one has
$\phi'_f - \phi''_f  \in \mathfrak{l}^n \left( {I}_4^\mathsf{S}(\mu + \mu', \epsilon', \mathcal{O}_\mathfrak{l}) \right).$ 
From which one gets
$\pi_{\mu + \mu'}^\mathfrak{I} \left( \phi'_f\right) - \pi_{\mu + \mu'}^\mathfrak{I} \left( \phi''_f\right) \in \mathfrak{l}^n 
\left( {\mathfrak{I}}^4(\mu, \mu' , \epsilon', \mathcal{O}_\mathfrak{l}) \right)$. 
Hence, 
$$
\pi^{\mu'(-2) + \mu(2)}_\mathfrak{I} \left(\pi_{\mu + \mu'}^\mathfrak{I} \left( \phi'_f\right) \right) - 
\pi^{\mu'(-2) + \mu(2)}_\mathfrak{I} \left(\pi_{\mu + \mu'}^\mathfrak{I} \left( \phi''_f\right) \right)\in 
\mathfrak{l}^n  \left({I}_4^\mathsf{S}(\mu'(-2)+\mu(2), \epsilon',\mathcal{O}_\mathfrak{l}) \right),
$$
whence, 		
$$
T_{\text{Eis}}(\sigma, \sigma', \epsilon', E_\mathfrak{l})\phi_f' - T_{\text{Eis}}(\sigma, \sigma'', \epsilon', E_\mathfrak{l}) \phi_f'' \in 
\mathfrak{l}^n \left(  {I}_4^\mathsf{S}(\mu'(-2)+\mu(2), \epsilon', \mathcal{O}_\mathfrak{l}) \right).
$$
\end{proof}

Define vectors in the induced space
$I_4^\mathsf{S}(\mu+\mu', \epsilon', \mathcal{O}_\mathfrak{l})$ (see \eqref{eqn:I-sigma-sigma'} and the definitions immediately thereafter) which are 
supported only one double coset:
\begin{gather*}
	{^\circ}\phi'_f(\xi_f) = \begin{cases}
		{^\circ}v^{\epsilon'}_{\mu + \mu'}(h,h'^\rho) &\,\,\, \xi_f = \xi_f^{(N')}\\
		0 &\,\,\, \xi_f \neq \xi_f^{(N')}
	\end{cases} \quad \text{and}\quad 
	{^\circ}\tilde{\phi}'_f(\xi_f) = \begin{cases}
			{^\circ}v^{\epsilon'}_{\mu'(-2)+\mu(2)}(h'^\rho,h) &\,\,\, \xi_f = \xi_f^{(N)}\\
		0 &\,\,\, \xi_f \neq \xi_f^{(N)}.
	\end{cases}
\end{gather*}
Similarly, define vectors for the pair $(h,h'').$ 
Then, from definitions it follows that
${^\circ}\phi'_f \in I_4^\mathsf{S}(\sigma_f, \sigma'_f, \epsilon', E_\mathfrak{l}),$ 
${^\circ}\phi''_f \in I_4^\mathsf{S}(\sigma_f, \sigma''_f, \epsilon', E_\mathfrak{l})$
and that ${^\circ}\phi'_f \equiv  {^\circ}\phi''_f \pmod{\mathfrak{l}^n},$ i.e., 
$ {^\circ}\phi'_f -  {^\circ}\phi''_f \in\mathfrak{l}^nI_4^\mathsf{S}(\mu+\mu', \epsilon', \mathcal{O}_\mathfrak{l}).$ So we get the following

\begin{corollary} 
$T_{\textup{Eis}}(\sigma, \sigma', \epsilon', E_\mathfrak{l}) {^\circ}\phi'_f  \ \equiv \ 
T_{\textup{Eis}}(\sigma, \sigma'', \epsilon',  E_\mathfrak{l}) {^\circ}\phi''_f \pmod{\mathfrak{l}^n}.$
\end{corollary}

\medskip
\section{Computing the Eisenstein operator on some special vectors}

In this section the effect of $T_{\text{Eis}}(\sigma, \sigma', \epsilon', E_\mathfrak{l}) \otimes \mathbb{C}$ on 
${ ^\circ}\phi'_f = {^\circ}\phi'_f \otimes_{\iota} \mathbb{C}$ will be determined. 
To do so we shall introduce periods attached to cohomology classes by comparing them with certain canonically defined vectors at the transcendental level.

\subsection{Periods attached to the cohomology classes}
\label{section: periods-attached-to-the-cohomology-classes}
\subsubsection{For ${\rm GL}_2/\mathbb{Q}$}

For a dominant integral weight $\mu$ for $G_2$, recall from \ref{sec:rep-infty}, 
the relative Lie algebra cohomology 
$H^{1}(\mathfrak{g}_2, K_{2, \infty}, \mathbb{D}_{\mu} \otimes \mathcal{M}_{\mu, \mathbb{C}})$ is 
a two-dimensional space in which both the trivial and sign character for the action of ${\rm O}(2)/{\rm SO}(2)$ appear once; 
for such a character $\epsilon'$ of ${\rm O}(2)/{\rm SO}(2)$, the $\epsilon'$-isotypic component 
$H^{1}(\mathfrak{g}_2, K_{2, \infty}, \mathbb{D}_{\mu} \otimes \mathcal{M}_{\mu, \mathbb{C}})(\epsilon')$ is one-dimensional. 
Fix a basis $w^{\epsilon'}_{\infty}(\mu)$ for this one-dimensional space as in \cite[Sect.\,5.2.1]{harder-raghuram}. 

\smallskip
Now, in our situation of $h' \in S_{k'}(N',\chi')^{\textup{new}}$ and $h \in S_k(N,\chi)^{\textup{new}}$, 
let $\mathbf{h}$ and $\mathbf{h}'^\rho$ be $\C$-valued automorphic forms on $\GL_2(\mathbb{A})$ attached to $h$ and $h'^\rho,$ respectively. 
Let $\mathbf{h}_f$ and $\mathbf{h}'^\rho_f$ denote their restrictions to $\GL_2(\mathbb{A}_f)$ respectively.  We have isomorphisms 
\begin{gather*}
	\Phi_{(2)} :	
	H^{1}(\mathfrak{g}_2, K_{2, \infty}, \mathbb{D}_{\mu} \otimes \mathcal{M}_{\mu, \mathbb{C}})(\epsilon')
	\otimes \C\mathbf{h}_f \ \xrightarrow{\sim} \ 
	H^1_!(S^{(2)}_1(N), \mathcal{M}_{\mu,\mathbb{C}})(\epsilon'\times \sigma_f),\\
	\Phi'_{(2)} : H^{1}(\mathfrak{g}_2, K_{2, \infty}, \mathbb{D}_{\mu'}  \otimes \mathcal{M}_{\mu', \mathbb{C}})(\epsilon')
	\otimes \C\mathbf{h}'^\rho_f  \ \xrightarrow{\sim} \ 
	H^1_!(S^{(2)}_1(N'), \mathcal{M}_{\mu',\mathbb{C}})(\epsilon'\times \sigma'_f)
\end{gather*}
between one-dimensional spaces. But there are vectors which already the span one dimensional target-spaces, 
namely the the base change of ${^\circ}v^{\epsilon'}_{\mu}(h)$ and ${^\circ}v^{\epsilon'}_{\mu'}(h'^\rho)$ to $\mathbb{C}$ via the embedding $\iota$. 
Hence, there are two complex numbers
 $\Omega^{\epsilon'}(\Phi'_{(2)}, \mu', \sigma')$ and 
$\Omega^{\epsilon'}(\Phi_{(2)}, \mu, \sigma)$ such that
$$
\Phi_{(2)}(w_\infty^{\epsilon'}(\mu)\otimes \mathbf{h}_f) \ = \ 
\Omega^{\epsilon'}(\Phi_{(2)}, w_\infty^{\epsilon'}(\mu), \sigma) \,\,{^\circ}v_{\mu}^{\epsilon'}(h)
$$
and
$$
\Phi'_{(2)}(w_\infty^{\epsilon'}(\mu')\otimes \mathbf{h}'^\rho_f) \ = \  
\Omega^{\epsilon'}(\Phi'_{(2)}, w_\infty^{\epsilon'}(\mu'), \sigma') \,\,{^\circ}v_{\mu'}^{\epsilon'}(h'^\rho).
$$
Exactly as in \cite[Sect.\,5.2.4]{harder-raghuram}, there is an invariance with respect to even Tate twists; for the generators of the relative Lie algebra cohomology, 
one has $w_\infty^{\epsilon'}(\mu(2m)) = w_\infty^{\epsilon'}(\mu)$, and hence for the periods: 
\begin{equation} 
\Omega^{\epsilon'}(\Phi^?_{(2)}, w_\infty^{\epsilon'}(\mu^?(2m)), \sigma^?(-2m)) 
= \Omega^{\epsilon'}(\Phi^?_{(2)}, w_\infty^{\epsilon'}(\mu^?), \sigma^?). \label{eqn: period-under-tate-twist} 
\end{equation}
for $? \in \{ ', ''\}.$

\medskip
\subsubsection{For the Levi quotient $M_P$}
The preceding discussion on periods for $\GL_2$ naturally boot-straps via the K\"unneth theorem for periods for the cohomology classes for 
$M_P = G_2 \times G_2.$
Begin by fixing the basis element $w^{\epsilon'}_{\infty}(\mu + \mu')$ for
$$H^{1}(\mathfrak{g}_2, K_{2, \infty}, \mathbb{D}_{\mu} \otimes \mathcal{M}_{\mu, \mathbb{C}})(\epsilon') \otimes 
H^{1}(\mathfrak{g}_2, K_{2, \infty}, \mathbb{D}_{\mu'} \otimes \mathcal{M}_{\mu', \mathbb{C}})(\epsilon')$$ 
defined as $w^{\epsilon'}_{\infty}(\mu + \mu') := w_\infty^{\epsilon'}(\mu) \otimes w^{\epsilon'}_\infty(\mu').$ 
One has the isomorphism 
$\Phi_{M_P}$ from the one-dimensional space 
$$
\biggl(
H^{1}(\mathfrak{g}_2, K_{2, \infty}, \mathbb{D}_{\mu} \otimes \mathcal{M}_{\mu, \mathbb{C}})(\epsilon') \otimes 
H^{1}(\mathfrak{g}_2, K_{2, \infty}, \mathbb{D}_{\mu'} \otimes \mathcal{M}_{\mu', \mathbb{C}})(\epsilon')
\biggr) \otimes 
(\C \mathbf{h}_f \otimes \C \mathbf{h}'^\rho_f)
$$
to the one-dimensional space
$$
H_{!}^2 (S^{M_P}_{N \times N'}, \widetilde{\mathcal{M}}_{\mu + \mu', \mathbb{C}}) (\tilde{\epsilon}' \otimes {}\sigma_f  \otimes {}\sigma'_f). 
$$
Using the the base change of the element ${^\circ}v^{\epsilon'}_{\mu + \mu'}(h,h'^\rho)$ generating the target space, 
gives us a period $\Omega(\Phi_{M_P}, w^{\epsilon'}_{\infty}(\mu + \mu'), {}\sigma   \otimes {}\sigma' ).$ 
Analogously, there is a map $\widetilde{\Phi}_{M_P}$ and a basis element $w^{\epsilon'}_{\infty}(\mu'(-2)+ \mu(2)),$ for the weight $\mu'(-2) + \mu(2)$ and the representation $\sigma'(2)\otimes \sigma(-2)$ which gives the period $\Omega^{\tilde{\epsilon}'}(\widetilde{\Phi}_{M_P}, w^{\epsilon'}_{\infty}(\mu'(-2)+ \mu(2)),  {}\sigma'(2) \otimes {}\sigma(-2) ).$
Using \eqref{eqn: period-under-tate-twist} one has the following period relation: 

\begin{theorem}
\label{thm: periods-are-equal}
$$
\Omega^{\tilde{\epsilon}'}(\Phi_{M_P}, w^{\epsilon'}_{\infty}(\mu + \mu'), {}\sigma   \otimes {}\sigma') \ = \ 
\Omega^{\tilde{\epsilon}'}(\widetilde{\Phi}_{M_P}, w^{\epsilon'}_{\infty}(\mu'(-2)+ \mu(2)),  {}\sigma'(2) \otimes {}\sigma(-2) ).
$$
\end{theorem}

\medskip
\subsubsection{For the ambient group ${\rm GL}_4/\mathbb{Q}$}
The discussion above on periods for cohomology classes for the Levi $M_P$ naturally boot-straps via Delorme's lemma 
for periods for the cohomology classes for the ambient group $G_4 = \GL_4/\mathbb{Q}.$
By Delorme's lemma one has the isomorphism between the relative Lie algebra cohomology of a parabolically induced representation 
with that of the inducing representation:
\begin{multline*} 
H^{1}(\mathfrak{g}_2, K_{2, \infty}, \mathbb{D}_{\mu} \otimes \mathcal{M}_{\mu, \mathbb{C}})(\epsilon') 
\otimes H^{1}(\mathfrak{g}_2, K_{2, \infty}, \mathbb{D}_{\mu'} \otimes \mathcal{M}_{\mu', \mathbb{C}})(\epsilon')  \\
\cong \ 
H^4(\mathfrak{g}_4, K_{4,\infty}, \aInd(\D_\mu \otimes \D_{\mu'}) \otimes \mathcal{M}_{\lambda, \mathbb{C}})(\epsilon' \times \epsilon'),	
\end{multline*}
Here we assume that the pair of weights $(\mu, \mu')$ satisfies the conditions of the combinatorial lemma (see \cite[Lem.\,7.14]{harder-raghuram}); 
this then gives us $\lambda$ a weight on $\GL_4$ which is of the form 
$\lambda = w^{-1}\cdot(\mu+\mu')$ for a Kostant representative $w$ with $l(w) = 2 = \dim(U_P)/2$. 
The vector $w^{\epsilon'}_{\infty}(\mu + \mu')$ is now also to be thought of as a generator for the 
cohomology group on $G_4$ via Delorme's lemma. Similarly, for $w^{\epsilon'}_{\infty}(\mu'(-2) + \mu(2))$.

At the finite places fix vectors in the one-dimensional space of invariants under $\mirf{N+N'}$ of induced representations 
which are supported only on one double-coset: 
$$
\psi'_f \in \aInd_{P(\mathbb{A}_f)}^{G_4(\mathbb{A}_f)}(\sigma_f \otimes \sigma'_f)^{\mirf{N+N'}}, \qquad 
\tilde{\psi}'_f \in \aInd_{P(\mathbb{A}_f)}^{G_4(\mathbb{A}_f)}(\sigma'_f(2) \otimes \sigma_f(-2))^{\mirf{N+N'}}
$$ 
such that
\begin{gather*}
	\psi'_f(\xi_f) = \begin{cases}
		\mathbf{h}_f \otimes \mathbf{h}'^\rho_f &\,\,\, \xi_f = \xi_f^{(N')}\\
		0 &\,\,\, \xi_f \neq \xi_f^{(N')}, 
	\end{cases} \quad \text{and}\quad 
	\tilde{\psi}'_f(\xi_f) = \begin{cases}
		\mathbf{h}'^\rho_f(2) \otimes \mathbf{h}_f(-2) &\,\,\, \xi_f = \xi_f^{(N)}\\
		0 &\,\,\, \xi_f \neq \xi_f^{(N)}.
	\end{cases}
\end{gather*}
We have an isomorphism $\Phi_{(4)}$ between the one-dimensional space
$$
H^4(\mathfrak{g}_4, K_{4,\infty}, \aInd(\D_\mu \otimes \D_{\mu'}) \otimes \mathcal{M}_{\lambda, \mathbb{C}})(\epsilon' \times \epsilon')
\otimes 
\aInd_{P(\mathbb{A}_f)}^{G_4(\mathbb{A}_f)}(\sigma_f \otimes \sigma'_f)^{\mirf{N+N'}}
$$
and the space $ I_4^\mathsf{S}(\sigma_f, \sigma'_f, \epsilon', \mathbb{C})$ (see Sect.\,\ref{sec:int-on-ind-space}) giving a 
period construction via comparison of chosen basis elements:
\begin{equation}
\label{eqn: period-relation-induced-vectors}
\Phi_{(4)} (w^{\epsilon'}_{\infty}(\mu + \mu') \otimes \psi'_f) \ = \ 
\Omega^{\tilde{\epsilon}'}(\Phi_{M_P}, \, w^{\epsilon'}_{\infty}(\mu + \mu'), \, \sigma \otimes \sigma')
\ {^\circ} \phi'_f. 
\end{equation}
Similarly, we have a map $\widetilde{\Phi}_{(4)}$ such that 
\begin{equation*}
	\widetilde{\Phi}_{(4)} (w^{\epsilon'}_{\infty}(\mu'(-2) + \mu(2)) \otimes \tilde{\psi}'_f) \ = \ 
	\Omega^{\tilde{\epsilon}'}(\widetilde{\Phi}_{M_P}, \, w^{\epsilon'}_{\infty}(\mu'(-2) + \mu(2)), \, \sigma(2) \otimes \sigma'(-2))
	\ {^\circ} \tilde{\phi}'_f.
\end{equation*}

\medskip
\subsection{The standard intertwining operator on the special vectors}

The reader is referred to \cite[Sect.\,6.3.3]{harder-raghuram} for the definition and notations for the standard intertwining operator.

\subsubsection{At the infinite place}
Recall that we have assumed $(\mu,\mu')$ satisfies the conditions of \cite[Lem.\,7.14]{harder-raghuram}; in particular, the values of 
$L(s, \sigma_\infty \times \sigma'^\sfv_\infty)$ are finite at $s = -1$ and $s = -2.$
Define an operator between induced representations:
$$T_{\textup{loc}}(\sigma_\infty \otimes \sigma'_\infty): 
\aInd_{P_\infty}^{G_{4,\infty}}(\D_\mu \otimes \D_{\mu'}) \ \rightarrow \ \aInd_{P_\infty}^{G_{4,\infty}}(\D_{\mu'}(2) \otimes \D_{\mu}(-2))
$$
such that the map it induces at the level of the relative Lie algebra cohomology is pinned down by:
$$
T_{\textup{\text{loc}}}(\sigma_\infty \otimes \sigma'_\infty)^\bullet (w^{\epsilon'}_{\infty}(\mu + \mu')) \ = \ w^{\epsilon'}_{\infty}(\mu'(-2) + \mu(2)).
$$
On the other other, there is the standard intertwining operator 
$$
T_{\rm st}(-2, \sigma_\infty \otimes \sigma'_\infty): 
\aInd_{P_\infty}^{G_\infty}(\D_\mu \otimes \D_{\mu'}) \ \rightarrow \ \aInd_{P_\infty}^{G_\infty}(\D_{\mu'}(2) \otimes \D_{\mu}(-2)).
$$
The operator $T_{\textup{\text{loc}}}(\sigma_\infty \otimes \sigma'_\infty)^\bullet$ 
and the map induced at the level of cohomology by the standard intertwining operator are equal up to a scalar multiple. 
From \cite[Thm.\,7.25]{harder-raghuram}, there exists a $c'_\infty \in \Q^\times$ such that
\begin{equation*}
	T_{\text{st}}(-2, \sigma_\infty \otimes \sigma'_\infty)^\bullet (\epsilon') \ = \ 
	c'_\infty \frac{L(-2, \sigma_\infty \times \sigma'^\sfv_\infty)}{L(-1, \sigma_\infty \times \sigma'^\sfv_\infty)} 
	T_{\textup{\text{loc}}}(\sigma_\infty \otimes \sigma'_\infty)^\bullet.
\end{equation*}
Hence
\begin{equation}
\label{eqn:T-st-on-basis}
	T_{\text{st}}(-2, \sigma_\infty \otimes \sigma'_\infty)^\bullet (w^{\epsilon'}_{\infty}(\mu + \mu')) \ = \ 
	c'_\infty 
	\frac{L(-2, \sigma_\infty \times \sigma'^\sfv_\infty)}{L(-1, \sigma_\infty \times \sigma'^\sfv_\infty)} 
	w^{\epsilon'}_{\infty}(\mu'(-2) + \mu(2)).
\end{equation}
Note that $L(s, \sigma_\infty \times \sigma_\infty'^\sfv)$ defined in \textit{loc.cit.} 
is a nonzero constant multiple of $L_\infty(s, h \times h')$ defined in \ref{section: classical-rankin-selberg}, but if we take ratios of critical values, we get equality: 
$$\frac{L(s, \sigma_\infty \times \sigma_\infty'^\sfv)}{L(s+1, \sigma_\infty \times \sigma_\infty'^\sfv)} = \frac{L_\infty(s+k'-1, h \times h')}{L_\infty(s+k', h \times h')}.$$

\subsubsection{At the finite places}

Let $\mathsf{S}_f$ denote the set of all finite places where either $\sigma_f$ or $\sigma'_f$ is ramified; it is the support of the integer $NN'.$ Let 
$\mathsf{S}$ denote $\mathsf{S}_f$ together with the archimedean place. 
We will now compute the effect of the standard intertwining operator:
$$
T_{\rm st}(-2, \sigma_f \otimes \sigma_f): 
\aInd_{P(\A_f)}^{G(\A_f)}(\sigma_f \otimes \sigma'_f) \ \rightarrow \ 
\aInd_{P(\A_f)}^{G(\A_f)}( \sigma'_f(2) \otimes \sigma_f(-2))
$$
on the vector $\psi'_f$. By multiplicity-one for the invariants under the mirahoric subgroup $\mirf{N+N'}$, the operator maps  $\psi'_f$ to a multiple of $\tilde{\psi}'_f$.

\begin{theorem}
\label{thm:ratios-of-l-values-finite}
There exists a nonzero constant $c'_{\mathsf{S}_f} \in E$, such that 
	\[ T_{\text{st}}(-2, \sigma_f \otimes \sigma'_f) {\psi'}_f \ = \ 
	c'_{\mathsf{S}_f}\frac{L^\mathsf{S}(-2, \sigma \times \sigma'^\sfv)}{L^\mathsf{S}(-1, \sigma \times \sigma'^\sfv)}\, \tilde{\psi}_f'.\] 
\end{theorem}
\begin{proof}
It is enough to compute the value
	$\left(T_{\text{st}}(2, \sigma'_f \otimes \sigma_f) {\psi'}_f\right)(\xi_f^{(N')}).$ Going through the definitions, there is a $c \in \mathbb{C}$ such that 
$T_{\text{st}}(-2, \sigma_f \otimes \sigma'_f) ({\psi'}_f(\xi_f^{(N')})) = c \,\,\mathbf{h}'^\rho_f(2) \otimes \mathbf{h}_f(-2).$
	The scalar $c$ can be determined by evalutaing at $\underline{\mathbf{1}} \in M_P(\mathbb{A}_f)$ as $(\mathbf{h}'^\rho_f(2) \otimes \mathbf{h}_f(-2))(\underline{\mathbf{1}}) = 1 \in \mathbb{C}.$
	At the unramified places $p \notin \mathsf{S}_f$ this is exactly the calculation of Langlands (known as the Gindikin-Karplevic formula) 
	that the constant is the ratio of local $L$-values; see Langlands \cite{langlands}. 
	At the finitely many ramified places we get scalars $c'_{\mathsf{S}_f} = \prod_{p \in \mathsf{S}_f} c'_p$; these  
	local constants are in $E$ follows from the main result in \cite{raghuram-cjm}. 
\end{proof}

When the levels $N$ and $N'$ of $h$ and $h'$ are square-free and coprime to each other, the local constants are explicitly 
calculated in Sect.\,\ref{section: local-calculation}, 
where it is shown that $c'_{\mathsf{S}_f}$ is exactly the product of ratios of the local $L$-values. {\it One hopes that this is true in all generality.}

\medskip
\subsubsection{At a global level}

Recall once again that we have assumed $(\mu,\mu')$ satisfies the conditions of \cite[Lem.\,7.14]{harder-raghuram}; in particular, 
$s = -1$ and $s = -2$ are critical points for $L(s, \sigma \times \sigma'^\sfv).$ Furthermore, we now assume that 
the pair $(\mu,\mu')$ is on the right of the unitary axis guaranteeing holomorphy of an Eisenstein series; see \cite[Thm.\,6.4]{harder-raghuram}.
The consequence of these conditions for the classical Rankin--Selberg $L$-functions were discussed in \ref{sec:set-up-eis-coh}. Also, recall $\iota: \hat{\bar{E}}_\mathfrak{l} \cong \mathbb{C}$ is an embedding fixed in the beginning.

\begin{theorem}
\label{thm: ratios-of-l-values}
	Under $T_{\textup{Eis}}(\sigma, \sigma',\epsilon',E_\mathfrak{l}) \otimes_{\iota}{\mathbb{C}}$  the image of ${^\circ}\phi'_f$ is
	\[ (T_{\textup{Eis}}(\sigma, \sigma',\epsilon',E_\mathfrak{l})\otimes_{\iota} \mathbb{C}) \,\,{^\circ}\phi'_f \ = \ 
	c'_\infty c'_{\mathsf{S}_f}
	\frac{L^{\mathsf{S}_f}(-2, \sigma \times \sigma'^\sfv)}
	{L^{\mathsf{S}_f}(-1, \sigma \times \sigma'^\sfv)} {^\circ}\tilde{\phi}_f'.\]
\end{theorem}

\begin{proof} 
The map 
$
T_{\textup{Eis}}(\sigma, \sigma', \epsilon', E)\otimes_{\iota, E} \mathbb{C} : 
I_4^\mathsf{S}(\sigma_f, \sigma'_f, \epsilon', \mathbb{C}) \ \to \  I_4^\mathsf{S}(\sigma'_f(2), \sigma_f(-2), \epsilon', \mathbb{C})
$
after using the isomorphisms $\Phi_{(4)}$ and $\widetilde{\Phi}_{(4)}$ is the same as the map 
$T_{\text{st}}(-2, \sigma_\infty\otimes \sigma_\infty)^\bullet \otimes T_{\textup{st}}(-2, \sigma_f \otimes \sigma'_f).$ 
For convenience of notation, put $\Omega'$ and $\tilde{\Omega}'$ for the periods 
$\Omega^{\tilde{\epsilon}'}(\Phi_{M_P}, \, w^{\epsilon'}_{\infty}(\mu + \mu'), \, \sigma \otimes \sigma')$ and 
$\Omega^{\tilde{\epsilon}'}(\widetilde{\Phi}_{M_P}, w^{\epsilon'}_{\infty}(\mu'(-2) + \mu(2)), \sigma'(2) \otimes \sigma(-2))$, respectively. 
Also, put $T_{\text{st}}$ for $T_{\text{st}}(-2, \sigma_\infty\otimes \sigma_\infty)^\bullet \otimes T_{\textup{st}}(-2, \sigma_f \otimes \sigma'_f).$ Then 
\begin{multline*}
		(T_{\text{Eis}}(\sigma, \sigma', \epsilon',E_\mathfrak{l})\otimes_\iota\mathbb{C}) {^\circ}\phi'_f = \widetilde{\Phi}_{(4)} \circ T_{\text{st}} \circ \Phi_{(4)}^{-1}( {^\circ}\phi'_f) \\= \frac{1}{\Omega'} \widetilde{\Phi}_{(4)} \circ T_{\text{st}} \circ \Phi_{(4)}^{-1} (\Omega' \,{^\circ}\phi'_f ) = \frac{1}{\Omega'} \widetilde{\Phi}_{(4)} \circ T_{\text{st}} ( w_\infty^{\epsilon'}(\mu+\mu')\otimes \psi'_f),
\end{multline*}
which, due to Thm.\,\ref{thm:ratios-of-l-values-finite} and \eqref{eqn:T-st-on-basis}, is equal to
\begin{equation*}
	\frac{1}{\Omega'} c'_\infty c'_{\mathsf{S}_f} \frac{L^{\mathsf{S}_f}(-2, \sigma \times \sigma'^\sfv)}{L^{\mathsf{S}_f}(-1, \sigma \times \sigma'^\sfv) } \widetilde{\Phi}_{(4)}(w^{\epsilon'}_{\infty}(\mu'(-2)+ \mu(2)) \otimes \tilde{\psi}'_f) 
	= \frac{\tilde{\Omega}'}{\Omega'} c'_\infty c'_{\mathsf{S}_f} \frac{L^{\mathsf{S}_f}(-2, \sigma \times \sigma'^\sfv)}{L^{\mathsf{S}_f}(-1, \sigma \times \sigma'^\sfv) } {^\circ}\tilde{\phi}'_f.
\end{equation*}
Hence the theorem because the periods $\Omega'$ and $\tilde{\Omega}'$ are equal by Thm.\,\ref{thm: periods-are-equal}. 
\end{proof}

\bigskip
\section{The main theorems on congruences for the ratios of $L$-values}

\medskip
\subsection{Summary of notations} 
\label{sec:notations-summary}
We have primitive cusp forms $h', h'' \in S_{k'}(N',\chi')^{\textup{new}}$ and $h \in S_k(N,\chi)^{\textup{new}};$  
highest weights $\mu = (k-2,0)$ and $\mu' = (k'-2,0)$ assumed to be regular, i.e., $k',k>2$; a number field $E$ which is Galois over $\Q$ containing all the 
Fourier coefficients of $h, h', h''$; 
Hecke modules $\sigma \in \text{Coh}_!(G_2, \mu)$ and $\sigma' , \sigma'' \in \text{Coh}_!(G_2,\mu')$ such that 
for an embedding $\iota : E \to E_\mathfrak{l} \cong \C$ one has: 
$\sigma \cong \Pi(h)|\cdot|^{-k/2+1},$ 
$\sigma' \cong \Pi(h'^\rho)|\cdot|^{-k'/2+1},$ and 
$\sigma'' \cong \Pi(h''^\rho)|\cdot|^{-k'/2+1},$ where 
$\Pi(h)$ (resp., $\Pi(h'^\rho)$, $\Pi(h'^\rho)$) is the unitary cuspidal automorphic representation attached to $h$ (resp., $h'^\rho$, $h''^\rho$).
The pair $(\mu,\mu')$ is such that $s = -1$ and $s = -2$ are critical points for $L(s, \sigma \times \sigma'^\sfv)$ 
(\cite[Lem.\,7.14]{harder-raghuram}) and is on the right of the unitary axis (\cite[Sect.\,6.3.6]{harder-raghuram}).
Recall, the set $\mathsf{S}_N$ consists of all the prime ideals of $\mathcal{O}_E$ which divide $6N$, and the set $\mathsf{S}_{k}$ 
contains all the prime ideals of $ \mathfrak{p} \subset \mathcal{O}_E$ such that $p \leq k,$ where $p$ is the rational prime lying below $\mathfrak{p}.$

\medskip
\subsection{The main results on the right of the unitary axis}

The first theorem on congruences is stated in the context of \cite{harder-raghuram}. 

\begin{theorem}
\label{thm: main-theorem-on-the-ratios-of-lvalues-0}
Let notations be as in Sect.\,\ref{sec:notations-summary}. 
Suppose for a prime ideal $\mathfrak{l}$ in $E$ outside of 
$\mathsf{S}_{k'} \cup \mathsf{S}_{N'} \cup \mathsf{S}_{\textup{Eis}} \cup \mathsf{S}_{c'_\infty},$ one has  
$h' \equiv h'' \pmod{\mathfrak{l}^n}$ then there exist nonzero constant $c'_{\mathsf{S}_f}, c''_{\mathsf{S}_f}, \in E$ such that 
\begin{equation}
c'_{\mathsf{S}_f}\frac{L^{\mathsf{S}_f}(-2, \sigma \times \sigma'^\sfv )}{L^{\mathsf{S}_f}(-1, \sigma \times \sigma'^\sfv )} \ \equiv \ 
c''_{\mathsf{S}_f}\frac{L^{\mathsf{S}_f}(-2, \sigma \times \sigma''^\sfv )}{L^{\mathsf{S}_f}(-1, \sigma \times \sigma''^\sfv )}  
\pmod{\mathfrak{l}^n}.
\label{eqn: main-theorem-on-the-ratios-of-lvalues-0}
\end{equation}
\end{theorem}

\begin{proof}
From Thm.\,\ref{thm: ratios-of-l-values} one has
$$
	(T_{\text{Eis}}(\sigma, \sigma', \epsilon', E_\mathfrak{l})\otimes_{\iota}\mathbb{C}) {^\circ}\phi_f' 
	= c_\infty' c'_{\mathsf{S}_f}
	\frac{L^{\mathsf{S}_f}(-2, \sigma \times \sigma'^\sfv )}
	{L^{\mathsf{S}_f}(-1, \sigma \times \sigma'^\sfv )} 
	{^\circ}\tilde{\phi}'_f. 
$$
Similarly, from Thm.\,\ref{thm: ratios-of-l-values} for the pair $(\sigma, \sigma'')$ one has
$$
	(T_{\text{Eis}}(\sigma, \sigma'', \epsilon', E_\mathfrak{l})\otimes_{\iota}\mathbb{C})  {^\circ}\phi_f'' 
	= c_\infty'' c''_{\mathsf{S}_f}
	\frac{L^{\mathsf{S}_f}(-2, \sigma \times \sigma''^\sfv )}
	{L^{\mathsf{S}_f}(-1, \sigma \times \sigma''^\sfv )} {^\circ}\tilde{\phi}''_f.
$$
Note that $c'_\infty =c''_\infty$ because they depend only on the representations at infinity and $\sigma'_\infty = \sigma''_\infty$. 
Applying Thm.\,\ref{thm: congruence-of-eisenstein-operators} for the vectors ${^\circ}\phi'_f$ and ${^\circ}\phi''_f$ and then base-changing to $\mathbb{C}$ one gets
$$
	(T_{\text{Eis}}(\sigma, \sigma', \epsilon',E_\mathfrak{l})\otimes_{\iota} \mathbb{C}\,\,{^\circ}\phi_f')(\xi_f^{(N)}) 
	\equiv (T_{\text{Eis}}(\sigma, \sigma'', \epsilon',E_\mathfrak{l})\otimes_{\iota} \mathbb{C}\,\,{^\circ}\phi_f'')(\xi_f^{(N)}) \pmod{\mathfrak{l}^n},
$$
where $x \equiv y \pmod{\mathfrak{l}^n}$ means that $x - y 
\in \mathfrak{l}^n \left( H_!^2(S^{M_P}_{N' \times N}, \widetilde{\mathcal{M}}_{\mu'(-2)+\mu(2), \mathcal{O}_\mathfrak{l}})\otimes_\iota \mathbb{C} \right).$
By Thm.\,\ref{thm: ratios-of-l-values} one has
$$
c_\infty' c'_{\mathsf{S}_f}\frac{L^{\mathsf{S}_f}(-2, \sigma \times \sigma'^\sfv )}
{L^{\mathsf{S}_f}(-1, \sigma \times \sigma'^\sfv )} {^\circ}\tilde{\phi}'_f(\xi_f^{({N})}) 
\equiv 
c'_\infty c''_{\mathsf{S}_f}\frac{L^{\mathsf{S}_f}(-2, \sigma\times \sigma''^\sfv )}
{L^{\mathsf{S}_f}(-1, \sigma \times \sigma''^\sfv )} {^\circ}\tilde{\phi}''_f(\xi_f^{({N})}) 
\pmod{\mathfrak{l}^n}. 
$$
Since ${^\circ}\tilde{\phi}'_f(\xi_f^{({N})}) \equiv {^\circ}\tilde{\phi}''_f(\xi_f^{({N})}) \pmod{\mathfrak{l}^n}$, one has 
$$
	\Biggl(c'_\infty c'_{\mathsf{S}_f}
	\frac{L^{\mathsf{S}_f}(-2, \sigma \times \sigma'^\sfv )}
	{L^{\mathsf{S}_f}(-1, \sigma \times \sigma'^\sfv )} - 
	c'_\infty c''_{\mathsf{S}_f}
	\frac{L^{\mathsf{S}_f}(-2, \sigma \times \sigma''^\sfv )}
	{L^{\mathsf{S}_f}(-1, \sigma \times \sigma''^\sfv )}\Biggr) 
	{^\circ}\tilde{\phi'_f}(\xi_f^{({N})}) \\ \equiv 0 \pmod{\mathfrak{l}^n}. 
$$	
But ${^\circ}\tilde{\phi}'_f(\xi_f^{({N})}) = {^\circ}v_{\mu'(-2)}^{\epsilon'}(h'^\rho) \otimes 
{^\circ}v^{\epsilon'}_{\mu(2)}(h) \not\equiv 0 \pmod{{}\mathfrak{l}^n}.$ 
Hence \eqref{eqn: main-theorem-on-the-ratios-of-lvalues-0} follows since $\mathfrak{l} \notin \mathsf{S}_{c'_\infty}$.
\end{proof}

Now, transcribe Thm.\,\ref{thm: main-theorem-on-the-ratios-of-lvalues-0} into the context of classical Rankin--Selberg $L$-functions while incorporating 
Tate-twists to the get the following

\begin{theorem}
\label{thm:reg-sem-poly}
	Let $h',h'' \in S_{k'}(N',\chi')^{\text{new}}$ and $h \in S_k(N,\chi)^{\text{new}}$ with $k',k > 2$ and $k'-k > 2.$  
	Suppose for a prime ideal $\mathfrak{l}$ of $E$ outside of 
$\mathsf{S}_{k'} \cup \mathsf{S}_{N'} \cup \mathsf{S}_{\textup{Eis}} \cup \mathsf{S}_{c'_\infty},$ one has  
$h' \equiv h'' \pmod{\mathfrak{l}^n}$, and suppose also that 
	the mod-$\mathfrak{l}$ Galois representations attached to $h' $ and $h''$ are irreducible. Then,
	 for an integer $m$ and $-1\leq m < \frac{k'-k}{2} - 1,$ one has the congruence:
	\begin{equation}
	\label{eqn:reg-sem-poly}
		c'_{\mathsf{S}_f}(m)\, \frac{L^{\mathsf{S}_f}(k'-m-3, h\times h' )}{L^{\mathsf{S}_f}(k'-m-2, h\times h' )} \ \equiv \ 
		c''_{\mathsf{S}_f}(m)\frac{L^{\mathsf{S}_f}(k'-m-3, h\times h'' )}{L^{\mathsf{S}_f}(k'-m-2, h\times h'')}  \pmod{\mathfrak{l}^n}, 
	\end{equation} 
	where $c'_{\mathsf{S}_f}(m) = \prod_{p |NN'} c'_p(m)$ and $c''_{\mathsf{S}_f}(m) = \prod_{p |NN'} c''_p(m)$ with $c'_p(m),c''_p(m) \in E.$	\label{thm: first-main-theorem-on-l-values}
\end{theorem}

\begin{proof}
Recall that integral cohomology groups
${H}^1_!(S^{(2)}_1(N), \mathcal{\widetilde{\mathcal{M}}}_{\mu,\mathcal{O}_\mathfrak{l}}) \cong 
{H}^1_!(S^{(2)}_1(N), \mathcal{\widetilde{\mathcal{M}}}_{\mu(m),\mathcal{O}_\mathfrak{l}})$ are identified. 
Now apply Thm.\,\ref{thm: main-theorem-on-the-ratios-of-lvalues-0} to the pair $(\mu(m), \mu')$ and 
use the dictionary between classical and automorphic $L$-functions.
\end{proof}

If we impose a restriction on the ramifications, then we can improve 
Thm.\,\ref{thm:reg-sem-poly} to get the following best possible result on ratios of critical values for completed $L$-functions. 
One hopes that this is also true without any restriction on ramification. 

\medskip
\begin{theorem} 
\label{thm: second-main-theorem-on-l-values}
Let the notations and assumptions be as in Thm.\,\ref{thm: first-main-theorem-on-l-values}. 
Assume furthermore that the levels $N$ and $N'$ are square-free and relatively prime. Then for $-1\leq m < \frac{k'-k}{2} - 1$ 
	\begin{equation}
	\label{eqn:best-possible-cong}
		\frac{L(k'-m-3, h\times h' )}{L(k'-m-3, h\times h' )} \ \equiv \  
		\frac{L(k'-m-2, h\times h'' )}{L(k'-m-2, h\times h'')} \pmod{\mathfrak{l}^n}.
	\end{equation} 	
\end{theorem}

\begin{proof}
Under the hypothesis on the levels $N$ and $N'$, for $p | NN'$, it is proved in Sect.\,\ref{section: local-calculation} that 
$$
		c'_p(m) \ = \ \frac{L_p(k'-3-m, h\times h' )}{L_p(k'-2-m, h \times h' )}.
$$
Similarly for $c''_p(m).$ Then the congruence in \eqref{eqn:reg-sem-poly} becomes \eqref{eqn:best-possible-cong}. 
\end{proof}

\medskip
\subsection{The left of the unitary axis}
\label{subsection: left-of-unitary-axis}
If the pair $(\mu, \mu')$ is on the left of the unitary axis, then we reverse the direction of the intertwining operator and consider the intertwining operator
$$
T_{\rm st}(s)|_{s = 2} : \aInd_P^G({}\sigma'(2) \times {}\sigma(-m-2)) \ \longrightarrow
 \ \aInd_P^G\left({}\sigma(-m) \times {}\sigma')\right). 
$$
Now we are in the right of the unitary axis for the pair $(\mu'(-2),\mu(m+2)).$ One can now define the $E$-linear isomorphisms analogous to
$\pi_{\mu+\mu'}^\mathfrak{I}$ and $\pi_{\mathfrak{I}}^{\mu'(-2) + \mu(2+m)},$ say $\pi^{\mathfrak{I}}_{\mu'(-2)+\mu(m+2)}$ and 
$\pi_{\mathfrak{I}}^{\mu(m) + \mu'}.$ Enlarge the set $\mathsf{S}_{\text{Eis}}$ and $\mathsf{S}_{c'_\infty}$ if necessary, 
which we shall denote again by $\mathsf{S}_{\text{Eis}}$ and $\mathsf{S}_{c'_\infty}$, respectively.
Assuming $\mathfrak{l} \not\in \mathsf{S}_{k'} \cup \mathsf{S}_{N'} \cup \mathsf{S}_{\text{Eis}} \cup \mathsf{S}_{c'_\infty},$ 
for an integer $m$ and $\frac{k'-k}{2} - 1 \leq m < k'-k-2$ one gets 
\begin{equation*}
	\tilde{c}'_{\mathsf{S}_f}(m)\dfrac{L^{\mathsf{S}_f}(2, \sigma' \times \sigma(-m)^\sfv)}
	{L^{\mathsf{S}_f}(3, \sigma' \times \sigma(-m)^\sfv)} - 
	\tilde{c}''_{\mathsf{S}_f}(m)\dfrac{L^{\mathsf{S}_f}(2, \sigma'' \times \sigma(-m)^\sfv)}
	{L^{\mathsf{S}_f}(3, \sigma'' \times \sigma(-m)^\sfv)} \in \mathfrak{l}^n
\end{equation*}
which is equivalent to 
\begin{equation*}
	\tilde{c}'_{\mathsf{S}_f}(m)\dfrac{L^{\mathsf{S}_f}(m+k+1, h'^\rho \times h^\rho)}
	{L^{\mathsf{S}_f}(m+k+2, h'^\rho \times h^\rho)} - 
	\tilde{c}''_{\mathsf{S}_f}(m)\dfrac{L^{\mathsf{S}_f}(m+k+1, h'^\rho \times h^\rho)}{L^{\mathsf{S}_f}(m+k+2, h'^\rho \times h^\rho)} \in \mathfrak{l}^n.
\end{equation*}
Yet we cannot conclude from \eqref{eqn: relation-between-ratios} that the ratios to the left of the line of the symmetry are congruent modulo $\mathfrak{l}^n$, 
because $\mathfrak{l} \not\in\mathsf{S}_{\text{Eis}} \cup \mathsf{S}_{c'_\infty}$ ensures that
\begin{equation}
	\tilde{c}'_{\mathsf{S}_f}(m)\dfrac{L^{\mathsf{S}_f}(m+k+1, h'^\rho \times h^\rho)}{L^{\mathsf{S}_f}(m+k+2, h'^\rho \times h^\rho)}, \quad
	\tilde{c}''_{\mathsf{S}_f}(m)\dfrac{L^{\mathsf{S}_f}(m+k+1, h''^\rho \times h^\rho)}{L^{\mathsf{S}_f}(m+k+2, h''^\rho \times h^\rho)} 
	\label{eqn: left-of-the-line-of-symmetry}
\end{equation}
are in $\mathcal{O}_\mathfrak{l},$ but they  need not be in $\mathcal{O}_\mathfrak{l}^\times.$ 
Therefore, if one further assumes that the quantities in \eqref{eqn: left-of-the-line-of-symmetry} 
are $\mathfrak{l}$-adic units then one gets, using the functional equation, for $k\leq m < k'-1$: 
\begin{equation*}
	c'_{\mathsf{S}_f}(m)\dfrac{L^{\mathsf{S}_f}(m, h \times h')}{L^{\mathsf{S}_f}(m+1, h \times h')}  \ \equiv \ 
	c'_{\mathsf{S}_f}(m)\dfrac{L^{\mathsf{S}_f}(m, h \times h'')}{L^{\mathsf{S}_f}(m+1, h \times h'')} \pmod{\mathfrak{l}^n}. 
\end{equation*}
As before, if furthermore the levels $N$ and $N'$ are square-free and relatively prime, then for $k\leq m <k'-1$, one has the congruence: 
\begin{equation*}
	\dfrac{L(m, h \times h')}{L(m+1, h \times h')} \ \equiv \  \dfrac{L(m, h \times h'')}{L(m+1, h\times h'')} \pmod{\mathfrak{l}^n}.
\end{equation*}

\medskip
\subsection{Varying the modular forms of lower weight}

For convenience, here we consider the twist $\mu(-m)$ instead of $\mu(m)$. 
For the integers $ 3 \leq m \leq k-k'+1 $, the two successive $L$-values considered below are critical. 
The pair $(\mu(-m), \mu')$ is on the right of the unitary axis only for integers $m$ and $\frac{k-k'}{2} + 1 < m \leq k-k'+1.$ 
Then going through the above proof with appropriately modified $\mathsf{S}_{\text{Eis}}$ and $\mathsf{S}_{c'_\infty}$ one gets: 

\begin{theorem}
\label{thm: varying-the-lower-weight}
	Let $h \in S_k(N,\chi)$ and $h',h'' \in S_{k'}(N',\chi')$  with $k',k > 2$ and $k-k' \geq 2.$  
	Assume for $\mathfrak{l} \not\in \mathsf{S}_{k'} \cup \mathsf{S}_{N'} \cup \mathsf{S}_{\textup{Eis}} \cup \mathsf{S}_{c'_\infty}$ that 
	$h' \equiv h'' \pmod{\mathfrak{l}^n}$,  
	and the mod-$\mathfrak{l}$ Galois representations attached to $h' $ and $h''$ both irreducible, 
	then for an integer $m$ with $\frac{k-k'}{2} + 1 < m \leq k-k'+1,$ one has
	\begin{equation*}
		c'_{\mathsf{S}_f}(m)\frac{L^{\mathsf{S}_f}(k'+m-3, h\times h' )}{L^{\mathsf{S}_f}(k'+m-2, h\times h' )} \ \equiv \ 
		 c''_{\mathsf{S}_f}(m)\frac{L^{\mathsf{S}_f}(k'+m-3, h\times h'' )}{L^{\mathsf{S}_f}(k+m-2, h\times h'')} \pmod{\mathfrak{l}^n},
	\end{equation*} 
	where $c'_{\mathsf{S}_f}(m) = \prod_{p |NN'} c'_p(m)$ and $c''_{\mathsf{S}_f}(m) = \prod_{p |NN'} c''_p(m)$ with $c'_p(m),c''_p(m) \in E.$
	If the levels $N$ and $N'$ are square-free and relatively prime then
	\begin{equation*}
		\frac{L(k'+m-3, h\times h' )}{L(k'+m-2, h\times h' )}  \ \equiv \  \frac{L(k'+m-3, h\times h'' )}{L(k'+m-2, h\times h'')} \pmod{\mathfrak{l}^n}.
	\end{equation*}
\end{theorem}
The reader should bear in mind that $c'_{\mathsf{S}_f}(m), c''_{\mathsf{S}_f}(m)$ in Thm.\,\ref{thm: first-main-theorem-on-l-values} and 
Thm.\,\ref{thm: varying-the-lower-weight} are different. 
When one is on the left of the unitary axis same remarks as in Sect.\,\ref{subsection: left-of-unitary-axis} applies here.

\bigskip
\subsection{A non-example}
In the companion paper \cite[Section 3.4]{narayanan-raghuram} there is a non-example, i.e., a specific situation when 
the ratios of the Rankin--Selberg $L$-values at certain critical values are not congruent. Let $h \in S_{26}({\rm SL}_2(\mathbb{Z}))$ and $h', h'' \in S_{13}(\Gamma_1(3)).$
Fix $h'$ to be form with rational Fourier coefficients and $h''$ to be the newform whose coefficients lie in 
$K:= \mathbb{Q}(\sqrt{-8424})$ an imaginary quadratic extension. Let $\mathfrak{l}$ be a prime ideal of $K$ lying above $13$. 
It happens that for all $n \in \mathbb{N}$ 
\begin{equation}
	a(n,h') \equiv a(n,h'') \! \pmod{\mathfrak{l}} \quad \textup{but} \quad 
	\frac{L(24, h \times h')}{L(25,h'\times h)} \not\equiv	\frac{L(24, h \times h'')}{L(25,h\times h'')} \! \pmod{\mathfrak{l}}.
\end{equation}
The levels are square-free and coprime to each other yet the congruence for the ratios of this particular $L$-values fail. There are two reasons our main theorem does not hold here. First the hypothesis $l > k'$ is violated as $l = k' = 13.$ The second being the hypothesis that the mod $l$ Galois representation 
$\varrho_{\Theta'}$ obatined from $h'$ is irreducible is \textit{not} satisfied here. There exists an Eisenstein series 
$E_{13} \in M_{13}(\Gamma_1(3))$ with $q$-expansion $ E_{13} = \frac{55601}{3} +  q - 4095 q^{2} +  q^{3} + 16773121 q^{4} \dots$ and that
\begin{equation*}
	E_{13} \equiv h' \equiv h'' \pmod{\mathfrak{l}},
\end{equation*}
which can be seen from Sturm's bound. It should be observed however that the ratios of $L$-values at other critical points are still congruent modulo $\mathfrak{l}$ 
even though it follows outside the purview of our main theorems of this paper.

\bigskip
\section{A local calculation}
\label{section: local-calculation}

As promised in the proof of Thm.\,\ref{thm: second-main-theorem-on-l-values}, in this section we 
compute the local constant $c'_p(m)$; see Thm.\,\ref{thm:local-calculation} below. 
Recall that $N$ and $N'$ are relatively prime square-free integers. 
Let $\sigma_p$ and $\sigma_p'$ denote the local representation at a prime $p|NN'$ obtained from the cusp forms 
$h \in S_k(N,\chi)^{\text{new}}$ and $h'^\rho \in S_{k'}(N',\chi')^{\text{new}}$ from the isomorphism in \eqref{eqn:choice-of-repns}. 
From Thm.\,\ref{thm: local-lemma-on-subgroup-of-the-levi} on double coset representatives the spaces of invariant vectors 
$$I_P^{G_4}(s, \sigma_p \otimes \sigma'_p)^{\mirp{n_p + n'_p}} \ \ {\rm and} \ \ 
I_P^{G_4}(-s, \sigma'_p \otimes \sigma_p)^{\mirp{n_p + n'_p}}$$ 
are both one-dimensional. 
Let $\phi'_p \in I_P^{G_4}(s, \sigma_p \otimes \sigma'_p)$ 
(resp., $\tilde{\phi}'_p \in I_P^{G_4}(-s, \sigma'_p \otimes \sigma_p)$) be vectors which span the one-dimensional spaces. 
Then $\phi'_p$ (resp., $\tilde{\phi}'_p$) is supported only on the double coset 
$P(\mathbb{Q}_p) \xi_p^{(n'_p)} \mirp{n_p + n_p'}$ (resp., $P(\mathbb{Q}_p) \xi_p^{(n_p)} \mirp{n_p + n_p'}$). 
Consider the standard intertwining operator
$T_{\text{st}}(s,\sigma_p \otimes \sigma'_p) : I_P^{G_4}(s, \sigma_p \otimes \sigma'_p) \rightarrow 
I_P^{G_4}(-s, \sigma'_p \otimes \sigma_p)$ given by the integral 
$$
T_{\text{st}}(s, \sigma_p \otimes \sigma'_p)(\phi'_p)(g) \ = \ \int_{U_P(\mathbb{Q}_p)} \phi'_p(w_0^{-1}u g) du,
$$
where 
$w_0 = \left(\begin{smallmatrix}
	0 & 0 & 1 & 0\\ 0& 0& 0 & 1\\ 1& 0 & 0 & 0 \\ 0 & 1 & 0 & 0
\end{smallmatrix}\right).$
The integral converges when $s=-2$ (see, for example, \cite{harder-raghuram}). Since it is a map between two one-dimensional spaces 
there exists $c'_p \in \mathbb{C}$ such that 
$$
(T_{\text{st}}(-2, \sigma_p \otimes \sigma'_p)(\phi'_p) \ = \ c'_p \tilde{\phi}'_p.
$$
In the main result of this section, see Thm.\,\ref{thm:local-calculation} below, 
we evaluate the constant $c'_p$, where, without loss of generality, we take $\sigma_p \cong \textup{St}\otimes \chi_p$ 
is an unramified twist of the Steinberg representation and $\sigma'_p \cong \pi(\chi'_{1,p}, \chi'_{2,p})$ is an unramified principal series representation.

\subsection{Fixing canonical new vectors}

Let $\sigma_p \cong \text{St}\otimes \chi_p$ where $\chi_p$ is unramified, hence $n_p =1 $. Fix the new vector 
$v_p \in\text{St}\otimes \chi_p \subsetneq  I_{B_2}^{G_2}(|\cdot|^{1/2}\chi_p, |\cdot |^{-1/2}\chi_p)$ as follows: 
It is a map $v_p : G_2(\mathbb{Q}_p) \rightarrow \mathbb{C}$ such that 
\begin{equation}
\left(\sigma_p \left(\begin{smallmatrix} t_1 & * \\ & t_2 \end{smallmatrix}\right)\right) v_p)(\mathbf{1}_2) = \chi_p(t_1) \chi_p(t_2), \quad 
v_p(\mathbf{1}_2) = 1 \quad \text{ and } v_p(\mathbf{w}) = -1/p, 
\end{equation}
where $\mathbf{w} := \left(\begin{smallmatrix}  0 & -1 \\ 1 & 0 \end{smallmatrix}\right).$ See Schmidt \cite[Sect.\,2.1]{ralf-schmidt}. 
This normalization is done so that there is a \textit{canonical} isomorphism between the $G_2(\mathbb{Q}_p)$ representation generated by $\mathbf{h}_f|_{G_2(\mathbb{Q}_p)}$ and $\text{St}\otimes \chi_p.$ 

\smallskip

For $\sigma'_p \cong \pi(\chi'_{1,p}, \chi'_{2,p})$, an unramified principal series representation, as one has $n_p$ is $0$. 
The characters $\chi'_{1,p}$ and $ \chi'_{2,p}$ are unramified.
The normalized spherical vector of $\sigma'_p$ is a function $v'_p : G_2(\mathbb{Q}_p) \rightarrow \mathbb{C}$ such that
\begin{equation}
	\left(\sigma'_p \left(\begin{smallmatrix} t_1 & * \\ & t_2 \end{smallmatrix}\right)\right) v'_p)(1) = |t_1 t_2^{-1}|^{1/2} \chi'_{1,p}(t_1) \chi'_{2,p}(t_2).
\end{equation}
Again there is a \textit{canonical} isomorphism between the $G_2(\mathbb{Q}_p)$ representation generated by $\mathbf{h}'^\rho_f|_{G_2(\mathbb{Q}_p)}$ and $\pi(\chi_{1,p}', \chi_{2,p}').$

One has the following relations with Fourier coefficients: 
\begin{equation} \label{eqn: relation-with-fourier-coefficient}
\chi_p(p) = p^{-1/2}a(p,h), \quad \chi'_{1,p}(p) + \chi'_{2,p}(p) = p^{-1/2} a(p,h'^\rho) , \quad \chi'_{1,p}(p)\chi'_{2,p}(p) = p^{k'-2}\chi'^{-1}(p).
\end{equation}

\subsection{Fixing vectors $\phi'_p$ and $\tilde{\phi}'_p$}

 Given $f_p \in I_P^{G_4}(s, \sigma_p \otimes \sigma'_p)$, since $f_p(g) \in V_{\sigma_p} \otimes V_{\sigma'_p},$ and  
the local representations being subrepresentations of induced representations, one can evaluate $f_p(g)$ at an element of 
$\underline{m}\in M_P(\mathbb{Q}_p)$ to get a complex number. Also,
$f_p(g)(\underline{m}) = (\underline{m}\cdot f_p(g))(\underline{\mathbf{1}}) = f_p(\underline{m}\cdot g)(\underline{\mathbf{1}}).$
So one can identify the induced vector $f_p(g)$ with the complex number $f_p(g)(\underline{\mathbf{1}})$.
Next, since $n_p = 1, n'_p = 0$ and $n_p + n'_p = 1$ the coset representative $\xi_p^{(n_p)}$ in Sect.\,\ref{section: double-coset-representatives} is an element of $\mirp{n_p + n'_p}.$ 
Hence $P(\mathbb{Q}_p) \xi_p^{(n_p)}\mirp{n_p +n'_p} = P(\mathbb{Q}_p) \mathbf{1}_4 \mirp{n_p +n'_p}.$ 
To make evaluations less cumbersome we take $\xi_p^{(n_p)} = \mathbf{1}_4$.  
Fix the vectors in $I_P^{G_4}(-2,\sigma_p \otimes \sigma'_p)$ and $I_P^{G_4}(2, \sigma'_p \otimes \sigma_p)$ by: 
\begin{equation}
	\phi'_p(\xi_p^{(n'_p)}) = v_p \otimes v'_p, \quad \text{and } \quad \tilde{\phi}'_p(\xi_p^{(n_p)}) = v'_p \otimes v_p,
\end{equation}
respectively. So, 
$\phi'_p(\xi_p^{(n'_p)})(\underline{\mathbf{1}}) = (v_p \otimes v'_p)(\underline{\mathbf{1}}) = 1$ and 
$\phi'_p(\xi_p^{(n_p)})(\mathbf{w}, \mathbf{1}) = v'_p(\mathbf{w}) v_p(\mathbf{1}) = -1/p,$ where $(\mathbf{w}, \mathbf{1}) \in M_P(\mathbb{Q}_p).$ 
Since $\tilde{\phi}'_p$ is normalized as $\tilde{\phi}'_p(\xi_p^{(n_p)})(\underline{\mathbf{1}}) = 1$, to determine the scalar $c'_p$, 
it is enough to evaluate the integral at $\underline{\mathbf{1}}$, i.e.,
\[c'_p = (T_{\text{st}}(-2, \sigma_p \otimes \sigma'_p)\phi'_p)(\xi_p^{(n_p)})(\underline{\mathbf{1}}) = (T_{\text{st}}(-2, \sigma_p \otimes \sigma'_p)\phi'_p)(\mathbf{1}_4)(\underline{\mathbf{1}}).\]
This constant $c'_p$ will be shown to be exactly the ratio of the local $L$-values: 

\begin{theorem}
\label{thm:local-calculation}
	With the above assumptions on the local components for $p|NN'$
	$$
		c'_p = \frac{L_p(k'-3, h\times h')}{L_p(k'-2, h \times h')}.
	$$
\end{theorem}

\medskip
\subsection{Certain formal integrals}
Fix a measure on $\mathbb{Q}_p$ by $\int_{\mathbb{Z}_p} dg = 1.$ Let $x_1, x_2, x_3, x_4 \in \mathbb{Q}_p^\times.$ Define
$$
	t(x_1) = \left(\begin{smallmatrix}
		1/x_1 & & & \\ & 1 & & \\ & & x_1 & \\ & & & 1
	\end{smallmatrix}\right),\ 
	t(x_2) = \left(\begin{smallmatrix}
		1 & & & \\ & 1/x_2 & & \\ & & x_2 & \\ & & & 1
	\end{smallmatrix}\right), \ 
	t(x_3) = \left(\begin{smallmatrix}
		1/x_3 & & & \\ & 1 & & \\ & & 1 & \\ & & & x_3
	\end{smallmatrix}\right), \ 
	t(x_4) = \left(\begin{smallmatrix}
		1 & & & \\ & 1/x_4 & & \\ & & 1 & \\ & & & x_4
	\end{smallmatrix}\right).
$$
For $\underline{m} = (m_1, m_2) \in M(\mathbb{Q}_p)$ let $\delta_P(\underline{m}) = |\det(m_1)|^2 |\det(m_2)|^{-2}.$ For any $x_i$'s as above and $ M \in \mathbb{Z}$ formally define the operators
\begin{equation}
\label{eqn: less-than-M-operator}
T_{< M} (x_i) \ := \ \int\limits_{\substack{{x_i \in \mathbb{Q}_p}\\ v_p(x_4) < M}} \delta_P(t(x_i))^{1/2} \,\,\,\underline{\sigma}_p'(t(x_i)) dx_i, \qquad 
T_{\geq M} (x_i) \ := \ \int\limits_{\substack{{x_i \in \mathbb{Q}_p} \\ v_p(x_i) \geq M}} dx_i, 
\end{equation}
where $\underline{\sigma}_p'(t(x_i)) = (\sigma_p \otimes \sigma'_p) (t(x_i)).$ 

\begin{lemma}
	One has
	\begin{equation}
		(T_{\geq 0}(x_1)\phi'_p)(\xi_p^{(n_p)}) (\underline{\mathbf{1}}) + \left(T_{< 0 }(x_1)\phi'_p\right)(\xi_p^{(n_p)})(\underline{\mathbf{1}}) = \frac{1-p\,\, \chi_p(p) \chi'_{1,p}(p^{-1})}{1-p^{2}\,\, \chi_p(p) \chi'_{1,p}(p^{-1})}.
	\end{equation}
	Similarly,
	\begin{align*}
		(T_{\geq 0}(x_2)\phi'_p)(\xi_p^{(n_p)}) (\underline{\mathbf{1}}) + \left(T_{< 0 }(x_2)\phi'_p\right)(\xi_p^{(n_p)})(\underline{\mathbf{1}}) = 
		\frac{1-p\,\, \chi_p(p) \chi'_{1,p}(p^{-1})}{1-p^{2}\,\, \chi_p(p) \chi'_{1,p}(p^{-1})},\\
		(T_{\geq 0}(x_3)\phi'_p)(\xi_p^{(n_p)}) (\underline{\mathbf{1}}) + \left(T_{< 0 }(x_3)\phi'_p\right)(\xi_p^{(n_p)})(\underline{\mathbf{1}}) = 
		\frac{1-p\,\, \chi_p(p) \chi'_{2,p}(p^{-1})}{1-p^{2}\,\, \chi_p(p) \chi'_{2,p}(p^{-1})},\\
		(T_{\geq 0}(x_4)\phi'_p)(\xi_p^{(n_p)}) (\underline{\mathbf{1}}) + \left(T_{< 0 }(x_4)\phi'_p\right)(\xi_p^{(n_p)})(\underline{\mathbf{1}}) = 
		\frac{1-p\,\, \chi_p(p) \chi'_{2,p}(p^{-1})}{1-p^{2}\,\, \chi_p(p) \chi'_{2,p}(p^{-1})}.
	\end{align*} \label{lemma: ratios-of-l-values}
\end{lemma}
\begin{proof}
We prove it only when $i=1$. Other cases are similar. 
$$
\phi'_p(\xi_p^{n_p}) (\underline{\mathbf{1}}) + \left(T_{< 0 }(x_1)\phi'_p)(\xi_p^{(n_p)})\right)(\underline{\mathbf{1}}) 
= 1 + \int\limits_{\substack{{x_i \in \mathbb{Q}_p}\\ v_p(x_1) < 0}} \delta_P(t(x_i))^{1/2} \,\,\left(\underline{\sigma}_p'(t(x_1)) \phi'_p\right)(\xi_p^{(n_p)})
(\underline{\mathbf{1}}) dx_1, 
$$	
the right hand side evaluates to
\begin{multline*}
1 + \sum_{M=1}^\infty  \int_{p^{-M}\mathbb{Z}_p^\times} \delta_P(t(x_i))^{1/2}\,\, \left(\underline{\sigma}_p'(t(x_1)) 
\phi'_p\right)(\xi_p^{(n_p)})(\underline{\mathbf{1}})dx_1 \\
= 1+ \sum_{M=1}^\infty \left(\frac{p-1}{p}\right) p^{2M} \chi_p(p^M) \chi'_{1,p}(p^{-M}) \phi'_p(\mathbf{1}_4)(\underline{\mathbf{1}}), 
\end{multline*}
further simplifying as
\begin{multline*}
1+ \left(\frac{p-1}{p}\right) \sum_{M=1}^\infty  (p^{2} \,\,\chi_p(p) \chi'_{1,p}(p^{-1}))^M  \\ 
= 1 + \left(\frac{p-1}{p}\right) \left(\frac{p^{2} \,\,\chi_p(p) \chi'_{1,p}(p^{-1})}{1-p^{2}\,\, \chi_p(p)\,\, \chi'_{1,p}(p^{-1})}\right) 
= \frac{1-p\,\,\chi_p(p) \chi'_{1,p}(p^{-1})}{1-p^{2} \,\,\chi_p(p) \,\,\chi'_{1,p}(p^{-1})}.
\end{multline*}
\end{proof}
The convergence is guaranteed here because we will be in the context of \cite{harder-raghuram}.

\subsubsection{Preliminaries on measure}
Fix the product measure $dx_1 dx_2 dx_3 dx_4$ on $U_P(\mathbb{Q}_p)$ normalized by ${\rm vol}(U_P(\mathbb{Z}_p)) = 1.$ 
Suppose $\underline{m} =(m_1,m_2) \in M_P(\mathbb{Q}_p)$ then for $f_p \in C^{\infty}(G_4(\mathbb{Q}_p))$
\begin{equation}
\int f_p\left(\underline{m} \left(\begin{smallmatrix}
		1 & & & \\
		0 & 1 & & \\
		x_1 & x_2 & 1 & \\
		x_3 & x_4 & 0 & 1
	\end{smallmatrix}\right) \underline{m}^{-1}\right) dx_1dx_2 dx_3 dx_4 = \delta_P(\underline{m})^{1/2} \int f_p\left(\left(\begin{smallmatrix}
		1 & & & \\
		0 & 1 & & \\
		x_1 & x_2 & 1 & \\
		x_3 & x_4 & 0 & 1
	\end{smallmatrix}\right)\right)dx_1dx_2 dx_3 dx_4, \label{eqn: conjuagation-by-levi}
\end{equation}
where $\delta_P(\underline{m}) = |\det(m_1)|^2 |\det(m_2)|^{-2}.$

\subsubsection{Some matrix identities}
\label{sec:matrix-identities}
Let us record some matrix identities in $\GL_4(\mathbb{Q}_p)$ which will be useful in Lem.\,\ref{lemma: main-lemma}.
$$
	\left(\begin{smallmatrix}
		\frac{1}{c} & 0 & 0 & 0 \\
		0 & 1 & 0 & 0 \\
		0 & 0 & 1 & 0 \\
		0 & 0 & 0 & c
	\end{smallmatrix}\right)^{-1} \left(\begin{smallmatrix}
		1 & 0 & 0 & \frac{1}{c} \\
		0 & 1 & 0 & 0 \\
		0 & 0 & 1 & 0 \\
		0 & 0 & 0 & 1
	\end{smallmatrix}\right)^{-1} \left(\begin{smallmatrix}
		1 & 0 & 0 & 0 \\
		0 & 1 & 0 & 0 \\
		a & b & 1 & 0 \\
		0 & d & 0 & 1
	\end{smallmatrix}\right)
	\left(\begin{smallmatrix}
		1 & 0 & 0 & \frac{1}{c} \\
		0 & 1 & 0 & 0 \\
		0 & 0 & 1 & 0 \\
		0 & 0 & 0 & 1
	\end{smallmatrix}\right)
	\left(\begin{smallmatrix}
		\frac{1}{c} & 0 & 0 & 0 \\
		0 & 1 & 0 & 0 \\
		0 & 0 & 1 & 0 \\
		0 & 0 & 0 & c
	\end{smallmatrix}\right) = 
	\left(\begin{smallmatrix}
		1 & -d & 0 & 0 \\
		0 & 1 & 0 & 0 \\
		\frac{a}{c} & b & 1 & a \\
		0 & \frac{d}{c} & 0 & 1
	\end{smallmatrix}\right). 
$$

$$
	\left(\begin{smallmatrix}
		1 & -d & 0 & 0 \\
		0 & 1 & 0 & 0 \\
		\frac{a}{c} & b & 1 & a \\
		0 & \frac{d}{c} & 0 & 1
	\end{smallmatrix}\right)
	=	\left(\begin{smallmatrix}
		1 & -d & 0 & 0 \\
		0 & 1 & 0 & 0 \\
		0 & 0 & 1 & a \\
		0 & 0 & 0 & 1
	\end{smallmatrix}\right) 
	\left(\begin{smallmatrix}
		1 & 0 & 0 & 0 \\
		0 & 1 & 0 & 0 \\
		\frac{a}{c} & b - \frac{a d}{c} & 1 & 0 \\
		0 & \frac{d}{c} & 0 & 1
	\end{smallmatrix}\right). 
$$

$$
	\left(\begin{smallmatrix}
		1 & 0 & 0 & 0 \\
		0 & \frac{1}{d} & 0 & 0 \\
		0 & 0 & 1 & 0 \\
		0 & 0 & 0 & d
	\end{smallmatrix}\right)^{-1} 
	\left(\begin{smallmatrix}
		1 & 0 & 0 & 0 \\
		0 & 1 & 0 & \frac{1}{d} \\
		0 & 0 & 1 & 0 \\
		0 & 0 & 0 & 1
	\end{smallmatrix}\right)^{-1}
	\left(\begin{smallmatrix}
		1 & 0 & 0 & 0 \\
		0 & 1 & 0 & 0 \\
		a & b & 1 & 0 \\
		c & 0 & 0 & 1
	\end{smallmatrix}\right) 
	\left(\begin{smallmatrix}
		1 & 0 & 0 & 0 \\
		0 & 1 & 0 & \frac{1}{d} \\
		0 & 0 & 1 & 0 \\
		0 & 0 & 0 & 1
	\end{smallmatrix}\right)
	\left(\begin{smallmatrix}
		1 & 0 & 0 & 0 \\
		0 & \frac{1}{d} & 0 & 0 \\
		0 & 0 & 1 & 0 \\
		0 & 0 & 0 & d
	\end{smallmatrix}\right) = \left(\begin{smallmatrix}
		1 & 0 & 0 & 0 \\
		-c & 1 & 0 & 0 \\
		a & \frac{b}{d} & 1 & b \\
		\frac{c}{d} & 0 & 0 & 1
	\end{smallmatrix}\right).
$$

$$
	\left(\begin{smallmatrix}
		1 & 0 & 0 & 0 \\
		-c & 1 & 0 & 0 \\
		a & \frac{b}{d} & 1 & b \\
		\frac{c}{d} & 0 & 0 & 1
	\end{smallmatrix}\right) = \left(\begin{smallmatrix}
		1 & 0 & 0 & 0 \\
		-c & 1 & 0 & 0 \\
		0 & 0 & 1 & b \\
		0 & 0 & 0 & 1
	\end{smallmatrix}\right)\left(\begin{smallmatrix}
		1 & 0 & 0 & 0 \\
		0 & 1 & 0 & 0 \\
		a - \frac{b c}{d} & \frac{b}{d} & 1 & 0 \\
		\frac{c}{d} & 0 & 0 & 1
	\end{smallmatrix}\right) 
$$

$$
	\left(\begin{smallmatrix}
		\frac{1}{a} & 0 & 0 & 0 \\
		0 & 1 & 0 & 0 \\
		0 & 0 & a & 0 \\
		0 & 0 & 0 & 1
	\end{smallmatrix}\right)^{-1}
	\left(\begin{smallmatrix}
		1 & 0 & \frac{1}{a} & 0 \\
		0 & 1 & 0 & 0 \\
		0 & 0 & 1 & 0 \\
		0 & 0 & 0 & 1
	\end{smallmatrix}\right)^{-1}
	\left(\begin{smallmatrix}
		1 & 0 & 0 & 0 \\
		0 & 1 & 0 & 0 \\
		0 & b & 1 & 0 \\
		c & d & 0 & 1
	\end{smallmatrix}\right) \\
	\left(\begin{smallmatrix}
		1 & 0 & \frac{1}{a} & 0 \\
		0 & 1 & 0 & 0 \\
		0 & 0 & 1 & 0 \\
		0 & 0 & 0 & 1
	\end{smallmatrix}\right)
	\left(\begin{smallmatrix}
		\frac{1}{a} & 0 & 0 & 0 \\
		0 & 1 & 0 & 0 \\
		0 & 0 & a & 0 \\
		0 & 0 & 0 & 1
	\end{smallmatrix}\right) = \left(\begin{smallmatrix}
		1 & -b & 0 & 0 \\
		0 & 1 & 0 & 0 \\
		0 & \frac{b}{a} & 1 & 0 \\
		\frac{c}{a} & d & c & 1
	\end{smallmatrix}\right) 
$$

$$
	\left(\begin{smallmatrix}
		1 & -b & 0 & 0 \\
		0 & 1 & 0 & 0 \\
		0 & \frac{b}{a} & 1 & 0 \\
		\frac{c}{a} & d & c & 1
	\end{smallmatrix}\right) =
	\left(\begin{smallmatrix}
		1 & -b & 0 & 0 \\
		0 & 1 & 0 & 0 \\
		0 & 0 & 1 & 0 \\
		0 & 0 & c & 1
	\end{smallmatrix}\right)\left(\begin{smallmatrix}
		1 & 0 & 0 & 0 \\
		0 & 1 & 0 & 0 \\
		0 & \frac{b}{a} & 1 & 0 \\
		\frac{c}{a} & -\frac{b c}{a} + d & 0 & 1
	\end{smallmatrix}\right). 
$$

\subsubsection{Some integrals within the unipotent radical of the Borel subgroup of $\GL_4$.}
\label{sec:unipotent-integral-identities}
Define the following matrices in $U_4(\mathbb{Q}_p)$, where $U_4$ is the unipotent radical of the upper triangular $B_4\subset G_4 = \GL_4/\mathbb{Q}.$
$$
	u(x_1) = 	\left(\begin{smallmatrix}
		1 & 0 & 1/{x_1} & 0 \\
		0 & 1 & 0 & 0 \\
		0 & 0 & 1 & 0 \\
		0 & 0 & 0 & 1
	\end{smallmatrix}\right), \ 
	u(x_2)= \left(\begin{smallmatrix}
		1 & 0 & 0 & 0 \\
		0 & 1 & {1}/{x_{2}} & 0 \\
		0 & 0 & 1 & 0 \\
		0 & 0 & 0 & 1
	\end{smallmatrix}\right),
	\quad u(x_3) =	\left(\begin{smallmatrix}
	1 & 0 & 0 & 1/x_3 \\
	0 & 1 & 0 & 0 \\
	0 & 0 & 1 & 0 \\
	0 & 0 & 0 & 1
	\end{smallmatrix}\right),\ 
	 u(x_4) = \left(\begin{smallmatrix} 1 & 0& 0& 0 \\ 0 & 1 &0 & 1/{x_4}\\ 0 & 0 & 1 &0 \\ 0 & 0 & 0 & 1 \end{smallmatrix}\right),
$$
where $x_1, x_2, x_3, x_4 \in \mathbb{Q}_p^\times$.

\begin{lemma}
\label{lemma: main-lemma}
	If $f_p \in C^\infty(G_4(\mathbb{Q}_p))$ with $f_p(ug) = f(g)$ for all $u \in U_4(\mathbb{Q}_p)$ we have the following three identities: 
		\begin{multline}
			\int f_p\left(t(x_3)^{-1} u(x_3)^{-1} \left(\begin{smallmatrix}
				1 & 0 & 0 & 0 \\
				0 & 1 & 0 & 0 \\
				x_1 & x_2  & 1 & 0 \\
				0 & x_4 & 0 & 1
			\end{smallmatrix}\right) u(x_3) t(x_3) \right) dx_1 dx_2 dx_4\\
			= \delta_P(t(x_3))^{1/2} \int f_p\left( \left(\begin{smallmatrix}
				1 & 0 & 0 & 0 \\
				0 & 1 & 0 & 0 \\
				x_1 & x_2  & 1 & 0 \\
				0 & x_4 & 0 & 1
			\end{smallmatrix}\right)\right) dx_1 dx_2 dx_4, 
			\label{eqn: eliminating-c}
		\end{multline}
		
		\begin{multline}
			\int f_p \left( t(x_4)^{-1}u(x_4)^{-1} \left(\begin{smallmatrix} 1 & & & \\ 0 & 1 & & \\ x_1 & x_2 & 1 & \\ 0 & 0 & 0 & 1 \end{smallmatrix}\right) 
			u(x_4)t(x_4)\right) dx_1 dx_2 \\
			= \delta_P(t(x_4))^{1/2} \int f_p\left( \left(\begin{smallmatrix} 1 & & & \\ 0 & 1 & & \\ x_1 & x_2 & 1 & \\ 0 & 0 & 0 & 1 \end{smallmatrix}\right)\right) dx_1 dx_2,
			\label{eqn: eliminating-d}
		\end{multline}
		
		\begin{multline}
			\int f_p \left( t(x_1)^{-1}u(x_1)^{-1} \left(\begin{smallmatrix} 1 & & & \\ 0 & 1 & & \\ 0 & x_2 & 1 & \\ 0 & 0 & 0 & 1\end{smallmatrix}\right)
			u(x_1) t(x_1) \right) dx_2 \\
			= \delta_P(t(x_4))^{1/2} \int f_p\left( \left(\begin{smallmatrix} 1 & & & \\ 0 & 1 & & \\ 0 & x_2 & 1 & \\ 0 & 0 & 0 & 1\end{smallmatrix}\right) \right) dx_2.
			\label{eqn: eliminating-a}
		\end{multline}
\end{lemma}

\begin{proof}
To begin, the integral in \eqref{eqn: eliminating-c} simplifies as
$$
		\int f_p\left(\left(\begin{smallmatrix}
			1 & -x_4 & 0 & 0 \\
			0 & 1 & 0 & 0 \\
			\frac{x_1}{x_3} & x_2 & 1 & x_1 \\
			0 & \frac{x_4}{x_3} & 0 & 1
		\end{smallmatrix}\right) \right) dx_1 dx_2 dx_4 
\ = \ 
\int f_p\left( \left(\begin{smallmatrix}
			1 & -x_4 & 0 & 0 \\
			0 & 1 & 0 & 0 \\
			0 & 0 & 1 & x_1 \\
			0 & 0 & 0 & 1
		\end{smallmatrix}\right) 
		\left(\begin{smallmatrix}
			1 & 0 & 0 & 0 \\
			0 & 1 & 0 & 0 \\
			\frac{x_1}{x_3} & x_2 - \frac{x_1 x_4}{x_3} & 1 & 0 \\
			0 & \frac{x_4}{x_3} & 0 & 1
		\end{smallmatrix}\right)\right)dx_1 dx_2 dx_4. 
$$		
Due to the invariance of $f_p$ under $U_4(\mathbb{Q}_p)$, the integral on the right is the same as
$$
\int f_p \left(\left(\begin{smallmatrix}
			1 & 0 & 0 & 0 \\
			0 & 1 & 0 & 0 \\
			\frac{x_1}{x_3} & x_2 - \frac{x_1 x_4}{x_3} & 1 & 0 \\
			0 & \frac{x_4}{x_3} & 0 & 1
		\end{smallmatrix}\right) \right) dx_1 dx_2 dx_4 
\ = \ 
\int f_p\left( \left(\begin{smallmatrix}
			1 & 0 & 0 & 0 \\
			0 & 1 & 0 & 0 \\
			\frac{x_1}{x_3} & x_2  & 1 & 0 \\
			0 & \frac{x_4}{x_3} & 0 & 1
		\end{smallmatrix}\right) \right) dx_1 dx_2 dx_4 
$$
which evaluates to 
\begin{multline*}
\int f_p\left( \left(\begin{smallmatrix}
			\frac{1}{x_3} & 0 & 0 & 0 \\
			0 & 1 & 0 & 0 \\
			0 & 0 & 1 & 0 \\
			0 & 0 & 0 & x_3
		\end{smallmatrix}\right)^{-1}
		\left(\begin{smallmatrix}
			1 & 0 & 0 & 0 \\
			0 & 1 & 0 & 0 \\
			x_1 & x_2  & 1 & 0 \\
			0 & x_4 & 0 & 1
		\end{smallmatrix}\right)
		\left(\begin{smallmatrix}
			\frac{1}{x_3} & 0 & 0 & 0 \\
			0 & 1 & 0 & 0 \\
			0 & 0 & 1 & 0 \\
			0 & 0 & 0 & x_3
		\end{smallmatrix}\right)\right) dx_1 dx_2 dx_4 \\
\ = \ \delta_P\left(t(x_3)\right)^{1/2} \int f_p\left( \left(\begin{smallmatrix}
			1 & 0 & 0 & 0 \\
			0 & 1 & 0 & 0 \\
			x_1 & x_2  & 1 & 0 \\
			0 & x_4 & 0 & 1
		\end{smallmatrix}\right)\right) dx_1 dx_2 dx_4, 
\end{multline*}
the last equality is due to \eqref{eqn: conjuagation-by-levi}. This completes the verification of \eqref{eqn: eliminating-c}. 
The other integrals are similar. 
\end{proof}

\subsection{The evaluation}
The purpose of this section is to evaluate the constant
\begin{multline}
	I:=	\int \phi'_p\left(w_0^{-1} \left(\begin{smallmatrix}
		1 & 0 & x_{1} & x_{2} \\
		0 & 1 & x_{3} & x_{4} \\
		0 & 0 & 1 & 0 \\
		0 & 0 & 0 & 1
	\end{smallmatrix}\right) \xi_p^{(n_p)} \right)(\underline{\mathbf{1}}) dx_3 dx_1 dx_4 dx_2\\
	= \int \phi'_p\left(w_0^{-1} \left(\begin{smallmatrix}
		1 & 0 & x_{1} & x_{2} \\
		0 & 1 & x_{3} & x_{4} \\
		0 & 0 & 1 & 0 \\
		0 & 0 & 0 & 1
	\end{smallmatrix}\right) \right)(\underline{\mathbf{1}}) dx_3 dx_1 dx_4 dx_2.
\end{multline}
This is due to the choice we made $\xi_p^{(n_p)} = \mathbf{1}_4.$ Writing out the evaluation at $\underline{\mathbf{1}}$ makes the notation cumbersome. So we shall drop them and assume it implicitly.

\subsubsection{Eliminating the variable $x_3$}
Split the innermost integral as 
$$
\int_{x_3 \in \mathbb{Q}_p} = \int_{x_3 \in \mathbb{Q}_p:\,\, v_p(x_3) \geq  0} + \int_{x_3 \in \mathbb{Q}_p:\,\, v_p(x_3) < 0}.
$$
If $v_p(x_3) \geq 0$ then
$\left(\begin{smallmatrix}
		1 & 0 & 0 & 0 \\
		0 & 1 & x_{3} & 0 \\
		0 & 0 & 1 & 0 \\
		0 & 0 & 0 & 1
	\end{smallmatrix}\right) \in \mirp{n_p + n'_p},$ hence the first integral is 
\begin{multline}
	\int
	\limits_{x_3 \in \mathbb{Q}_p:\,\, v_p(x_3) \geq 0} \phi'_p\left(w_0^{-1} \left(\begin{smallmatrix}
		1 & 0 & x_{1} & x_{2} \\
		0 & 1 & x_{3} & x_{4} \\
		0 & 0 & 1 & 0 \\
		0 & 0 & 0 & 1
	\end{smallmatrix}\right) \right) dx_3 dx_1 dx_4 dx_2  \\ = \int
	\limits_{x_3 \in \mathbb{Q}_p:\,\, v_p(x_3) \geq 0} \phi'_p\left(w_0^{-1} \left(\begin{smallmatrix}
		1 & 0 & x_{1} & x_{2} \\
		0 & 1 & 0 & x_{4} \\
		0 & 0 & 1 & 0 \\
		0 & 0 & 0 & 1
	\end{smallmatrix}\right) \right) dx_3 dx_1 dx_4 dx_2 \\ = T_{\geq 0}(x_3)\biggl[\int \phi'_p\left(w_0^{-1} \left(\begin{smallmatrix}
	1 & 0 & x_{1} & x_{2} \\
	0 & 1 & 0 & x_{4} \\
	0 & 0 & 1 & 0 \\
	0 & 0 & 0 & 1
	\end{smallmatrix}\right)  \right) dx_1 dx_4 dx_2 \biggr].  \label{eqn: x3-geq-0}
\end{multline}
If $x_3 \in \mathbb{Q}_p^\times $ and $v_p(x_3) < 0$, then second integral is
$$
\int \phi'_p\left(w_0^{-1} \left(\begin{smallmatrix}
		1 & 0 & x_{1} & x_{2} \\
		0 & 1 & x_{3} & x_{4} \\
		0 & 0 & 1 & 0 \\
		0 & 0 & 0 & 1
	\end{smallmatrix}\right)  \right)dx_3 dx_1 dx_4 dx_2 
	= \int \phi'_p \left( \left(\begin{smallmatrix}
		1 & 0 & 0 & 0 \\
		0 & 1 & 0 & 0 \\
		x_{1} & x_{2} & 1 & 0 \\
		0 & x_{4} & 0 & 1
	\end{smallmatrix}\right) \left(\begin{smallmatrix}
		0 & 0 & 1 & 0 \\
		0 & 0 & 0 & 1 \\
		1 & 0 & 0 & 0 \\
		0 & 1 & x_3 & 0
	\end{smallmatrix}\right) \right) dx_3 dx_1 dx_4 dx_2.
$$
which can be written as
$$	
\int \phi'_p \biggl( \left(\begin{smallmatrix}
		1 & 0 & 0 & 0 \\
		0 & 1 & 0 & 0 \\
		x_{1} & x_{2} & 1 & 0 \\
		0 & x_{4} & 0 & 1
	\end{smallmatrix}\right) 
	\left(\begin{smallmatrix}
		1 & 0 & 0 & \frac{1}{x_{3}} \\
		0 & 1 & 0 & 0 \\
		0 & 0 & 1 & 0 \\
		0 & 0 & 0 & 1
	\end{smallmatrix}\right) 
	\left(\begin{smallmatrix}
		\frac{1}{x_{3}} & 0 & 0 & 0 \\
		0 & 1 & 0 & 0 \\
		0 & 0 & 1 & 0 \\
		0 & 0 & 0 & x_{3}
	\end{smallmatrix}\right)
	\left(\begin{smallmatrix}
		1 & 0 & 0 & 0 \\
		0 & 0 & 0 & 1 \\
		0 & 0 & 1 & 0 \\
		0 & 1 & 0 & 0
	\end{smallmatrix}\right)
	\left(\begin{smallmatrix}
		0 & -1 & 0 & 0 \\
		0 & \frac{1}{x_{3}} & 1 & 0 \\
		1 & 0 & 0 & 0 \\
		0 & 0 & 0 & 1
	\end{smallmatrix}\right)
	\biggr) dx_3 dx_4 dx_1 dx_2
$$
When $x_3 \in \mathbb{Q}_p^\times $ and $v_p(x_3) < 0 $, 
$
	\left(\begin{smallmatrix}
		0 & -1 & 0 & 0 \\
		0 & \frac{1}{x_{3}} & 1 & 0 \\
		1 & 0 & 0 & 0 \\
		0 & 0 & 0 & 1
	\end{smallmatrix}\right) \in \mirp{n_p + n_p'}.
$
Hence it simplifies as
$$
	\int\limits_{x_3 \in \mathbb{Q}_p^\times,\,\, v_p(x_3)< 0} \phi'_p \left( \left(\begin{smallmatrix}
		1 & 0 & 0 & 0 \\
		0 & 1 & 0 & 0 \\
		x_{1} & x_{2} & 1 & 0 \\
		0 & x_{4} & 0 & 1
	\end{smallmatrix}\right) u(x_3) t(x_3) \left(\begin{smallmatrix}
		1 & 0 & 0 & 0 \\
		0 & 0 & 0 & 1 \\
		0 & 0 & 1 & 0 \\
		0 & 1 & 0 & 0
	\end{smallmatrix}\right)\right) dx_3 dx_1 dx_4 dx_2,
$$
which can be written as
$$
\int\limits_{x_3 \in \mathbb{Q}_p^\times,\,\, v_p(x_3)< 0}  \underline{\sigma}_p'(t(x_3)) \cdot \phi'_p \biggl( t(x_3)^{-1} u(x_3^{-1}) \left(\begin{smallmatrix}
		1 & 0 & 0 & 0 \\
		0 & 1 & 0 & 0 \\
		x_{1} & x_{2} & 1 & 0 \\
		0 & x_{4} & 0 & 1
	\end{smallmatrix}\right) \\ u(x_3) t(x_3) \left(\begin{smallmatrix}
		1 & 0 & 0 & 0 \\
		0 & 0 & 0 & 1 \\
		0 & 0 & 1 & 0 \\
		0 & 1 & 0 & 0
	\end{smallmatrix}\right) \biggr)  dx_3 dx_4 dx_1 dx_2,
$$
which in turn becomes
$$	
 \int\limits_{x_3 \in \mathbb{Q}_p^\times,\,\, v_p(x_3) < 0}  \delta_P(t(x_3))^{1/2}\underline{\sigma}_p'(t(x_3)) \cdot \phi'_p \biggl( \left(\begin{smallmatrix}
		1 & 0 & 0 & 0 \\
		0 & 1 & 0 & 0 \\
		x_{1} & x_{2} & 1 & 0 \\
		0 & x_{4} & 0 & 1
	\end{smallmatrix}\right) 
	\left(\begin{smallmatrix}
		1 & 0 & 0 & 0 \\
		0 & 0 & 0 & 1 \\
		0 & 0 & 1 & 0 \\
		0 & 1 & 0 & 0
	\end{smallmatrix}\right) \biggr)  dx_3 dx_4 dx_1 dx_2.
$$
Using the notation in \eqref{eqn: less-than-M-operator} we can write the above integral as
\begin{multline*}
T_{<0}(t(x_3)) \biggl[\int\phi'_p\biggl( \left(\begin{smallmatrix}
		1 & 0 & 0 & 0 \\
		0 & 1 & 0 & 0 \\
		x_{1} & x_{2} & 1 & 0 \\
		0 & x_{4} & 0 & 1
	\end{smallmatrix}\right)  \left(\begin{smallmatrix}
		1 & 0 & 0 & 0 \\
		0 & 0 & 0 & 1 \\
		0 & 0 & 1 & 0 \\
		0 & 1 & 0 & 0
	\end{smallmatrix}\right)  \biggr) dx_1 dx_4 dx_2 \biggr] \\
\ = \ 
T_{<0}(t(x_3)) \biggl[ \int \phi'_p \biggl( w_0^{-1}\left(\begin{smallmatrix}
		1 & 0 & x_{1} & x_{2} \\
		0 & 1 & 0 & x_{4} \\
		0 & 0 & 1 & 0 \\
		0 & 0 & 0 & 1
	\end{smallmatrix}\right)  w_0\left(\begin{smallmatrix}
		1 & 0 & 0 & 0 \\
		0 & 0 & 0 & 1 \\
		0 & 0 & 1 & 0 \\
		0 & 1 & 0 & 0
	\end{smallmatrix}\right)  \biggr) dx_1 dx_4 dx_2 \biggr], \\
\ = \
T_{<0}(t(x_3))\biggl[ \int \phi'_p \biggl( w_0^{-1}\left(\begin{smallmatrix}
		1 & 0 & x_{1} & x_{2} \\
		0 & 1 & 0 & x_{4} \\
		0 & 0 & 1 & 0 \\
		0 & 0 & 0 & 1
	\end{smallmatrix}\right) \left(\begin{smallmatrix}
		0 & 0 & 1 & 0 \\
		0 & 1 & 0 & 0 \\
		1 & 0 & 0 & 0 \\
		0 & 0 & 0 & 1
	\end{smallmatrix}\right)  \biggr) dx_1 dx_4 dx_2 \biggr].
\end{multline*}
The last integral equals
\begin{equation}
	T_{<0}(t(x_3))\biggl[ \int \phi'_p \biggl( w_0^{-1}\left(\begin{smallmatrix}
		1 & 0 & x_{1} & x_{2} \\
		0 & 1 & 0 & x_{4} \\
		0 & 0 & 1 & 0 \\
		0 & 0 & 0 & 1
	\end{smallmatrix}\right)   \biggr) dx_1 dx_4 dx_2 \biggr] \label{eqn: x3-less-than-0}.
\end{equation}
Combine \eqref{eqn: x3-geq-0} and \eqref{eqn: x3-less-than-0} to get
\begin{equation}
	I = \left(T_{\geq 0}(x_3) + T_{< 0 }(x_3) \right)\biggl[ \int \phi'_p \biggl( w_0^{-1}\left(\begin{smallmatrix}
		1 & 0 & x_{1} & x_{2} \\
		0 & 1 & 0 & x_{4} \\
		0 & 0 & 1 & 0 \\
		0 & 0 & 0 & 1
	\end{smallmatrix}\right)  \biggr) dx_1 dx_4 dx_2\biggr]. \label{eqn: variable-x3-eliminated}
\end{equation}

\medskip
\subsubsection{Eliminating the variable $x_1$}
Carrying out a very similar computation with the the inner integral in \eqref{eqn: variable-x3-eliminated}, by splitting it up as 
$\int\limits_{x_1 \in \mathbb{Q}_p} = \int\limits_{x_1 \in \mathbb{Q}_p:\,\, v_p(x_1) \geq 0} + \int\limits_{x_1 \in \mathbb{Q}_p^:\,\, v_p(x_1) < 0 },$ 
 the integral $I$ equals
\begin{equation}
	I = (T_{\geq 0}(x_3) + T_{<0}(x_3))(T_{\geq 0}(x_1) + T_{<0}(x_1)) \biggl[ \int \phi'_p \biggl( w_0^{-1}\left(\begin{smallmatrix}
		1 & 0 & 0 & x_{2} \\
		0 & 1 & 0 & x_{4} \\
		0 & 0 & 1 & 0 \\
		0 & 0 & 0 & 1
	\end{smallmatrix}\right)   \biggr) dx_4 dx_2 \biggr]. \label{eqn: variable-x1-eliminated}
\end{equation}

\medskip
\subsubsection{Eliminating the variable $x_4$}
Once again, carrying out a very similar computation with the the inner integral in \eqref{eqn: variable-x1-eliminated}, we now get
\begin{multline}
	I = (T_{\geq 0}(x_3) + T_{<0}(x_3))(T_{\geq 0}(x_1) + T_{<0}(x_1)) \bigl[ \\ T_{\geq 0}(x_4) \biggl[ \int \phi'_p \biggl( w_0^{-1}\left(\begin{smallmatrix}
		1 & 0 & 0 & x_{2} \\
		0 & 1 & 0 & 0 \\
		0 & 0 & 1 & 0 \\
		0 & 0 & 0 & 1
	\end{smallmatrix}\right)   \biggr) dx_2 \biggr] + T_{< 0}(x_4) \biggl[ \int \phi'_p \biggl( w_0^{-1}\left(\begin{smallmatrix}
		1 & 0 & 0 & x_{2} \\
		0 & 1 & 0 & 0 \\
		0 & 0 & 1 & 0 \\
		0 & 0 & 0 & 1
	\end{smallmatrix}\right)  w_0  \biggr) dx_2 \biggr] \bigr]. \label{eqn: variable-x4-eliminated}
\end{multline}

\medskip
\subsubsection{Eliminating the variable $x_2$}
There are two integrals in \eqref{eqn: variable-x4-eliminated} to be evaluated now. 
The second integral in \eqref{eqn: variable-x4-eliminated} will turn out to be $0,$ i.e.,
\begin{equation}
	\int \phi_p \biggl( w_0^{-1}\left(\begin{smallmatrix}
		1 & 0 & 0 & x_2 \\
		0 & 1 & 0 & 0 \\
		0 & 0 & 1 & 0 \\
		0 & 0 & 0 & 1
	\end{smallmatrix}\right)w_0 \biggr)dx_2 = 0. \label{eqn: integral-with-x2-w0-is-zero}
\end{equation}
To see this, split the integral $\int\limits_{x_2 \in \mathbb{Q}_p} = \int\limits_{x_2 \in \mathbb{Q}_p:\,\, v_p(x_2) \geq 0} + \int\limits_{x_2 \in \mathbb{Q}_p:\,\, v_p(x_2) < 0},$ 
and observe that 
$$
	\int\limits_{x_2 \in \mathbb{Q}_p:\,\, v_p(x_2) \geq 0} \phi'_p\left(w_0^{-1} \left(\begin{smallmatrix}
		1 & 0 & 0 & x_{2} \\
		0 & 1 & 0 & 0 \\
		0 & 0 & 1 & 0 \\
		0 & 0 & 0 & 1
	\end{smallmatrix}\right) w_0  \right)  dx_2 = \int \phi'_p\left( \left(\begin{smallmatrix}
	1 & 0 & 0 & 0 \\
	0 & 1 & 0 & 0 \\
	0 & x_{2} & 1 & 0 \\
	0 & 0 & 0 & 1
	\end{smallmatrix}\right)  \right)  dx_2 = 0 
$$
because $\left(\begin{smallmatrix}
	1 & 0 & 0 & 0 \\
	0 & 1 & 0 & 0 \\
	0 & x_{2} & 1 & 0 \\
	0 & 0 & 0 & 1
\end{smallmatrix}\right) \in \mirp{n_p + n'_p} \subset P(\mathbb{Q}_p) \mathbf{1}_4 \mirp{n_p + n'_p}$ and 
$\phi'_p$ is \textit{not} supported in the coset $P(\mathbb{Q}_p) \mathbf{1}_4 \mirp{n_p + n'_p}.$ Similarly, 
\begin{equation*}
	\int\limits_{v(x_2) < 0} \phi'_p \left( \left(\begin{smallmatrix}
		0 & 0 & 1 & 0 \\
		0 & 0 & 0 & 1 \\
		1 & 0 & 0 & x_{2} \\
		0 & 1 & 0 & 0
	\end{smallmatrix}\right) w_0 \right) dx_2
	= \int \phi'_p \left( u(x_2) t(x_2)   \left(\begin{smallmatrix}
		1 & 0 & 0 & 0 \\
		0 & 0 & -1 & 0 \\
		0 & 1 & \frac{1}{x_{2}} & 0 \\
		0 & 0 & 0 & 1
	\end{smallmatrix}\right)\right) dx_2 
	= 0,
\end{equation*}
since 
$\left(\begin{smallmatrix}
		1 & 0 & 0 & 0 \\
		0 & 0 & -1 & 0 \\
		0 & 1 & \frac{1}{x_{2}} & 0 \\
		0 & 0 & 0 & 1
	\end{smallmatrix}\right) \in \mirp{n_p + n'_p}.$ 
This proves \eqref{eqn: integral-with-x2-w0-is-zero}. Next, we evaluate the first integral in \eqref{eqn: variable-x4-eliminated} 
\[\int \phi'_p \biggl( w_0^{-1}\left(\begin{smallmatrix}
	1 & 0 & 0 & x_{2} \\
	0 & 1 & 0 & 0 \\
	0 & 0 & 1 & 0 \\
	0 & 0 & 0 & 1
\end{smallmatrix}\right)   \biggr) dx_2,\]
By once again splitting the integral as 
$\int\limits_{x_2 \in \mathbb{Q}_p} = \int\limits_{x_2 \in \mathbb{Q}_p:\,\, v_p(x_2) \geq 0} + \int\limits_{x_2 \in \mathbb{Q}_p:\,\, v_p(x_2) < 0};$ the first of which equals
\[\int\limits_{x_2 \in \mathbb{Q}_p: \,\, v_p(x_2) \geq 0} \phi'_p \biggl( w_0^{-1}\left(\begin{smallmatrix}
	1 & 0 & 0 & x_{2} \\
	0 & 1 & 0 & 0 \\
	0 & 0 & 1 & 0 \\
	0 & 0 & 0 & 1
\end{smallmatrix}\right)   \biggr) dx_2 = \int\limits_{x_2 \in \mathbb{Q}_p: \,\, v_p(x_2) \geq 0} \phi'_p(w_0^{-1}) dx_2 = (v_p\otimes v'_p)(\underline{\mathbf{1}});\]
(the last equality is due to the facts, that $w_0^{-1}= Q\xi_{p}^{(n_p)}K_1K_2$ for some $Q\in P(\mathbb{Q}_p)$ and 
$K_1, K_2 \in \mirp{n_p + n'_p}$ and that $\phi'_p$ is supported in $ P(\mathbb{Q}_p) \xi_p^{(n_p)} \mirp{n_p + n'_p}$ 
and on the double coset it takes the value $v_p \otimes v'_p.$ Recall that in the beginning the evaluation at $\underline{\mathbf{1}}$ was assumed implicitly.) The latter integral over $v_p(x_2) < 0$ is
\begin{multline*}
\int\limits_{x_2 \in \mathbb{Q}_p: \,\, v_p(x_2) < 0} \phi'_p \biggl( w_0^{-1}\left(\begin{smallmatrix}
		1 & 0 & 0 & x_{2} \\
		0 & 1 & 0 & 0 \\
		0 & 0 & 1 & 0 \\
		0 & 0 & 0 & 1
	\end{smallmatrix}\right)    \biggr) dx_2 \\ 
= \ 
\int \phi'_p\left(\left(\begin{smallmatrix}
		1 & 0 & 0 & 0 \\
		0 & 1 & \frac{1}{x_{2}} & 0 \\
		0 & 0 & 1 & 0 \\
		0 & 0 & 0 & 1
	\end{smallmatrix}\right) \left(\begin{smallmatrix}
		1 & 0 & 0 & 0 \\
		0 & \frac{1}{x_{2}} & 0 & 0 \\
		0 & 0 & x_{2} & 0 \\
		0 & 0 & 0 & 1
	\end{smallmatrix}\right) \left(\begin{smallmatrix}
		0 & 1 & 0 & 0 \\
		1 & 0 & 0 & 0 \\
		0 & 0 & 0 & 1 \\
		0 & 0 & 1 & 0
	\end{smallmatrix}\right) \left(\begin{smallmatrix}
	-1 & 0 & 0 & 0 \\
	0 & 0 & 1 & 0 \\
	0 & 1 & 0 & 0 \\
	\frac{1}{x_{2}} & 0 & 0 & 1
	\end{smallmatrix}\right) \right) dx_2 
\ = \ 0.
\end{multline*}
Again due to the fact $\phi'_p$ is not supported in $P(\mathbb{Q}_p) \mathbf{1}_4 \mirp{n_p + n'_p}.$
Therefore,
\[\int \phi'_p \biggl( w_0^{-1}\left(\begin{smallmatrix}
	1 & 0 & 0 & x_{2} \\
	0 & 1 & 0 & 0 \\
	0 & 0 & 1 & 0 \\
	0 & 0 & 0 & 1
\end{smallmatrix}\right)    \biggr) dx_2 = T_{\geq 0} (x_2)\phi'_p(\xi_p^{(n_p)}).\]

\medskip
\subsubsection{Final evaluation}
Using the above calculations, \eqref{eqn: variable-x4-eliminated} reduces to
\begin{multline*}
I = (T_{\geq 0}(x_3) + T_{<0}(x_3))(T_{\geq 0}(x_1) + T_{<0}(x_1)) \\ 
\left(T_{\geq 0}(x_4) \biggl[ \int \phi_p \biggl( w_0^{-1}
	\left(\begin{smallmatrix}
		1 & 0 & 0 & x_{2} \\
		0 & 1 & 0 & 0 \\
		0 & 0 & 1 & 0 \\
		0 & 0 & 0 & 1
	\end{smallmatrix}\right)   \biggr) dx_2 \biggr] + T_{< 0}(x_4) \biggl[ \int \phi'_p \biggl( w_0^{-1}
	\left(\begin{smallmatrix}
		1 & 0 & 0 & x_{2} \\
		0 & 1 & 0 & 0 \\
		0 & 0 & 1 & 0 \\
		0 & 0 & 0 & 1
	\end{smallmatrix}\right)  w_0  \biggr) dx_2 \biggr] \right)
\end{multline*}
which reduces to 
$$
(T_{\geq 0}(x_3) + T_{<0}(x_3))(T_{\geq 0}(x_1) + T_{<0}(x_1))
\left(T_{\geq 0}(x_4) \biggl[ \int \phi'_p \biggl( w_0^{-1}
	\left(\begin{smallmatrix}
		1 & 0 & 0 & x_{2} \\
		0 & 1 & 0 & 0 \\
		0 & 0 & 1 & 0 \\
		0 & 0 & 0 & 1
	\end{smallmatrix}\right)   \biggr) dx_2 + 0 \biggr] \right)
$$
which in turn simplifies as
$$
(T_{\geq 0}(x_3) + T_{<0}(x_3))(T_{\geq 0}(x_1) + T_{<0}(x_1))
\left(T_{\geq 0}(x_4) \biggl[ T_{\geq 0}(x_2) \phi'_p(\xi_p^{(n_p)})(\underline{\mathbf{1}})\biggr] \right).
$$
Using Lem.\,\ref{lemma: ratios-of-l-values} for $x_2, x_1$ and $x_3$, we get:
$$
I 
\ = \ 
	\left(\dfrac{1-p \,\,\chi_p(p) \chi'_{1,p}(p^{-1})}{1-p^{2} \,\,\chi_p(p) \,\,\chi'_{1,p}(p^{-1})}\right)
	\left( \dfrac{1-p \,\,\chi_p(p) \,\,\chi'_{2,p}(p^{-1})}{1-p^{2} \,\,\chi_p(p) \,\,\chi'_{2,p}(p^{-1})} \right) 
\ = \ 
\frac{L_p(k'-3, h\times h')}{L_p(k'-2, h \times h')}.
$$
The last equality is due to \eqref{eqn: relation-with-fourier-coefficient}. 

Similar computation yields the results for the other possible twists as well.

\end{document}